\documentclass[english]{article}
\usepackage[T1]{fontenc}
\usepackage[latin1]{inputenc}
\usepackage{amsthm}
\usepackage{amsmath}
\usepackage{amssymb}
\usepackage[all]{xy}

\theoremstyle{plain}
\newtheorem{thm}{Theorem}[section]
  \theoremstyle{remark}
  \newtheorem*{acknowledgement*}{Acknowledgement}
  \theoremstyle{plain}
  \newtheorem{lem}[thm]{Lemma}
  \theoremstyle{definition}
  \newtheorem{defn}[thm]{Definition}
  \theoremstyle{remark}
  \newtheorem*{rem*}{Remark}
  \theoremstyle{plain}
  \newtheorem{prop}[thm]{Proposition}
  \theoremstyle{plain}
  \newtheorem{cor}[thm]{Corollary}
  \theoremstyle{remark}
  \newtheorem*{note*}{Note}
  \theoremstyle{remark}
  \newtheorem{rem}[thm]{Remark}

\makeatother

\usepackage{babel}

\usepackage{babel}

\begin{document}

\title{Local-global compatibility and the action of monodromy on nearby
cycles}

\author{Ana Caraiani}
\maketitle
\begin{abstract}
We strengthen the local-global compatibility of Langlands correspondences
for $GL_{n}$ in the case when $n$ is even and $l\not=p$. Let $L$
be a CM field and $\Pi$ be a cuspidal automorphic representation
of $GL_{n}(\mathbb{A}_{L})$ which is conjugate self-dual. Assume
that $\Pi_{\infty}$ is cohomological and not {}``slightly regular'',
as defined by Shin. In this case, Chenevier and Harris constructed
an $l$-adic Galois representation $R_{l}(\Pi)$ and proved the local-global
compatibility up to semisimplification at primes $v$ not dividing
$l$. We extend this compatibility by showing that the Frobenius semisimplification
of the restriction of $R_{l}(\Pi)$ to the decomposition group at
$v$ corresponds to the image of $\Pi_{v}$ via the local Langlands
correspondence. We follow the strategy of Taylor-Yoshida, where it
was assumed that $\Pi$ is square-integrable at a finite place. To
make the argument work, we study the action of the monodromy operator
$N$ on the complex of nearby cycles on a scheme which is locally
etale over a product of semistable schemes and derive a generalization
of the weight-spectral sequence in this case. We also prove the Ramanujan-Petersson
conjecture for $\Pi$ as above. 
\end{abstract}

\section{Introduction}

In this paper, we aim to strengthen the local-global compatibility
of the Langlands correspondence.
\begin{thm}
\label{local-global compatibility}Let $n\in\mathbb{Z}_{\geq2}$ be
an integer and $L$ be any CM field with complex conjugation $c$.
Let $l$ be a prime of $\mathbb{Q}$ and $\iota_{l}$ be an isomorphism
$\iota_{l}:\bar{\mathbb{Q}}_{l}\to\mathbb{C}$. Let $\Pi$ be a cuspidal
automorphic representation of $GL_{n}(\mathbb{A}_{L})$ satisfying
\begin{itemize}
\item $\Pi^{\vee}\simeq\Pi\circ c$
\item $\Pi$ is cohomological for some irreducible algebraic representation
$\Xi$ of $GL_{n}(L\otimes_{\mathbb{Q}}\mathbb{C})$. 
\end{itemize}
Let \[
R_{l}(\Pi):\mbox{Gal}(\bar{L}/L)\to GL_{n}(\bar{\mathbb{Q}}_{l})\]
 be the Galois representation associated to $\Pi$ by \cite{Shin,CH}.
Let $p\not=l$ and let $y$ be a place of $L$ above $p$. Then we
have the following isomorphism of Weil-Deligne respresentations \[
WD(R_{l}(\Pi)|_{Gal(\bar{L}_{y}/L_{y})})^{F-ss}\simeq\iota_{l}^{-1}\mathcal{L}_{n,L_{y}}(\Pi_{y}).\]

\end{thm}
Here $\mathcal{L}_{n,L_{y}}(\Pi_{y})$ is the image of $\Pi_{y}$
under the local Langlands correspondence, where the geometric normalization
is used. 

In the process of proving Theorem \ref{local-global compatibility},
we also prove the Ramanujan-Petersson conjecture for $\Pi$ as above. 
\begin{thm}
\label{Ramanujan}Let $n\in\mathbb{Z}_{\geq2}$ be an integer and
$L$ be any CM field. Let $\Pi$ be a cuspidal automorphic representation
of $GL_{n}(\mathbb{A}_{L})$ satisfying 
\begin{itemize}
\item $\Pi^{\vee}\simeq\Pi\circ c$
\item $\Pi_{\infty}$ is cohomological for some irreducible algebraic representation
$\Xi$ of $GL_{n}(L\otimes_{\mathbb{Q}}\mathbb{C})$.
\end{itemize}
Then $\Pi$ is tempered at any finite place of $L$. 

\end{thm}
The above theorems are already known when $n$ is odd or when $n$
is even and $\Pi$ is slightly regular, by work of Shin \cite{Shin}.
They are also known if $\Pi$ is square integrable at a finite place,
by the work of Harris-Taylor \cite{H-T} and Taylor-Yoshida \cite{T-Y}.
If $n$ is even then Chenevier and Harris construct in \cite{CH}
a global $\mbox{Gal}(\bar{L}/L)$-representation $R_{l}(\Pi)$ which
is compatible with the local Langlands correspondence up to semisimplification.
We extend the local-global compatibility up to Frobenius semisimplification,
by proving that both Weil-Deligne representations are pure. We use
Theorem \ref{Ramanujan} to deduce that $\iota_{l}^{-1}\mathcal{L}_{n,\mathcal{L}_{y}}(\Pi_{y})$
is pure. For the representation $WD(R_{l}(\Pi)|_{Gal(\bar{L}_{y}/L_{y})})$,
our strategy is as follows: we find the Galois representation $R_{l}(\Pi)^{\otimes2}$
in the cohomology of a system of Shimura varieties $X_{U}$ associated
to a unitary group which looks like \[
U(1,n-1)\times U(1,n-1)\times U(0,n)^{d-2}\]
at infinity. Following the same structure of argument as Taylor-Yoshida
in \cite{T-Y}, we prove that the Weil-Deligne representation associated
to \[
R_{l}(\Pi)^{\otimes2}|_{Gal(\bar{L}_{y}/L)}\]
is pure by explicitly computing the action of the monodromy operator
$N$ on the cohomology of the system of Shimura varieties. We use
Theorem \ref{Ramanujan} at a crucial point in the computation. We
conclude that $WD(R_{l}(\Pi)|_{Gal(\bar{L}_{y}/L_{y})})^{F-ss}$ must
also be pure. 

We briefly outline our computation of the action of $N$ on the Weil-Deligne
representation associated to $R_{l}(\Pi)^{\otimes2}|_{\mbox{Gal}(\bar{L}_{y}/L_{y})}.$
First, we base change $\Pi$ to a CM field $F'$ such that there is
a place $\mathfrak{p}$ of $F'$ above the place $y$ of $L$ where
$BC_{F'/L}(\Pi)_{\mathfrak{p}}$ has an Iwahori fixed vector. It suffices
to study the Weil-Deligne representation corresponding to $\Pi^{0}=BC_{F'/L}(\Pi)$
and prove that it is pure. We then take a quadratic extension $F$
of $F'$ which is also a CM field and in which the place $\mathfrak{p}$
splits $\mathfrak{p}=\mathfrak{p}_{1}\mathfrak{p}_{2}$. We let $\sigma\in\mbox{Gal}(F/F')$
be the automorphism which sends $\mathfrak{p}_{1}$ to $\mathfrak{p}_{2}$.
We choose $F$ and $F'$ such that they contain an imaginary quadratic
field $E$ in which $p$ splits. We take a $\mathbb{Q}$-group $G$
which satisfies the following: 
\begin{itemize}
\item $G$ is quasi-split at all finite places.
\item $G(\mathbb{R})$ has signature $(1,n-1)$ at two embeddings which
differ by $\sigma$ and $(0,n)$ everywhere else. 
\item $G(\mathbb{A}_{E})\simeq GL_{1}(\mathbb{A}_{E})\times GL_{n}(\mathbb{A}_{F})$. 
\end{itemize}
We let $\Pi^{1}=BC_{F/F'}(\Pi^{0})$. Then the Galois representation
$R_{l}(\Pi^{0})$ can be seen in the $\Pi^{1,\infty}$-part of the
(base change of the) cohomology of a system of Shimura varieties associated
to $G$. We let $X_{U}$ be the inverse system of Shimura varieties
associated to the group $G$. We let the level $U$ vary outside $\mathfrak{p}_{1}\mathfrak{p}_{2}$
and be equal to the Iwahori subgroup at $\mathfrak{p}_{1}$ and $\mathfrak{p}_{2}$.
We construct an integral model of $X_{U}$ which parametrizes abelian
varieties with Iwahori level structure at $\mathfrak{p}_{1}$ and
$\mathfrak{p}_{2}$. By abuse of notation, we will denote this integral
model by $X_{U}$ as well. The special fiber $Y_{U}$ of $X_{U}$
has a stratification by $Y_{U,S,T}$ where the $S,T\subseteq\{1,\dots n\}$
are related to the Newton polygons of the $p$-divisible groups above
$\mathfrak{p}_{1}$ and $\mathfrak{p}_{2}$. We compute the completed
strict local rings at closed geometric points of $X_{U}$ and use
this computation to show that $X_{U}$ is locally etale over a product
of semistable schemes, which on the special fiber are closely related
to the strata $Y_{U,S,T}$. If we let $\mathcal{A}_{U}$ be the universal
abelian variety over $X_{U}$, then $\mathcal{A}_{U}$ has the same
stratification and the same geometry as $X_{U}$. 

Let $\xi$ be an irreducible algebraic representation of $G$ over
$\bar{\mathbb{Q}}_{l}$, which determines non-negative integers $t_{\xi},m_{\xi}$
and an endomorphism $a_{\xi}\in\mbox{End}(\mathcal{A}_{U}^{m_{\xi}}/X_{U})\otimes_{\mathbb{Z}}\mathbb{Q}$.
We are interested in understanding the $\Pi^{1,\infty}$-part of \[
H^{j}(X_{U},\mathcal{L}_{\xi})=a_{\xi}H^{j+m_{\xi}}(\mathcal{A}_{U}^{m_{\xi}},\bar{\mathbb{Q}}_{l}(t_{\xi})).\]
Thus, we study the cohomology of the generic fiber $H^{j}(\mathcal{A}_{U}^{m_{\xi}},\bar{\mathbb{Q}}_{l})$
and we do so via the cohomology of the complex of nearby cycles $R\psi\bar{\mathbb{Q}}_{l}$
over the special fiber of $\mathcal{A}_{U}^{m_{\xi}}$. The key ingredients
in studying the complex of nearby cycles together with the action
of monodromy are logarithmic schemes, the weight spectral sequence
as constructed by Saito \cite{Saito} (which on the level of complexes
of sheaves describes the action of monodromy on the complex of nearby
cycles for semistable schemes), and the formula \[
(R\psi\bar{\mathbb{Q}}_{l})_{X_{1}\times X_{2}}\simeq(R\psi\bar{\mathbb{Q}}_{l})_{X_{1}}\otimes^{L}(R\psi\bar{\mathbb{Q}}_{l})_{X_{2}},\]
when $X_{1}$ and $X_{2}$ are semistable schemes. Using these ingredients,
we deduce the existence of two spectral sequences which end up relating
terms of the form $H^{j}(\mathcal{A}_{U,S,T}^{m_{\xi}},\bar{\mathbb{Q}}_{l})$
(up to twisting and shifting) to the object we're interested in, $H^{j}(\mathcal{A}_{U}^{m_{\xi}},\bar{\mathbb{Q}}_{l})$.
The cohomology of each stratum $H^{j}(\mathcal{A}_{U,S,T}^{m_{\xi}},\bar{\mathbb{Q}}_{l})$
is closely related to the cohomology of Igusa varieties. The next
step is to compute the $\Pi^{1,\infty}$-part of the cohomology of
certain Igusa varieties, for which we adapt the strategy of Theorem
6.1 of \cite{Shin}. Using the result on Igusa varieties, we prove
Theorem \ref{Ramanujan} and then we also make use of the classification
of tempered representations. We prove that the $\Pi^{1,\infty}$-part
of each $H^{j}(\mathcal{A}_{U,S,T}^{m_{\xi}},\bar{\mathbb{Q}}_{l})$
vanishes outside the middle dimension and thus that both our spectral
sequences degenerate at $E_{1}$. The $E_{1}$ page of the second
spectral sequence provides us with the exact filtration of the $\Pi^{1,\infty}$-part
of \[
\lim_{\stackrel{\longrightarrow}{U^{p}}}H^{2n-2}(X_{U},\mathcal{L}_{\xi})\]
 which exhibits its purity. 
\begin{acknowledgement*}
I am very grateful to my advisor, Richard Taylor, for suggesting this
problem and for his constant encouragement and advice. I am also greatly
indebted to Luc Illusie, Sophie Morel, Arthur Ogus, Sug Woo Shin and
Teruyoshi Yoshida for their help in completing this work. I am grateful
to David Geraghty and Jack Thorne for useful conversations. Finally,
I would like to thank Sophie Morel and Sug Woo Shin for providing
comments on an earlier draft of this paper. 
\end{acknowledgement*}

\section{An integral model}

\subsection{Shimura varieties}

Let $E$ be an imaginary quadratic field in which $p$ splits, let
$c$ be the non-trivial element in $\mbox{Gal}(E/\mathbb{Q})$ and
choose a prime $u$ of $E$ above $p$.

Let $F_{1}$ be a totally real field of finite degree over $\mathbb{Q}$
and $w$ a prime of $F_{1}$ above $p$. Let $F_{2}$ be a quadratic
totally real extension of $F_{1}$ in which $w$ splits $w=w_{1}w_{2}$.
Let $d=[F_{2}:\mathbb{Q}]$ and we assume that $d\geq3$. Let $F=F_{2}.E$.
Let $\mathfrak{p}_{i}$ be the prime of $F$ above $w_{i}$ and $u$
for $i=1,2$. We denote by $\mathfrak{p}_{i}$ for $2<i\leq r$ the
rest of the primes which lie above the prime $u$ of $E$. We choose
embeddings $\tau_{i}:F\hookrightarrow\mathbb{C}$ with $i=1,2$ such
that $\tau_{2}=\tau_{1}\circ\sigma$, where $\sigma$ is the element
of $\mbox{Gal}(F/\mathbb{Q})$ which takes $\mathfrak{p}_{1}$ to
$\mathfrak{p}_{2}$. In particular, this means that $\tau_{E}:=\tau_{1}|_{E}=\tau_{2}|_{E}$
is well-defined. By abuse of notation we will also denote by $\sigma$
the Galois automorphism of $F_{2}$ taking $w_{1}$ to $w_{2}$. 

We will work with a Shimura variety corresponding to the PEL datum
$(F,*,V,\langle\cdot,\cdot\rangle,h)$, where $F$ is the CM field
defined above and $*=c$ is the involution given by complex conjugation.
We take $V$ to be the $F$-vector space $F^{n}$ for some integer
$n$. The pairing \[
\langle\cdot,\cdot\rangle:V\times V\to\mathbb{Q}\]
is a non-degenerate Hermitian pairing such that $\langle fv_{1},v_{2}\rangle=\langle v_{1},f^{*}v_{2}\rangle$
for all $f\in F$ and $v_{1},v_{2}\in V$. The last element we need
is an $\mathbb{R}$-algebra homomorphism $h:\mathbb{C}\to\mbox{End}_{F}(V)\otimes_{\mathbb{Q}}\mathbb{R}$
such that the bilinear pairing \[
(v_{1},v_{2})\to\langle v_{1},h(i)v_{2}\rangle\]
 is symmetric and positive definite. 

We define an algebraic group $G$ over $\mathbb{Q}$ by \[
G(R)=\{(g,\lambda)\in\mbox{End}_{F\otimes_{\mathbb{Q}}R}(V\otimes_{\mathbb{Q}}R)\times R^{\times}\mid\langle gv_{1},gv_{2}\rangle=\lambda\langle v_{1},v_{2}\rangle\}\]
for any $\mathbb{Q}$-algebra $R$. For $\sigma\in\mbox{Hom}_{E,\tau_{E}}(F,\mathbb{C})$
we let $(p_{\sigma},q_{\sigma})$ be the signature at $\sigma$ of
the pairing $\langle\cdot,\cdot\rangle$ on $V\otimes_{\mathbb{Q}}\mathbb{R}$.
We claim that we can find a PEL datum as above, such that $(p_{\tau},q_{\tau})=(1,n-1)$
for $\tau=\tau_{1}$ or $\tau_{2}$ and $(p_{\tau},q_{\tau})=(0,n)$
otherwise and such that $G_{\mathbb{Q}_{v}}$ is quasi-split at every
finite place $v$. 
\begin{lem}
Let $F$ be a CM field as above. For any embeddings $\tau_{1},\tau_{2}:F\hookrightarrow\mathbb{C}$
there exists a PEL datum $(F,*,V,\langle\cdot,\cdot\rangle,h)$ as
above such that the associated group $G$ is quasi-split at every
finite place and has signature $(1,n-1)$ at $\tau_{1}$ and $\tau_{2}$
and $(0,n)$ everywhere else. \end{lem}
\begin{proof}
This lemma is standard and follows from computations in Galois cohomology
found in section 2 of \cite{Cl1}, but see also lemma 1.7 of \cite{H-T}.
The problem is that of constructing a global unitary similitude group
with prescribed local conditions. It is enough to consider the case
of a unitary group $G^{0}$ over $\mathbb{Q}$, by taking it to be
the algebraic group defined by $\ker(G(R)\to R^{\times})$ sending
$(g,\lambda)\mapsto\lambda$. 

A group $G$ defined as above has a quasi-split inner form over $\mathbb{Q}$
denoted $G_{n}$, defined as in section 3 of \cite{Shin}. This inner
form $G_{n}$ is the group of similitudes which preserve the non-degenerate
Hermitian pairing $\langle v_{1},v_{2}\rangle=v_{1}\zeta\Phi^{t}v_{2}^{c}$
with $\Phi\in GL_{n}(\mathbb{Q})$ having entries \[
\Phi_{ij}=(-1)^{i+1}\delta_{i,n+1-j}\]
and $\zeta\in F^{*}$ an element of trace $0$. Let $G'$ be the adjoint
group of $G_{n}^{0}$. It suffices to show that the tuple of prescribed
local conditions, classified by elements in $\oplus_{v}H^{1}(F_{2,v},G')$
is in the image of \[
H^{1}(F_{2},G')\to\oplus_{v}H^{1}(F_{2,v},G'),\]
where the sum is taken over all places $v$ of $F_{2}$. For $n$
odd, Lemma 2.1 of \cite{Cl1} ensures that the above map is surjective,
so there is no cohomological obstruction for finding the global unitary
group. For $n$ even the image of the above map is equal to the kernel
of \[
\bigoplus_{v}H^{1}(F_{2,v},G')\to\mathbb{Z}/2\mathbb{Z}.\]
We can use Lemma 2.2 of \cite{Cl1} to compute all the local invariants
(i.e. the images of $H^{1}(F_{2,v},G')\to\mathbb{Z}/2\mathbb{Z}$
for all places $v)$. At the finite places, the sum of the invariants
is $0\pmod{2}$ (this is guaranteed by the existence of the quasi-split
inner form $G{}_{n}$ of $G$, which has the same local invariants
at finite places). At the infinite places $\tau_{1}$ and $\tau_{2}$
the invariants are $\frac{n}{2}+1\pmod{2}$ and at all other infinite
places they are $\frac{n}{2}\pmod{2}$. The global invariant is $\frac{nd}{2}+2\pmod{2}$,
where $d$ is the degree of $F_{2}$ over $\mathbb{Q}$. Since $d$
is even, the image in $\mathbb{Z}/2\mathbb{Z}$ is equal to $0\pmod{2}$,
so the prescribed local unitary groups arise from a global unitary
group. 
\end{proof}
We will choose the $\mathbb{R}$-homomorphism $h:$$\mathbb{C}\to\mbox{End}_{F}(V)\otimes_{\mathbb{Q}}\mathbb{R}$
such that under the natural $\mathbb{R}$-algebra isomorphism $\mbox{End}_{F}(V)_{\mathbb{R}}\simeq\prod_{\tau|_{E}=\tau_{E}}M_{n}(\mathbb{C})$
it equals \[
z\mapsto\left(\left(\begin{array}{cc}
zI_{p_{\tau}} & 0\\
0 & \bar{z}I_{q_{\tau}}\end{array}\right)_{\tau}\right),\]
where $\tau$ runs over elements of $\mbox{Hom}_{E,\tau_{E}}(F,\mathbb{C})$. 

Now that we've defined the PEL datum we can set up our moduli problem.
Note that the reflex field of the PEL datum is $F'=F_{1}\cdot E$.
Let $S/F'$ be a scheme and $A/S$ an abelian scheme of dimension
$dn$. Suppose we have an embedding $i:F\hookrightarrow\mbox{End}(A)\otimes_{\mathbb{Z}}\mathbb{Q}$.
$\mbox{Lie}A$ is a locally free $\mathcal{O}_{S}$-module of rank
$dn$ with an action of $F$. We can decompose $\mbox{Lie}A=\mbox{Lie}^{+}A\oplus\mbox{Lie}^{-}A$
where $\mbox{Lie}^{+}A=\mbox{Lie}A\otimes_{\mathcal{O}_{S}\otimes E}\mathcal{O_{S}}$
and the map $E\hookrightarrow F'\to\mathcal{O_{S}}$ is the natural
map followed by the structure map. $\mbox{Lie}^{-}A$ is defined in
the same way using the complex conjugate of the natural map $E\hookrightarrow F'$.
We ask that $\mbox{Lie}^{+}A$ be a free $\mathcal{O}_{S}$-module
of rank $2$ and that $\mbox{Lie}^{+}A\simeq\mathcal{O}_{S}\otimes_{F_{1}}F_{2}$
as an $\mathcal{O}_{S}$-module with an action of $F_{2}$. 
\begin{defn}
\label{compatibility}If the the conditions above are satisfied, we
will call the pair $(A,i)$ \emph{compatible}. \end{defn}
\begin{rem*}
This is an adptation to our situation of the notion of compatibility
defined in section III.1 of \cite{H-T}, which fulfills the same purpose
as the determinant condition defined on page 390 of \cite{Kottwitz}.
\end{rem*}
For an open compact subgroup $U\subset G(\mathbb{A}^{\infty})$ we
consider the contravariant functor $\mathfrak{X}_{U}$ mapping \[
\left(\begin{array}{c}
\mbox{Connected, locally noetherian }\\
F'\mbox{-schemes with geometric point}\\
(S,s)\end{array}\right)\to\left(\mbox{Sets}\right)\]
\[
(S,s)\mapsto\{(A,\lambda,i,\bar{\eta})\}/\sim\]
 where 
\begin{itemize}
\item $A$ is an abelian scheme over $S$;
\item $\lambda:A\to A^{\vee}$ is a polarization;
\item $i:F\hookrightarrow\mbox{End}^{0}(A)=\mbox{End}A\otimes_{\mathbb{Z}}\mathbb{Q}$
is such that $(A,i)$ is compatible and $\lambda\circ i(f)=i(f^{*})^{\vee}\circ\lambda,$
for all $f\in F$;
\item $\bar{\eta}$ is a $\pi_{1}(S,s)$-invariant $U$-orbit of isomorphism
of Hermitian $F\otimes_{\mathbb{Q}}\mathbb{A}^{\infty}$-modules \[
\eta:V\otimes_{\mathbb{Q}}\mathbb{A}^{\infty}\to VA_{s}\]
which take the fixed pairing $\langle\cdot,\cdot\rangle$ on $V$
to on $(\mathbb{A}^{\infty})^{\times}$-multiple of the $\lambda$-Weil
pairing on $VA_{s}$. Here, \[
VA_{s}=\left(\lim_{\leftarrow}A[N](k(s))\right)\otimes_{\mathbb{Z}}\mathbb{Q}\]
 is the adelic Tate module. 
\end{itemize}
We consider two quadruples as above equivalent if there is an isogeny
between the abelian varieties which is compatible with the additional
structures. If $s'$ is a different geometric point of $S$ then there
is a canonical bijection between $\mathfrak{X}_{U}(S,s)$ and $\mathfrak{X}_{U}(S,s')$.
We can forget about the geometric points and extend the definition
from connected to arbitrary locally noetherian $F'$-schemes. When
$U$ is sufficiently small, this functor is representable by a smooth
and quasi-projective Shimura variety $X_{U}/F'$ of dimension $2n-2$
(this is explained on page 391 of \cite{Kottwitz}). As $U$ varies,
the inverse system of the $X_{U}$ has a natural right action of $G(\mathbb{A}^{\infty})$. 

Let $\mathcal{A}_{U}$ be the universal abelian variety over $X_{U}$.
The action of $G(\mathbb{A}^{\infty})$ on the inverse system of the
$X_{U}$ extends to an action by quasi-isogenies on the inverse system
of the $\mathcal{A}_{U}$. The following construction goes through
as in section III.2 of \cite{H-T}. Let $l$ be a rational prime different
from $p$ and $\xi$ an irreducible algebraic representation of $G$
over $\mathbb{Q}_{l}^{ac}$. This defines a lisse $\mathbb{Q}_{l}^{ac}$-sheaf
$\mathcal{L}_{\xi,l}$ over each $X_{U}$ and the action of $G(\mathbb{A}^{\infty}$)
extends to the inverse system of sheaves. The direct limit \[
H^{i}(X,\mathcal{L}_{\xi,l})=\lim_{\to}H^{i}(X_{U}\times_{F'}\bar{F}',\mathcal{L}_{\xi,l})\]
is a (semisimple) admissible representation of $G(\mathbb{A}^{\infty})$
with a continuous action of $\mbox{Gal}(\bar{F}'/F')$. We can decompose
it as \[
H^{i}(X,\mathcal{L}_{\xi,l})=\bigoplus_{\pi}\pi\otimes R_{\xi,l}^{i}(\pi)\]
where the sum runs over irreducible admissible representations $\pi$
of $G(\mathbb{A}^{\infty})$ over $\mathbb{Q}_{l}^{ac}$. The $R_{\xi,l}^{i}(\pi)$
are finite dimensional continuous representations of $\mbox{Gal}(\bar{F}'/F')$
over $\mathbb{Q}_{l}^{ac}$. We shall suppress the $l$ from $\mathcal{L}_{\xi.l}$
and $R_{\xi,l}^{i}(\pi)$ where it is understood from context. To
the irreducible representation $\xi$ of $G$ we can associate as
in section III.2 of \cite{H-T} non-negative integers $m_{\xi}$ and
$t_{\xi}$ and an idempotent $\epsilon_{\xi}\in\mathbb{Q}[S_{m_{\xi}}]$
(where $S_{m_{\xi}}$ is the symmetric group on $m_{\xi}$ letters).
As on page 476 of \cite{T-Y}, define for each integer $N\geq2$,\[
\epsilon(m_{\xi},N)=\prod_{x=1}^{m_{\xi}}\prod_{y\not=1}\frac{[N]_{x}-N}{N-N^{y}}\in\mathbb{Q}[(N^{\mathbb{Z}_{\geq0}})^{m_{\xi}}],\]
where $[N]_{x}$ denotes the endomorphism generated by multiplication
by $N$ on the $x$-th factor and $y$ ranges from $0$ to $2[F_{2}:\mathbb{Q}]n^{2}$
but excluding $1$. Set \[
a_{\xi}=a_{\xi,N}=\epsilon_{\xi}P(\epsilon(m_{\xi},N)),\]
which can be thought of as an element of $\mbox{End}(\mathcal{A}_{U}^{m_{\xi}}/X_{U})\otimes_{\mathbb{Z}}\mathbb{Q}$.
Here $P(\epsilon(m_{\xi},N))$ is the polynomial \[
P(X)=((X-1)^{4n-3}+1)^{4n-3}.\]
If we let $\mbox{proj}:\mathcal{A}_{U}^{m_{\xi}}\to X_{U}$ be the
natural projection, then $\epsilon(m_{\xi},N)$ is an idempotent on
each of the sheaves $R^{j}\mbox{proj}_{*}\bar{\mathbb{Q}}_{l}(t_{\xi})$,
hence also on \[
H^{i}(X_{U}\times_{F'}\bar{F}',R^{j}\mbox{proj}_{*}\bar{\mathbb{Q}}_{l}(t_{\xi}))\Rightarrow H^{i+j}(\mathcal{A}_{U}^{m_{\xi}}\times_{F'}\bar{F}',\bar{\mathbb{Q}}_{l}(t_{\xi})).\]
We get an endomorphism $\epsilon(m_{\xi},N)$ of $H^{i+j}(\mathcal{A}_{U}^{m_{\xi}}\times_{F'}\bar{F}',\bar{\mathbb{Q}}_{l}(t_{\xi}))$
which is an idempotent on each graded piece of a filtration of length
at most $4n-3$. In this case, $P(\epsilon(m_{\xi},N))$ must be an
idempotent on all of $H^{i+j}(\mathcal{A}_{U}^{m_{\xi}}\times_{F'}\bar{F}',\bar{\mathbb{Q}}_{l}(t_{\xi}))$.
We have an isomorphism \[
H^{i}(X_{U}\times_{F'}\bar{F}',\mathcal{L}_{\xi})\cong a_{\xi}H^{i+m_{\xi}}(\mathcal{A}_{U}^{m_{\xi}}\times_{F'}\bar{F'},\bar{\mathbb{Q}}_{l}(t_{\xi})),\]
which commutes with the action of $G(\mathbb{A}^{\infty})$.

\subsection{An integral model for Iwahori level structure }

Let $K=F_{\mathfrak{p}_{1}}\simeq F_{\mathfrak{p}_{2}}$, where the
isomorphism is via $\sigma$, denote by $\mathcal{O}_{K}$ the ring
of integers of $K$ and by $\pi$ a uniformizer of $\mathcal{O}_{K}$.

Let $S/\mathcal{O}_{K}$ be a scheme and $A/S$ an abelian scheme
of dimension $dn$. Suppose we have an embedding $i:\mathcal{O}_{F}\hookrightarrow\mbox{End}(A)\otimes_{\mathbb{Z}}\mathbb{Z}_{(p)}$.
$\mbox{Lie}A$ is a locally free $\mathcal{O}_{S}$-module of rank
$dn$ with an action of $F$. We can decompose $\mbox{Lie}A=\mbox{Lie}^{+}A\oplus\mbox{Lie}^{-}A$
where $\mbox{Lie}^{+}A=\mbox{Lie}A\otimes_{\mathbb{Z}_{p}\otimes\mathcal{O}_{E}}\mathcal{O}_{E,u}$.
There are two natural actions of $\mathcal{O}_{F}$ on $\mbox{Lie}^{+}A$,
via $\mathcal{O}_{F}\to\mathcal{O}_{F_{\mathfrak{p}_{j}}}\stackrel{\sim}{\to}\mathcal{O}_{K}$
composed with the structure map for $j=1,2$. These two actions differ
by the automorphism $\sigma\in\mbox{Gal}(F/\mathbb{Q})$. There is
also a third action via the embedding $i$ of $\mathcal{O}_{F}$ into
the ring of endomorphisms of $A$. We ask that $\mbox{Lie}^{+}A$
be locally free of rank $2$, that the part of $\mbox{Lie}^{+}A$
where the first action of $\mathcal{O}_{F}$ on $\mbox{Lie}^{+}A$
coincides with $i$ be locally free of rank $1$ and that the part
where the second action coincides with $i$ also be locally free of
rank $1$. 
\begin{defn}
If the above conditions are satisfied, then we call $(A,i)$ \emph{compatible}.
One can check that for $S/K$ this notion of compatibility coincides
with the one in Definition \ref{compatibility}. 
\end{defn}
If $p$ is locally nilpotent on $S$ then $(A,i)$ is compatible if
and only if
\begin{itemize}
\item $A[\mathfrak{p}_{i}^{\infty}]$ is a compatible, one-dimensional Barsotti-Tate
$\mathcal{O}_{K}$-module for $i=1,2$ and
\item $A[\mathfrak{p}_{i}^{\infty}]$ is ind-etale for $i>2$. 
\end{itemize}
We will now define a few integral models for our Shimura varieties
$X_{U}$. We can decompose $G(\mathbb{A}^{\infty})$ as \[
G(\mathbb{A}^{\infty})=G(\mathbb{A}^{\infty,p})\times\mathbb{Q}_{p}^{\times}\times\prod_{i=1}^{r}GL_{n}(F_{\mathfrak{p}_{i}}).\]
For each $i$, let $\Lambda_{i}$ be an $\mathcal{O}_{F_{\mathfrak{p}_{i}}}$-lattice
in $F_{\mathfrak{p}_{i}}^{n}$ which is stable under $GL_{n}(\mathcal{O}_{F_{\mathfrak{p}_{i}}})$
and self-dual with respect to $\langle\cdot,\cdot\rangle$. For each
$\vec{m}=(m_{1},\dots,m_{r})$ and compact open $U^{p}\subset G(\mathbb{A}^{\infty,p})$
we define the compact open subgroup $U^{p}(\vec{m})$ of $G(\mathbb{A}^{\infty})$
as \[
U^{p}(\vec{m})=U^{p}\times\mathbb{Z}_{p}^{\times}\times\prod_{i=1}^{r}\ker(GL_{\mathcal{O}_{F_{\mathfrak{p}_{i}}}}(\Lambda_{i})\to GL_{\mathcal{O}_{F_{\mathfrak{p}_{i}}}}(\Lambda_{i}/\mathfrak{m}_{F_{\mathfrak{p}_{i}}}^{m_{i}}\Lambda_{i})).\]
The corresponding moduli problem of sufficiently small level $U^{p}(\vec{m})$
over $\mathcal{O}_{K}$ is given by the functor 

\[
\left(\begin{array}{c}
\mbox{Connected, locally noetherian }\\
\mathcal{O}_{K}\mbox{-schemes with geometric point}\\
(S,s)\end{array}\right)\to\left(\mbox{Sets}\right)\]
\[
(S,s)\mapsto\{(A,\lambda,i,\bar{\eta}^{p},\{\alpha_{i}\}_{i=1}^{r})\}/\sim\]
 where 
\begin{itemize}
\item $A$ is an abelian scheme over $S$;
\item $\lambda:A\to A^{\vee}$ is a prime-to-$p$ polarization;
\item $i:\mathcal{O}_{F}\hookrightarrow\mbox{End}(A)\otimes_{\mathbb{Z}}\mathbb{Z}_{(p)}$
such that $(A,i)$ is compatible and $\lambda\circ i(f)=i(f^{*})^{\vee}\circ\lambda,\forall f\in\mathcal{O}_{F}$;
\item $\bar{\eta}^{p}$ is a $\pi_{1}(S,s)$-invariant $U^{p}$-orbit of
isomorphisms of Hermitian $F\otimes_{\mathbb{Q}}\mathbb{A}^{\infty,p}$-modules
\[
\eta:V\otimes_{\mathbb{Q}}\mathbb{A}^{\infty,p}\to V^{p}A_{s}\]
which take the fixed pairing $\langle\cdot,\cdot\rangle$ on $V$
to an $(\mathbb{A}^{\infty,p})^{\times}$-multiple of the $\lambda$-Weil
pairing on $VA_{s}$. Here $V^{p}A_{s}$ is the adelic Tate module
away from $p$;
\item for $i=1,2$, $\alpha_{i}:\mathfrak{p}_{i}^{-m_{i}}\Lambda_{i}/\Lambda_{i}\to A[\mathfrak{p}_{i}^{m_{i}}]$
is a Drinfeld $\mathfrak{p}_{i}^{m_{i}}$-structure;
\item for $i>2$, $\alpha_{i}:(\mathfrak{p}_{i}^{-m_{i}}\Lambda_{i}/\Lambda)\stackrel{\sim}{\to}A[\mathfrak{p}_{i}^{m_{i}}]$
is an isomorphism of $S$-schemes with $\mathcal{O}_{F_{\mathfrak{p}_{i}}}$-actions; 
\item Two tuples $(A,\lambda,i,\bar{\eta}^{p},\{\alpha_{i}\}_{i=1}^{r})$
and $(A',\lambda',i',(\bar{\eta}^{p})^{'},\{\alpha'_{i}\}_{i=1}^{r}$
are equivalent if there is a prime-to-p isogeny $A\to A'$ taking
$\lambda,i,\bar{\eta}^{p},\alpha_{i}$ to $\gamma\lambda',i',(\bar{\eta}^{p})^{'},\alpha'_{i}$
for some $\gamma\in\mathbb{Z}_{(p)}^{\times}$.
\end{itemize}
This moduli problem is representable by a projective scheme over $\mathcal{O}_{K}$,
which will be denoted $X_{U^{p},\vec{m}}.$ The projectivity follows
from Theorem 5.3.3.1 and Remark 5.3.3.2 of \cite{Lan}. If $m_{1}=m_{2}=0$
this scheme is smooth as in Lemma III.4.1.2 of \cite{H-T}, since
we can check smoothness on the completed strict local rings at closed
geometric points and these are isomorphic to deformation rings for
$p$-divisible groups (with level structure only at $\mathfrak{p}_{i}$
for $i>2$, when the $p$-divisible group is etale). Moreover, if
$m_{1}=m_{2}=0$ the dimension of $X_{U^{p},\vec{m}}$ is $2n-1$. 

When $m_{1}=m_{2}=0$, we will denote $X_{U^{p},\vec{m}}$ by $X_{U_{0}}$.
If $\mathcal{A}_{U_{0}}$ is the universal abelian scheme over $X_{U_{0}}$
we write $\mathcal{G}_{i}=\mathcal{A}_{U_{0}}[\mathfrak{p}_{i}^{\infty}]$
for $i=1,2$ and $\mathcal{G}=\mathcal{G}_{1}\times\mathcal{G}_{2}$.
Over a base where $p$ is nilpotent, each of the $\mathcal{G}_{i}$
is a one-dimensional compatible Barsotti-Tate $\mathcal{O}_{K}$-module. 

Let $\bar{X}_{U_{0}}=X_{U_{0}}\times_{\mbox{Spec }\mathcal{O}_{K}}\mbox{Spec }\mathbb{F}$
be the special fiber of $X_{U_{0}}$. We define a stratification on
$\bar{X}_{U_{0}}$ in terms of $0\leq h_{1},h_{2}<n-1$. The scheme
$\bar{X}_{U_{0}}^{[h_{1},h_{2}]}$ will be the reduced closed subscheme
of $\bar{X}_{U_{0}}$ whose closed geometric points $s$ are those
for which the maximal etale quotient of $\mathcal{G}_{i}$ has $\mathcal{O}_{K}$-height
at most $h_{i}$. Let $\bar{X}_{U_{0}}^{(h_{1},h_{2})}=\bar{X}_{U_{0}}^{[h_{1},h_{2}]}-(\bar{X}_{U_{0}}^{[h_{1}-1,h_{2}]}\cup\bar{X}_{U_{0}}^{[h_{1},h_{2}-1]})$. 
\begin{lem}
The scheme $\bar{X}_{U_{0}}^{(h_{1},h_{2})}$ is non-empty and smooth
of pure dimension $h_{1}+h_{2}$. \end{lem}
\begin{proof}
In order to see that this is true, note that the formal completion
of $\bar{X}_{U_{0}}$ at any closed point is isomorphic to $\bar{\mathbb{F}}[[T_{2},\dots,T_{n},S_{2},\dots,S_{n}]]$
since it is the universal formal deformation ring of a product of
two one-dimensional compatible Barsotti-Tate groups of height $n$
each. (In fact it is the product of the universal deformation rings
for each of the two Barsotti-Tate groups.) Thus, $\bar{X}_{U_{0}}$
has dimension $2n-2$ and as in Lemma II.1.1 of \cite{H-T} each closed
stratum $\bar{X}_{U_{0}}^{[h_{1},h_{2}]}$ has dimension at least
$h_{1}+h_{2}$. The lower bound on dimension also holds for each open
stratum $\bar{X}_{U_{0}}^{(h_{1},h_{2})}$. In order to get the upper
bound on the dimension it suffices to show that the lowest stratum
$\bar{X}_{U_{0}}^{(0,0)}$ is non-empty. Indeed, once we have a closed
point $s$ in any stratum $\bar{X}_{U_{0}}^{(h_{1},h_{2})}$, we can
compute the formal completion $(\bar{X}_{U_{0}}^{(h_{1},h_{2})})_{s}^{\wedge}$
as in Lemma II.1.3 of \cite{H-T} and find that the dimension is exactly
$h_{1}+h_{2}$. We start with a closed point of the lowest stratum
$\bar{X}_{U_{0}}^{(0,0)}=\bar{X}_{U_{0}}^{[0,0]}$ and prove that
this stratum has dimension $0$. The higher closed strata $\bar{X}_{U_{0}}^{[h_{1},h_{2}]}=\cup_{j_{1}\leq h_{1},j_{2}\leq h_{2}}\bar{X}_{U_{0}}^{(j_{1},j_{2})}$
are non-empty and it follows by induction on $(h_{1},h_{2})$ that
the open strata $\bar{X}_{U_{0}}^{(h_{1},h_{2})}$ are also non-empty. 

It remains to see that $\bar{X}_{U_{0}}^{(0,0)}$ is non-empty. This
can be done using Honda-Tate theory as in the proof of Corollary V.4.5.
of \cite{H-T}, whose ingredients for Shimura varieties associated
to more general unitary groups are supplied in sections 8 through
12 of \cite{Shin-1}. In our case, Honda-Tate theory exhibits a bijection
between $p$-adic types over $F$ (see section 8 of \cite{Shin-1}
for the general definition) and pairs $(A,i)$ where $A/\bar{\mathbb{F}}$
is an abelian variety of dimension $dn$ and $i:F\hookrightarrow\mbox{End}(A)\otimes_{\mathbb{Z}}\mathbb{Q}$.
The abelian variety $A$ must also satisfy the following: $A[\mathfrak{p}_{i}^{\infty}]$
is ind-etale for $i>2$ and $A[\mathfrak{p}_{i}^{\infty}]$ is one-dimensional
of etale height $h_{i}$ for $i=1,2$. Note that the slopes of the
p-divisible groups $A[\mathfrak{p}_{i}^{\infty}]$ are fixed for all
$i$. All our $p$-adic types will be simple and given by pairs $(M,\eta)$
where $M$ is a CM field extension of $F$ and $\eta\in\mathbb{Q}[\mathfrak{P}]$
where $\mathfrak{P}$ is the set of places of $M$ above $p$. The
coefficients in $\eta$ of places $x$ of $M$ above $\mathfrak{p}_{i}$
are related to the slope of the corresponding $p$-divisible group
at $\mathfrak{p}_{i}$ as in Corollary 8.5 of \cite{Shin}. More precisely,
$A[x^{\infty}]$ has pure slope $\eta_{x}/e_{x/p}$. It follows that
the coefficients of $\eta$ at places $x$ and $x^{c}$ above $p$
satisfy the compatibility \[
\eta_{x}+\eta_{x^{c}}=e_{x/p}\]
 so to know $\eta$ it is enough to specify $\eta_{x}\cdot x$ as
$x$ runs through places of $M$ above $u$. 

In order to exhibit a pair $(A,i)$ with the right slope of $A[\mathfrak{p}_{i}^{\infty}]$
it suffices to exhibit its corresponding $p$-adic type. For this,
we can simply take $M=F$ and $\eta_{\mathfrak{p}_{i}}=\frac{e_{\mathfrak{p}_{i}/p}}{n[F_{\mathfrak{p}_{i}}:\mathbb{Q}_{p}]}\cdot\mathfrak{p}_{i}$
for $i=1,2$ and $\eta_{\mathfrak{p}_{i}}=0$ otherwise. The only
facts remaining to be checked are that the associated pair $(A,i)$
has a polarization $\lambda$ which induces $c$ on $F$ (this follows
from Lemma 9.2 of \cite{Kottwitz}) and that the triple $(A,i,\lambda)$
can be given additional structure to make it into a point on $\bar{X}_{U_{0}}^{(0,0)}$.
This will be proven in more generality in Lemma \ref{vanishing of kottwitz invariant},
an analogue of Lemma V.4.1 of \cite{H-T}. Note that the argument
is not circular, since the proof of Lemma \ref{vanishing of kottwitz invariant}
is independent of this section. 
\end{proof}
The next Lemma is an analogue of Lemma 3.1 of \cite{T-Y}.
\begin{lem}
\label{closure of stratum}If $0\leq h_{1},h_{2}\leq n-1$ then the
Zariski closure of $\bar{X}_{U_{0}}^{(h_{1},h_{2})}$ contains $\bar{X}_{U_{0}}^{(0,0)}$
. \end{lem}
\begin{proof}
The proof follows exactly like the proof of Lemma 3.1 of \cite{T-Y}.
Let $x$ be a closed geometric point of $\bar{X}_{U_{0}}^{(0,0)}$.
The main point is to note that the formal completion of $\bar{X}_{U_{0}}\times\mbox{Spec }\bar{\mathbb{F}}$
at $x$ is isomorphic to the equicharacteristic universal deformation
ring of $\mathcal{G}_{1,x}\times\mathcal{G}_{2,x}$, so it is isomorphic
to \[
\mbox{Spf }\bar{\mathbb{F}}[[T_{2},\dots,T_{n},S_{2},\dots,S_{n}]].\]
We can choose the $T_{i}$, the $S_{i}$ and formal parameters $X$
on the universal deformation of $\mathcal{G}_{1,x}$ and $Y$ on the
universal deformation of $\mathcal{G}_{2,x}$ such that \[
[\pi](X)\equiv\pi X+\sum_{i=2}^{n}T_{i}X^{\#\mathbb{F}^{i-1}}+X^{\#\mathbb{F}^{n}}\pmod{X^{\#\mathbb{F}^{n}+1}}\mbox{ and }\]
\[
[\pi](Y)\equiv\pi Y+\sum_{i=2}^{n}S_{i}X^{\#\mathbb{F}^{i-1}}+S^{\#\mathbb{F}^{n}}\pmod{S^{\#\mathbb{F}^{n}+1}}.\]
We get a morphism \[
\mbox{Spec }\bar{\mathbb{F}}[[T_{2},\dots,T_{n},S_{2},\dots,S_{n}]]\to\bar{X}_{U_{0}}\]
lying over $x:\mbox{ Spec }\bar{\mathbb{F}}\to\bar{X}_{U_{0}}$ such
that if $k$ denotes the algebraic closure of the field of fractions
of \[
\mbox{Spec }\bar{\mathbb{F}}[[T_{2},\dots,T_{n},S_{2},\dots S_{n}]]/(T_{2},\dots,T_{n-h_{1}},S_{2},\dots,S_{n-h_{2}})\]
then the induced map $\mbox{Spec }k\to\bar{X}_{U_{0}}$ factors through
$\bar{X}_{U_{0}}^{(h_{1},h_{2})}$. 
\end{proof}
For $i=1,2$, let $\mbox{Iw}_{n,\mathfrak{p}_{i}}$ be the subgroup
of matrices in $GL_{n}(\mathcal{O}_{K})$ which reduce modulo $\mathfrak{p}_{i}$
to $B_{n}(\mathbb{F})$. We will define an integral model for $X_{U}$,
where $U\subseteq G(\mathbb{A}^{\infty})$ is equal to \[
U^{p}\times U_{p}^{\mathfrak{p}_{1},\mathfrak{p}_{2}}(m)\times\mbox{Iw}_{n,\mathfrak{p}_{1}}\times\mbox{Iw}_{n,\mathfrak{p}_{2}}\times\mathbb{Z}_{p}^{\times}.\]
We define the following functor $\mathfrak{X}_{U}$ from connected
locally noetherian $\mathcal{O}_{K}$-schemes with a geometric point
to sets sending \[
(S,s)\mapsto(A,\lambda,i,\bar{\eta}^{p},\mathcal{C}_{1},\mathcal{C}_{2},\alpha_{i}),\]
where $(A,\lambda,i,\bar{\eta}^{p},\alpha_{i})$ is as in the definition
of $X_{U_{0}}$ and for $i=1,2,$ $\mathcal{C}_{i}$ is a chain of
isogenies \[
\mathcal{C}_{i}:\mathcal{G}_{i,A}=\mathcal{G}_{i,0}\to\mathcal{G}_{i,1}\to\dots\to\mathcal{G}_{i,n}=\mathcal{G}_{i,A}/\mathcal{G}_{i,A}[\mathfrak{p}_{i}]\]
of compatible Barsotti-Tate $\mathcal{O}_{K}$-modules each of degree
$\#\mathbb{F}$ and with composite the canonical map $\mathcal{G}_{i,A}\to\mathcal{G}_{i,A}/\mathcal{G}_{i.A}[\mathfrak{p}_{i}]$. 
\begin{lem}
\label{dimension}If $U^{p}$ is sufficiently small, the functor $\mathfrak{X}_{U}$
is represented by a scheme $X_{U}$ which is finite over $X_{U_{0}}$.
The scheme $X_{U}$ has some irreducible components of dimension $2n-1$. \end{lem}
\begin{proof}
The chains of isogenies $\mathcal{C}_{i}$ can be viewed as flags
\[
0=\mathcal{K}_{i,0}\subset\mathcal{K}_{i,1}\dots\subset\mathcal{K}_{i,n}=\mathcal{G}_{i}[\mathfrak{p}_{i}],\]
where $\mathcal{K}_{i,j}=\ker(\mathcal{G}_{i,0}\to\mathcal{G}_{i,j})$.
All the $\mathcal{K}_{i,j}$ are closed finite flat subgroup schemes
with $\mathcal{O}_{K}$-action and $\mathcal{K}_{i,j}/\mathcal{K}_{i,j-1}$
of order $\#\mathbb{F}$. The representability can be proved in the
same way as in Lemma 3.2 of \cite{T-Y} except in two steps: first
we note that the functor sending $S$ to points of $X_{U_{0}}(S)$
together with flags $\mathcal{C}_{1}$ of $\mathcal{G}_{1}[\mathfrak{p}_{1}]$
is representable by a scheme $X'_{U}$ over $X_{U_{0}}$. (If we let
$\mathcal{H}_{1}$ denote the sheaf of Hopf algebras over $X_{U_{0}}$
defining $\mathcal{G}_{1}[\mathfrak{p}_{1}]$, then $X'_{U}$ will
be a closed subscheme of the Grassmanian of chains of locally free
direct summands of $\mathcal{H}_{1}$.) Then we see in the same way
that the functor sending $S$ to points of $X'_{U}(S)$ together with
flags $\mathcal{C}_{2}$ of $\mathcal{G}_{2}[\mathfrak{p}_{2}]$ is
representable by a scheme $X_{U}$ over $X'_{U}$. We also have that
$X_{U}$ is projective and finite over $X_{U_{0}}$. (Indeed, for
each closed geometric point $x$ of $X_{U_{0}}$ there are finitely
many choices of flags of $\mathcal{O}_{K}$-submodules of each $\mathcal{G}_{i,x}$.)
On the generic fiber, the morphism $X_{U}\to X_{U_{0}}$ is finite
etale and $X_{U_{0}}$ has dimension $2n-1$, so $X_{U}$ has some
components of dimension $2n-1$. 
\end{proof}
We say that an isogeny $\mathcal{G}\to\mathcal{G}'$ of one-dimensional
compatible Barsotti-Tate $\mathcal{O}_{K}$-modules of degree $\#\mathbb{F}$
has connected kernel if it induces the zero map on $\mbox{Lie }\mathcal{G}$.
If we let $f=[\mathbb{F}:\mathbb{F}_{p}]$ and $F:\mathcal{G}\to\mathcal{G}^{(p)}$
be the Frobenius map, then $F^{f}:\mathcal{G}\to\mathcal{G}^{(\#\mathbb{F})}$
is an isogeny of one-dimensional compatible Barsotti-Tate $\mathcal{O}_{K}$-modules
and has connected kernel. The following lemma appears as Lemma 3.3
in \cite{T-Y}.
\begin{lem}
\label{rigidity lemma}Let $W$ denote the ring of integers of the
completion of the maximal unramified extension of $K$. Suppose that
$R$ is an Artinian local $W$-algebra with residue field $\bar{\mathbb{F}}.$
Suppose that \[
\mathcal{C}:\mathcal{G}_{0}\to\mathcal{G}_{1}\to\dots\to\mathcal{G}_{g}=\mathcal{G}_{0}/\mathcal{G}_{0}[\mathfrak{p}_{i}]\]
is a chain of isogenies of degree $\#\mathbb{F}$ of one-dimensional
compatible formal Barsotti-Tate $\mathcal{O}_{K}$-modules over $R$
of $\mathcal{O}_{K}$-height $g$ with composite equal to multiplication
by $\pi$. If every isogeny has connected kernel then $R$ is a $\bar{\mathbb{F}}$-algebra
and $\mathcal{C}$ is the pullback of a chain of isogenies of Barsotti-Tate
$\mathcal{O}_{K}$-modules over $\bar{\mathbb{F}}$, with all isogenies
isomorphic to $F^{f}$. 
\end{lem}
Now let $\bar{X}_{U}=X_{U}\times_{\mbox{Spec }K}\mbox{Spec }\mathbb{F}$
denote the special fiber of $X_{U}$. For $i=1,2$ and $1\leq j\leq n,$
let $Y_{i,j}$ denote the closed subscheme of $\bar{X}_{U}$ over
which $\mathcal{G}_{i,j-1}\to\mathcal{G}_{i,j}$ has connected kernel.
Note that, since each $\mbox{Lie}\mathcal{G}_{i,j}$ is locally free
of rank $1$ over $\mathcal{O}_{X_{U}}$, we can pick a local basis
for all of them. Then we can find locally $X_{i,j}\in\mathcal{O}_{X_{U}}^{\times}$
to represent the linear maps $\mbox{Lie}\mathcal{G}_{i,j-1}\to\mbox{Lie}\mathcal{G}_{i,j}$.
Thus, each $Y_{i,j}$ is cut out locally in $X_{U}$ by the equation
$X_{i,j}=0$. 
\begin{prop}
Let $s$ be a closed geometric point of $X_{U}$ such that $\mathcal{G}_{i,s}$
has etale height $h_{i}$ for $i=1,2$. Let $W$ be the ring of integers
of the completion of the maximal unramified extension of $K$. Let
$\mathcal{O}_{X_{U},s}^{\wedge}$ be the completion of the strict
henselization of $X_{U}$ at $s$, i.e. the completed local ring of
$X\times_{\mbox{Spec }\mathcal{O}_{K}}\mbox{Spec }W$ at $s$. Then
\[
\mathcal{O}_{X_{U},s}^{\wedge}\simeq W[[T_{1},\dots,T_{n},S_{1},\dots,S_{n}]]/(\prod_{i=h_{1}+1}^{n}T_{i}-\pi,\prod_{i=h_{2}+1}^{n}S_{i}-\pi).\]

Assume that $Y_{1,j_{k}}$ for $k=1,\dots,n-h_{1}$ and $j_{k}\in\{1,\dots,n\}$
distinct are subschemes of $X_{U}$ which contain $s$ as a geometric
point. We can choose the generators $T_{i}$ such that the completed
local ring $\mathcal{O}_{Y_{1,j_{k}},s}^{\wedge}$ is cut out in $\mathcal{O}_{X_{U},s}^{\wedge}$
by the equation $T_{k+h_{1}}=0$. The analogous statement is true
for $Y_{2,j_{k}}$ with $k=1,\dots,n-h_{2}$ and $S_{k+h_{2}}=0$.\end{prop}
\begin{proof}
First we prove that $X_{U}$ has pure dimension $2n-1$ by using Deligne's
homogeneity principle. We will follow closely the proof of Proposition
3.4.1 of \cite{T-Y}. The dimension of $\mathcal{O}_{X_{U},s}^{\wedge}$
as $s$ runs over geometric points of $X_{U}$ above $\bar{X}_{U_{0}}^{(0,0)}$
is constant, say it is equal to $m$. Then we claim that $\mathcal{O}_{X_{U},s}^{\wedge}$
has dimension $m$ for every closed geometric point of $X_{U}$. Indeed,
assume the subset of $X_{U}$ where $\mathcal{O}_{X_{U},s}^{\wedge}$
has dimension different from $m$ is non-empty. Then this subset is
closed, so its projection to $X_{U_{0}}$ is also closed and so it
must contain some $\bar{X}_{U_{0}}^{(h_{1},h_{2})}$ (since the dimension
of $\mathcal{O}_{X_{U},s}^{\wedge}$ only depends on the stratum of
$X_{U_{0}}$ that $s$ is above). By Lemma \ref{closure of stratum},
the closure of $\bar{X}_{U_{0}}^{(h_{1},h_{2})}$ contains $\bar{X}_{U_{0}}^{(0,0)}$,
which is a contradiction. Thus, $X_{U}$ has pure dimension $m$ and
by Lemma \ref{dimension}, $m=2n-1$.

The completed local ring $ $$\mathcal{O}_{X_{U},s}^{\wedge}$ is
the universal deformation ring for tuples $(A,\lambda,i,\bar{\eta}^{p},\mathcal{C}_{1},\mathcal{C}_{2},\alpha_{i})$
deforming $(A_{s},\lambda_{s},i_{s},\bar{\eta}_{s}^{p},\mathcal{C}_{1,s},\mathcal{C}_{2,s},\alpha_{i,s})$.
Deforming the abelian variety $A_{s}$ is the same as deforming its
$p$-divisible group $A_{s}[p^{\infty}]$ by Serre-Tate and $A_{s}[p^{\infty}]=A_{s}[u^{\infty}]\times A_{s}[(u^{c})^{\infty}]$.
The polarization $\lambda$ together with $A[u^{\infty}]$ determine
$A[(u^{c})^{\infty}]$, so it suffices to deform $A_{s}[u^{\infty}]$
as an $\mathcal{O}_{F}$-module together with the level structure.
At primes other than $\mathfrak{p}_{1}$ and $\mathfrak{p}_{2}$,
the $p$-divisible group is etale, so the deformation is uniquely
determined. Moreover, $A[(\mathfrak{p}_{1}\mathfrak{p}_{2})^{\infty}]$
decomposes as $A[\mathfrak{p}_{1}^{\infty}]\times A[\mathfrak{p}_{2}^{\infty}]$
(because $\mathcal{O}_{F}\otimes_{\mathcal{O}_{F'}}\mathcal{O}_{F'_{\mathfrak{p}_{1}\mathfrak{p}_{2}}}\simeq\mathcal{O}_{F,\mathfrak{p}_{1}}\times\mathcal{O}_{F,\mathfrak{p}_{2}})$,
so it suffices to consider deformations of the chains\[
\mathcal{C}_{i,s}:\mathcal{G}_{i,s}=\mathcal{G}_{i,0}\to\mathcal{G}_{i,1}\to\dots\to\mathcal{G}_{i,n}=\mathcal{G}_{i,s}/\mathcal{G}_{i,s}[\mathfrak{p}_{i}]\]
for $i=1,2$ separately. 

Let $\mathcal{G}\simeq\Sigma\times(K/\mathcal{O}_{K})^{h}$ be a $p$-divisible
$\mathcal{O}_{K}$-module over $\bar{\mathbb{F}}$ of dimension one
and total height $n$. Let \[
\mathcal{C}:\mathcal{G}=\mathcal{G}_{0}\to\mathcal{G}_{1}\to\dots\to\mathcal{G}_{n}=\mathcal{G}/\mathcal{G}[\pi]\]
be a chain of isogenies of degree $\#\mathbb{F}$. Since we are working
over $\bar{\mathbb{F}},$ the chain $\mathcal{C}$ splits into a formal
part and an etale part. Let $\mathcal{C}^{0}$ be the chain obtained
from $\mathcal{C}$ by restricting it to the formal part: \[
\tilde{\Sigma}\to\tilde{\Sigma}_{1}\to\dots\to\tilde{\Sigma}_{n}=\tilde{\Sigma}/\tilde{\Sigma}[\pi].\]
Let $J\subseteq\{1,\dots,n\}$ be the subset of indices $j$ for which
$\mathcal{G}_{j-1}\to\mathcal{G}_{j}$ has connected kernel. (The
cardinality of $J$ is $n-h$.) Also assume that the chain $\mathcal{C}^{\mathrm{et}}$
consists of \[
\mathcal{G}_{j}^{\mathrm{et}}=(K/\pi^{-1}\mathcal{O}_{K})^{j}\oplus(K/\mathcal{O}_{K})^{h-j}\]
for all $j\in J$ with the obvious isogenies between them. 

We claim that the universal deformation rings of $\mathcal{C}$ is
isomorphic to \[
W[[T_{1},\dots,T_{n}]]/(\prod_{j\in J}T_{j}-\pi).\]
We will follow the proof of Proposition 4.5 of \cite{D}. To see the
claim, we first consider deformations of $\mathcal{G}$ without level
structure. By proposition 4.5 of \cite{D}, the universal deformation
ring of $\Sigma$ is \[
R^{0}\simeq W[[X_{h+1},\dots,X_{h}]]/(X_{h+1}\cdot\dots\cdot X_{n}-\pi).\]
Let $\tilde{\Sigma}$ be the universal deformation of $\Sigma$. By
considering the connected-etale exact sequence, we see that the deformations
of $\mathcal{G}$ are classified by extensions of the form\[
0\to\tilde{\Sigma}\to\tilde{\mathcal{G}}\to(K/\mathcal{O}_{K})^{h}\to0.\]
Thus, the universal deformations of $\mathcal{G}$ are classified
by elements of $\mbox{Hom}(T\mathcal{G},\tilde{\Sigma})$, where $T\mathcal{G}$
is the Tate module of $\mathcal{G}$. The latter ring is non-canonically
isomorphic to \[
R\simeq W[[X_{1},\dots,X_{n}]]/(\prod_{j\in J}X_{j}-\pi).\]

Let $S$ be the universal deformation ring for deformations of the
chain $\mathcal{C}$ and $S^{0}$ be the universal deformation ring
for the chain $\mathcal{C}^{0}$. Let \[
\mathcal{\tilde{C}}:\mathcal{\tilde{G}}=\mathcal{\tilde{G}}_{0}\to\mathcal{\tilde{G}}_{1}\to\dots\to\mathcal{\tilde{G}}_{n}=\tilde{\mathcal{G}}/\mathcal{\tilde{G}}[\pi]\]
be the universal deformation of $\mathcal{C}$ which corresponds when
restricted to the formal part to the universal chain \[
\tilde{\Sigma}\to\tilde{\Sigma}_{1}\to\dots\to\tilde{\Sigma}_{n}=\tilde{\Sigma}/\tilde{\Sigma}[\pi].\]
Each deformation $\tilde{\mathcal{G}}_{j}$ of $\mathcal{G}_{j}$
is defined by a connected-etale exact sequence \[
0\to\tilde{\Sigma}_{j}\to\tilde{\mathcal{G}}_{j}\to(K/\mathcal{O}_{K})^{h}\to0,\]
so by an element $f_{j}\in\mbox{Hom}(T\mathcal{G}_{j},\tilde{\Sigma}_{j})$.
We will explore the compatibilities between the $\mbox{Hom}(T\mathcal{G}_{j},\tilde{\Sigma}_{j})$
as $j$ ranges from $0$ to $n$. If $j\in J$ then $\tilde{\mathcal{G}}_{j-1}\to\tilde{\mathcal{G}_{j}}$
has connected kernel, so $T\mathcal{G}_{j-1}\simeq T\mathcal{G}_{j}$.
The isogeny $\tilde{\Sigma}_{j-1}\to\tilde{\Sigma}_{j}$ determines
a map $\mbox{Hom}(T\mathcal{G}_{j-1},\tilde{\Sigma}_{j-1})\to\mbox{Hom}(T\mathcal{G}_{j},\tilde{\Sigma}_{j})$,
which determines the extension $\tilde{\mathcal{G}_{j}}$. Thus, in
order to know the extension classes of $\tilde{\mathcal{G}_{j}}$
it suffices to focus on the case $j\not\in J$. 

Let $(e_{j})_{j\in J}$ be a basis of $\mathcal{O}_{K}^{h}$, which
we identify with $T\mathcal{G}_{j}$ for each $j$. We claim that
it suffices to know $f_{j}(e_{j})\in\tilde{\Sigma}_{j}$ for each
$j\not\in J$. Indeed, if $j\not\in J$ then we know that $\tilde{\Sigma}_{j-1}\simeq\tilde{\Sigma}_{j}$
and we also have a map $T\mathcal{G}_{j-1}\to T\mathcal{G}_{j}$ sending
\[
e_{j'}\mapsto e_{j'}\mbox{ for }j'\not=j\mbox{ and }e_{j}\mapsto\pi e_{j}.\]
Thus, for $i\not=j$ we can identify $f_{j-1}(e_{i})\in\tilde{\Sigma}_{j-1}$
with $f_{j}(e_{i})\in\tilde{\Sigma}_{j}.$ Hence if we know $f_{j}(e_{j})$
then we also know $f_{j'}(e_{j})$ for all $j'>j$. Thus we know $f_{n}(e_{j})$,
but recall that $f_{n}$ corresponds to the extension\[
0\to\tilde{\Sigma}/\tilde{\Sigma}[\pi]\to\tilde{\mathcal{G}}/\mathcal{\tilde{G}}[\pi]\to(K/\pi^{-1}\mathcal{O}_{K})^{h}\to0,\]
which is isomorphic to the extension \[
0\to\tilde{\Sigma}\to\tilde{\mathcal{G}}\to(K/\mathcal{O}_{K})^{h}\to0.\]
Therefore we also know $f_{0}(e_{j})$ and by extension all $f_{j'}(e_{j})$
for $j'<j$. This proves the claim that the only parameters needed
to construct all the extensions $\tilde{\mathcal{G}_{j}}$ are the
elements $f_{j}(e_{j})\in\tilde{\Sigma}_{j}$ for all $j\not\in J$. 

We have a map $S^{0}\otimes_{R^{0}}R\to S$ induced by restricting
the Iwahori level structure to the formal part. From the discussion
above, we see that this map is finite and that $S$ is obtained from
$S^{0}\otimes_{R^{0}}R$ by adjoining for each $j\in J$ a root $T_{j}$
of \[
f(T_{j})=X_{j}\]
in $\tilde{\Sigma}$, where $f:\tilde{\Sigma}\to\tilde{\Sigma}$ is
the composite of the isogenies $\tilde{\Sigma}_{j}\to\tilde{\Sigma}_{j+1}\to\dots\to\tilde{\Sigma}_{n}$.
If we quotient $S$ by all the $T_{j}$ for $j\not\in J$, we are
left only with deformations of the chain $\mathcal{C}^{0}$, since
all of the connected-etale exact sequences will split. Thus $S/(T_{j})_{j\not\in J}\simeq S^{0}$. 

Now, the formal part $\tilde{\mathcal{C}}^{0}$ can be written as
a chain \[
\tilde{\Sigma}=\tilde{\Sigma}_{0}\to\dots\to\tilde{\Sigma}_{j}\to\dots\to\tilde{\Sigma}/\tilde{\Sigma}[\pi]\]
of length $n-h$. Choose bases $e_{j}$ for $\mbox{Lie }\mathcal{G}_{j}$
over $S^{0}$ as $j$ runs over $J$, such that \[
e_{n}=e_{j}\mbox{ for the largest }j\in J\]
maps to\[
e_{0}=e_{j}\mbox{ for the smallest }j\in J\]
under the isomorphism $\mathcal{G}_{n}=\mathcal{G}_{0}/\mathcal{G}_{0}[\pi]\stackrel{\sim}{\to}\mathcal{G}_{0}$
induced by $\pi$. Let $T_{j}\in S^{0}$ represent the linear map
$\mbox{Lie }\tilde{\Sigma}_{j'}\to\mbox{Lie }\tilde{\Sigma}_{j}$,
where $j'$ is the largest element of $J$ for which $j'<j$. Then
\[
\prod_{j\in J}T_{j}=\pi.\]
Moreover, $S^{0}/(T_{j})_{j\in J}=\bar{\mathbb{F}}$ by Lemma \ref{rigidity lemma}.
(See also the proof of Proposition 3.4 of \cite{T-Y}.) Hence we have
a surjection \[
W[[T_{1},\dots,T_{n}]]/(\prod_{j=h_{1}+1}^{n}T_{j}-\pi)\twoheadrightarrow S,\]
which by dimension reasons must be an isomorphism. 

Applying the preceding argument to the chains $\mathcal{C}_{1,s}$
and $\mathcal{C}_{2,s}$, we conclude that \[
\mathcal{O}_{X_{U},s}^{\wedge}\simeq W[[T_{1},\dots,T_{n},S_{1},\dots,S_{n}]]/(\prod_{i=h_{1}+1}^{n}T_{i}-\pi,\prod_{i=h_{2}+1}^{n}S_{i}-\pi).\]
Moreover, the closed subvariety $Y_{1,j_{k}}$ of $X_{U}$ is exactly
the locus where $\mathcal{G}_{j_{k}-1}\to\mathcal{G}_{j_{k}}$ has
connected kernel, so, if $s$ is a geometric point of $Y_{1,j_{k}}$,
then $\mathcal{O}_{Y_{1,j_{k}},s}^{\wedge}$ is cut out in $\mathcal{O}_{X_{U},s}^{\wedge}$
by the equation $T_{k+h_{1}}=0$. (Indeed, by our choice of the parameters
$T_{k+h_{1}}$ with $1\leq k\leq n-h_{1}$, the condition that $\mathcal{G}_{1,j_{k}-1}\to\mathcal{G}_{1,j_{k}}$
has connected kernel is equivalent to $T_{k+h_{1}}=0$.)
\end{proof}
For $S,T\subseteq\{1,\dots,n\}$ non-empty let \[
Y_{U,S,T}=\left(\bigcap_{i\in S}Y_{1,i}\right)\cap\left(\bigcap_{j\in T}Y_{2,j}\right).\]
Then $Y_{U,S,T}$ is smooth over $\mbox{Spec }\mathbb{F}$ of pure
dimension $2n-\#S-\#T$ (we can check smoothness on completed local
rings) and it is also proper over $\mbox{Spec }\mathbb{F}$, since
$Y_{U,S,T}\hookrightarrow\bar{X}_{U}$ is a closed immersion and $\bar{X}_{U}$
is proper over $\mbox{Spec }\mathbb{F}$. We also define \[
Y_{U,S,T}^{0}=Y_{U,S,T}\backslash\left(\left(\bigcup_{\substack{S'\supsetneq S}
}Y_{U,S',T}\right)\cup\left(\bigcup_{\substack{T'\supsetneq T}
}Y_{U,S,T'}\right)\right).\]
Note that the inverse image of $\bar{X}_{U}^{(h_{1},h_{2})}$ with
respect to the finite flat map $\bar{X}_{U}\to\bar{X}_{U_{0}}$ is
\[
\bigcup_{\substack{\#S=n-h_{1}\\
\#T=n-h_{2}}
}Y_{U,S,T}^{0}.\]

\begin{lem}
\label{locally etale over}The Shimura variety $X_{U}$ is locally
etale over \[
X_{r,s}=\mbox{Spec }\mathcal{O}_{K}[X_{1},\dots,X_{n},Y_{1},\dots Y_{n}]/(\prod_{i=1}^{r}X_{i}-\pi,\prod_{j-1}^{s}Y_{j}-\pi)\]
with $1\leq r,s\leq n$.\end{lem}
\begin{proof}
Let $x$ be a closed point of $X_{U}$. The completion of the strict
henselization of $X_{U}$ at $x$ $\mathcal{O}_{X_{U},x}^{\wedge}$
is isomorphic to \[
\mathcal{O}_{r,s}=W[[X_{1},\dots,X_{n},Y_{1}\dots,Y_{n}]]/(\prod_{i=1}^{r}X_{i}-\pi,\prod_{j=1}^{s}Y_{j}-\pi)\]
for certain $1\leq r,s\leq n$. We will show that there is an open
affine neighbourhood $U$ of $x$ in $X$ such that $U$ is etale
over $X_{r,s}$. Note that there are local equations $T_{i}=0$ with
$1\leq i\leq r$ and $S_{j}=0$ with $1\leq j\leq s$ which define
the closed subschemes $Y_{1,i}$ with $1\leq i\leq r$ and $Y_{2,j}$
with $1\leq j\leq s$ passing through $x$. Moreover, the parameters
$T_{i}$ and $S_{j}$ satisfy \[
\prod_{i=1}^{r}T_{i}=u\pi\mbox{ and }\prod_{j=1}^{s}S_{i}=u'\pi\]
with $u$ and $u'$ units in the local ring $\mathcal{O}_{X_{U},x}$.
We will explain why this is the case for the $T_{i}$. In the completion
of the strict henselization $\mathcal{O}_{X_{U},x}^{\wedge}$ both
$T_{i}$ and $X_{i}$ cut out the completion of the strict henselization
$\mathcal{O}_{Y_{1,i},x}^{\wedge}$, which means that $T_{i}$ and
$X_{i}$ differ by a unit. Taking the product of the $T_{i}$ we find
that $\prod_{i=1}^{r}T_{i}=u\pi$ for $u\in\mathcal{O}_{X_{U},x}^{\wedge}$
a unit in the completion of the strict henselization of the local
ring. At the same time, in an open neighborhood of $x$, the special
fiber of $X$ is a union of the divisors corresponding to $T_{i}=0$
for $1\leq i\leq r$, so that $\prod_{i=1}^{r}T_{i}$ belongs to the
ideal of $\mathcal{O}_{X_{U},x}$ generated by $\pi$. We conclude
that $u$ is actually a unit in the local ring $\mathcal{O}_{X_{U},x}$,
not only in $\mathcal{O}_{X_{U},x}^{\wedge}$. In a neighborhood of
$x$, we can change one of the $T_{i}$ by $u^{-1}$ and one of the
$S_{i}$ by $(u')^{-1}$ to ensure that\[
\]
\[
\prod_{i=1}^{r}T_{i}=\pi\mbox{ and }\prod_{j=1}^{s}S_{i}=\pi.\]

We will now adapt the argument used in the proof of Proposition 4.8
of \cite{Y} to our situation. We first construct an unramified morphism
$f$ from a neighborhood of $x$ in $X_{U}$ to $\mbox{Spec }\mathcal{O}_{K}[X_{1},\dots,X_{n},Y_{1},\dots Y_{n}]$.
We can do this simply by sending the $X_{i}$ to the $T_{i}$ and
the $Y_{j}$ to the $S_{j}$. Then $f$ will be formally unramified
at the point $x$. By \cite{EGA4} 18.4.7 we see that when restricted
to an open affine neighbourhood $\mbox{Spec }A$ of $x$ in $X$,
$f|_{\mathrm{Spec}A}$ can be decomposed as a closed immersion $\mbox{Spec }A\to\mbox{Spec }B$
followed by an etale morphism $\mbox{Spec }B\to\mbox{Spec }\mathcal{O}_{K}[X_{1},\dots X_{n},Y_{1},\dots,Y_{n}]$.
The closed immersion translates into the fact that $A\simeq B/I$
for some ideal $I$ of $B$. The inverse image of $I$ in $W[X_{1},\dots,X_{n},Y_{1},\dots Y_{n}]$
is an ideal $J$ which contains $\prod_{i=1}^{r}X_{i}-\pi$ and $\prod_{j=1}^{s}Y_{j}-\pi$.
The morphism $f$ factors through the morphism $g:\mbox{Spec }A\to\mbox{Spec }\mathcal{O}_{K}[X_{1},\dots,X_{n},Y_{1},\dots,Y_{n}]/J$
which is etale. Moreover, $J$ is actually generated by $\prod_{i=1}^{r}X_{i}-\pi$
and $\prod_{j=1}^{s}Y_{j}-\pi$, since $g$ induces an isomorphism
on completed strict local rings \[
W[[X_{1},\dots,X_{n},Y_{1},\dots,Y_{n}]]/J\stackrel{\sim}{\to}\mathcal{O}_{r,s}.\]
This completes the proof of the lemma. 
\end{proof}
Let $\mathcal{A}_{U}$ be the universal abelian variety over the integral
model $X_{U}$. Recall that $\xi$ was an irreducible representation
of $G$ over $\bar{\mathbb{Q}}_{l}$. The sheaf $\mathcal{L}_{\xi}$
extends to a lisse sheaf on the integral models $X_{U_{0}}$ and $X_{U}$.
Also, $a_{\xi}\in\mbox{End}(\mathcal{A}_{U}^{m_{\xi}}/X_{U})\otimes_{\mathbb{Z}}\mathbb{Q}$
extends as an etale morphism on $\mathcal{A}_{U}^{m_{\xi}}$ over
the integral model. We have \[
H^{j}(X_{U}\times_{F'}\bar{F}'_{\mathfrak{p}},\mathcal{L}_{\xi})\simeq a_{\xi}H^{j+m_{\xi}}(\mathcal{A}_{U}^{m_{\xi}}\times_{F'}\bar{F}_{\mathfrak{p}}',\bar{\mathbb{Q}}_{l}(t_{\xi}))\]
and we can compute the latter via the nearby cycles $R\psi\bar{\mathbb{Q}}_{l}$
on $\mathcal{A}_{U}^{m_{\xi}}$ over the integral model of $X_{U}$.
Note that $\mathcal{A}_{U}^{m_{\xi}}$ is smooth over $X_{U}$, so
$\mathcal{A}_{U}^{m_{\xi}}$ is locally etale over \[
X_{r,s,m}=\mbox{Spec }\mathcal{O}_{K}[X_{1},\dots,X_{n},Y_{1},\dots Y_{n},Z_{1},\dots,Z_{m}]/(\prod_{j=1}^{r}X_{i_{j}}-\pi,\prod_{j=1}^{s}Y_{i_{j}}-\pi)\]
for some non-negative integer $m$.

\section{Sheaves of nearby cycles}

Let $K/\mathbb{Q}_{p}$ be finite with ring of integers $\mathcal{O}_{K}$
which has uniformiser $\pi$ and residue field $\mathbb{F}$. Let
$I_{K}=\mbox{Gal}(\bar{K}/K^{\mbox{ur}})\subset G_{K}=\mbox{Gal}(\bar{K}/K)$
be the inertia subgroup of $K$. Let $\Lambda$ be either one of $\mathbb{Z}/l^{r}\mathbb{Z}$,
$\mathbb{Z}_{l}$, $\mathbb{Q}_{l}$ or $\bar{\mathbb{Q}}_{l}$ for
$l\not=p$ prime. Let $X/\mathcal{O}_{K}$ be a scheme such that $X$
is locally etale over \[
\mbox{Spec }\mathcal{O}_{K}[X_{1},\dots,X_{n},Y_{1},\dots Y_{n},Z_{1},\dots,Z_{m}]/(\prod_{j=1}^{r}X_{i_{j}}-\pi,\prod_{j=1}^{s}Y_{i_{j}}-\pi).\]
Let $Y$ be the special fiber of $X$. Assume that $Y$ is a union
of closed subschemes $Y_{1,j}$ with $j\in\{1,\dots,n\}$ which are
cut out locally by one equation and that this equation over $X_{r,s,m}$
corresponds to $X_{j}=0$. Similarly, assume that $Y$ is a union
of closed subschemes $Y_{2,j}$ with $j\in\{1,\dots,n\}$ which are
cut out over $X_{r,s,m}$ by $Y_{j}=0$. 

Let $j:X_{K}\hookrightarrow X$ be the inclusion of the generic fiber
and $i:Y\hookrightarrow X$ be the inclusion of the special fiber.
Let $S=\mbox{Spec }\mathcal{O}_{K},$ with generic point $\eta$ and
closed point $s$. Let $\bar{K}$ be an algebraic closure of $K$,
with ring of integers $\mathcal{O}_{\bar{K}}$. Let $\bar{S}=\mbox{Spec }\mathcal{O}_{\bar{K}}$,
with generic point $\bar{\eta}$ and closed point $\bar{s}$. Let
$\bar{X}=X\times_{S}\bar{S}$ be the base change of $X$ to $\bar{S}$,
with generic fiber $\bar{j}:X_{\bar{\eta}}\hookrightarrow\bar{X}$
and special fiber $\bar{i}:X_{\bar{s}}\hookrightarrow\bar{X}$. The
sheaves of nearby cycles associated to the constant sheaf $\Lambda$
on $X_{\bar{K}}$ are sheaves $R^{k}\psi\Lambda$ on $X_{\bar{s}}$
defined for $k\geq0$ as \[
R^{k}\psi\Lambda=\bar{i}^{*}R^{k}\bar{j}_{*}\Lambda\]
and they have continuous actions of $I_{K}$. 
\begin{prop}
\label{trivial inertia}The action of $I_{K}$ on $R^{k}\psi\Lambda$
is trivial for any $k\geq0$. 
\end{prop}
The proof of this proposition is based on endowing $X$ with a logarithmic
structure, showing that the resulting log scheme is log smooth over
$\mbox{Spec }\mathcal{O}_{K}$ (with the canonical log structure determined
by the special fiber) and then using the explicit computation of the
action of $I_{K}$ on the sheaves of nearby cycles that was done by
Nakayama \cite{Nakayama}.

\subsection{Log structures}
\begin{defn}
A log structure on a scheme $Z$ is a sheaf of monoids $M$ together
with a morphism $\alpha:M\to\mathcal{O}_{Z}$ such that $\alpha$
induces an isomorphism $\alpha^{-1}(\mathcal{O}_{Z}^{*})\simeq\mathcal{O}_{Z}^{*}$. 
\end{defn}
From now on, we will regard $\mathcal{O}_{Z}^{*}$ as a subsheaf of
$M$ via $\alpha^{-1}$ and define $\bar{M}:=M/\mathcal{O}_{Z}^{*}$. 

Given a scheme $Z$ and a closed subscheme $V$ with complement $U$
there is a canonical way to associate to $V$ a log structure. If
$j:U\hookrightarrow X$ is an open immersion, we can simply define
$M=j_{*}((\mathcal{O}_{X}|U)^{*})\cap\mathcal{O}_{X}\to\mathcal{O}_{X}$.
This amounts to formally {}``adjoining'' the sections of $\mathcal{O}_{X}$
which are invertible outside $V$ to the units $\mathcal{O}_{X}^{*}$.
The sheaf $\bar{M}$ will be supported on $V$. 

If $P$ is a monoid, then the scheme $\mbox{Spec }\mathbb{Z}[P]$
has a canonical log structure associated to the natural map $P\to\mathbb{Z}[P]$.
A chart for a log structure on $Z$ is given by a monoid $P$ and
a map $Z\to\mbox{Spec }\mathbb{Z}[P]$ such that the log structure
on $Z$ is pulled back from the canonical log structure on $\mbox{Spec }\mathbb{Z}[P]$.
A chart for a morphism of log schemes $Z_{1}\to Z_{2}$ is a triple
of maps $Z_{1}\to\mbox{Spec }\mathbb{Z}[Q],$ $Z_{2}\to\mbox{Spec }\mathbb{Z}[P]$
and $P\to Q$ such that the first two maps are charts for the log
structures on $Z_{1}$ and $Z_{2}$ and such that the obvious diagram
is commutative.
\begin{defn}
A scheme endowed with a log structure is a log scheme. A morphism
of log schemes $(Z_{1},M_{1})\to(Z_{2},M_{2})$ consists of a pair
$(f,h)$ where $f:Z_{1}\to Z_{2}$ is a morphism of schemes and $h:f^{*}M_{2}\to M_{1}$
is a morphism of sheaves of monoids.
\end{defn}
For more background on log schemes, the reader should consult \cite{I1,Kato}. 

We endow $S=\mbox{Spec }\mathcal{O}_{K}$ with the log structure given
by $N=j_{*}(K^{*})\cap\mathcal{O}_{K}\hookrightarrow\mathcal{O}_{K}$.
The sheaf $\bar{N}$ is trivial outside the closed point and is isomorphic
to a copy of $\mathbb{N}$ over the closed point. Another way to describe
the log structure on $S$ is by pullback of the canonical log structure
via the map \[
S\to\mbox{Spec }\mathbb{Z}[\mathbb{N}]\]
where $1\mapsto\pi\in\mathcal{O}_{K}$. 

We endow $X$ with the log structure given by $M=j_{*}(\mathcal{O}_{X_{K}}^{*})\cap\mathcal{O}_{X}\hookrightarrow\mathcal{O}_{X}$.
It is easy to check that the only sections of $\mathcal{O}_{X}$ which
are invertible outside the special fiber, but not invertible globally
are those given locally by the images of the $X_{i}$ for $1\leq i\leq n$
and the $Y_{j}$ for $1\leq j\leq n$ . On etale neighborhoods $U$
of $X\times_{\mathcal{O}_{K}}W$ which are etale over $X_{r,s,m}$
this log structure is given by the chart \[
U\to X_{r,s,m}\to\mbox{Spec }\mathbb{Z}[P_{r,s}]\]
 where \[
P_{r,s}=\mathbb{N}^{r}\oplus\mathbb{N}^{s}/(1,\dots1,0,\dots0)=(0,\dots0,1,\dots1).\]
The map $X_{r,s,m}\to\mbox{Spec }\mathbb{Z}[P_{r,s}]$ can be described
as follows: the element with $1$ only in the $k$th place, $(0,\dots,0,1,0,\dots0)\in P_{r,s}$
maps to $X_{k}$ if $k\leq r$ and to $Y_{k-r}$ if $k\geq r+1$.
Note that the log structure on $X$ is trivial outside the special
fiber, so $X$ is a \textbf{vertical} log scheme.

The map $X\to S$ induces a map of the corresponding log schemes.
Etale locally, this map has a chart subordinate to the map of monoids
$\mathbb{N}\to P_{r,s}$ such that \[
1\mapsto(1,\dots,1,0,\dots,0)=(0,\dots0,1,\dots1)\]
to reflect the relations $X_{1}\dots X_{r}=Y_{1}\dots Y_{s}=\pi$. 
\begin{lem}
The map of log schemes $(X,M)\to(S,N)$ is log smooth. \end{lem}
\begin{proof}
The map of monoids $\mathbb{N}\to P_{r,s}$ induces a map on groups
$\mathbb{Z}\to P_{r,s}^{gp}$, which is injective and has torsion-free
cokernel $\mathbb{Z}^{r+s-2}$ . Since the map of log schemes $(X,M)\to(S,N)$
is given etale locally by charts subordinate to such maps of monoids,
by Theorem 3.5 of \cite{Kato} the map $(X,M)\to(S,N)$ is log smooth.
\end{proof}

\subsection{Nearby cycles and log schemes}

There is a generalization of the functor of nearby cycles to the category
of log schemes.

Recall that $\mathcal{O}_{\bar{K}}$ is the integral closure of $\mathcal{O}_{K}$
in $\bar{K}$ and $\bar{S}=\mbox{Spec }\mathcal{O}_{\bar{K}}$, with
generic point $\bar{\eta}$ and closed point $\bar{s}$. The canonical
log structure associated to the special fiber (given by the inclusion
$\bar{j}_{*}(\bar{K}^{*})\cap\mathcal{O}_{\bar{K}}\hookrightarrow\mathcal{O}_{\bar{K}})$
defines a log scheme $\tilde{S}$ with generic point $\bar{\eta}$
and closed point $\tilde{s}$. Note that $\tilde{s}$ is a log geometric
point of $\tilde{S}$, so it has the same underlying scheme as $\bar{s}$.
The Galois group $G_{K}$ acts on $\tilde{s}$ through its tame quotient.
Let $\tilde{X}=X\times_{S}\tilde{S}$ in the category of log schemes,
with special fiber $X_{\tilde{s}}$ and generic fiber $X_{\bar{\eta}}$.
Note that, in general, the underlying scheme of $X_{\tilde{s}}$ is
not the same at $X_{\bar{s}}$. This is because $X_{\tilde{s}}$ is
the fiber product of $X_{\bar{s}}$ and $\tilde{s}$ in the category
of fine and saturated log schemes and saturation corresponds to normalization,
so it changes the underlying scheme. 

The sheaves of log nearby cycles are sheaves on $X_{\tilde{s}}$ defined
by \[
R^{k}\psi^{\mathrm{log}}\Lambda=\tilde{i}^{*}R^{k}\tilde{j}_{*}\Lambda,\]
where $\tilde{i},\tilde{j}$ are the obvious maps and the direct and
inverse images are taken with respect to the Kummer etale topology.
Theorem 3.2 of \cite{Nakayama} states that when $X/S$ is a log smooth
scheme we have $R^{0}\psi^{\mathrm{log}}\Lambda\cong\Lambda$ and
$R^{p}\psi^{\mathrm{log}}\Lambda=0$ for $p>0$. Let \[
\tilde{\epsilon}:\tilde{X}\to\bar{X},\]
which restricts to $\epsilon:X_{\bar{\eta}}\to X_{\bar{\eta}}$, be
the morphism that simply forgets the log structure. Note that we have
$\bar{j}_{*}\epsilon_{*}=\tilde{\epsilon}_{*}\tilde{j}_{*}$, by commutativity
of the square \[
\xymatrix{X_{\bar{\eta}}\ar[r]^{\tilde{j}}\ar[d]_{\epsilon} & \tilde{X}\ar[d]_{\tilde{\epsilon}}\\
X_{\bar{\eta}}\ar[r]^{\bar{j}} & \bar{X}}
\]
We also have $\bar{i}^{*}\tilde{\epsilon}_{*}\mathcal{F}=\tilde{\epsilon}_{*}\tilde{i}^{*}\mathcal{F}$
for every Kummer etale sheaf $\mathcal{F}$, since the sections of
both sheaves over some etale open $V$ of $X_{\bar{s}}$ are obtained
as the direct limit of $\mathcal{F}(\tilde{U})$, where $U$ runs
over etale neighborhoods of $V$ in $\bar{X}$ and $\tilde{U}$ has
the canonical log structure associated to the special fiber. $ $We
deduce that \[
\bar{i}^{*}\bar{j}_{*}\epsilon_{*}=\tilde{\epsilon}_{*}\tilde{i}^{*}\tilde{j}_{*}\]
so the corresponding derived functors must satisfy a similar relation.
When we write this out, using $R\psi^{\mathrm{log}}\Lambda\cong\Lambda$
by Nakayama's result and $R\epsilon_{*}\Lambda\cong\Lambda$ because
the log structure is vertical and so $\epsilon$ is an isomorphism,
we get \[
R^{k}\psi^{\mathrm{cl}}\Lambda=R^{k}\tilde{\epsilon}_{*}(\Lambda|X_{\tilde{s}}).\]
 Therefore, it suffices to figure out what the sheaves $R^{k}\tilde{\epsilon}_{*}\Lambda$
look like and how $I_{K}$ acts on them, where $\tilde{\epsilon}:X_{\tilde{s}}\to X_{\bar{s}}$.
This has been done in general by Nakayama, Theorem 3.5 of \cite{Nakayama},
thus deriving an SGA 7 I.3.3-type formula for log smooth schemes.
We will describe his argument below and specialize to our particular
case. 
\begin{lem}
$I_{K}$ acts on $R^{p}\epsilon_{*}(\Lambda|X_{\tilde{s}}$) through
its tame quotient. \end{lem}
\begin{proof}
Let $S^{t}=\mbox{Spec }\mathcal{O}_{K^{t}}$ endowed with the canonical
log structure (here $K^{t}\subset\bar{K}$ is the maximal extension
of $K$ which is tamely ramified). The closed point $s^{t}$ with
its induced log structure is a universal Kummer etale cover of $s$
and $I_{K}$ acts on it through its tame quotient $I^{t}$. Moreover,
the projection $\tilde{s}\to s^{t}$ is a limit of universal Kummer
homeomorphisms and it remains so after base change with $X$. (See
Theorem 2.8 of \cite{I1}). Thus, every automorphism of $X_{\tilde{s}}$
comes from a unique automorphism of $X_{s^{t}}$, on which $I_{K}$
acts through $I^{t}$. 
\end{proof}
Now we have the commutative diagram \[
\xymatrix{X_{\tilde{s}}^{log}\ar[d]_{\epsilon}\ar[r]^{\alpha} & X_{\bar{s}}^{log}\ar[d]^{\epsilon}\\
X_{\tilde{s}}^{cl}\ar[r]^{\beta} & X_{\bar{s}}^{cl}}
,\]
where the objects in the top row are log schemes and the objects in
the bottom row are their underlying schemes. The morphisms labeled
$\epsilon$ are forgetting the log structure and we have $\tilde{\epsilon}=\epsilon\circ\alpha=\beta\circ\epsilon$.
We can use either of these decompositions to compute $R^{k}\tilde{\epsilon}_{*}\Lambda$.
For example, we have $R\tilde{\epsilon}_{*}\Lambda=R\beta_{*}R\epsilon_{*}\Lambda$,
which translates into having a spectral sequence \[
R^{n-k}\beta_{*}R^{k}\epsilon_{*}\Lambda\Rightarrow R^{n}\tilde{\epsilon}_{*}\Lambda.\]
We know that $R^{k}\epsilon_{*}\Lambda=\wedge^{k}\bar{M}_{rel}^{gp}\otimes\Lambda(-k)$,
where \[
\bar{M}_{rel}^{gp}=\mbox{coker }(\bar{N}^{gp}\to\bar{M}^{gp})/torsion.\]
This follows from theorem 2.4 of \cite{Kato-Nakayama}. On the other
hand, at a geometric point $\bar{x}$ of $X_{\bar{s}}^{cl}$, we have
$(\beta_{*}\mathcal{F})_{\bar{x}}\cong\mathcal{F}[E_{\bar{x}}]$ for
a sheaf $\mathcal{F}$ of $\Lambda$-modules on $X_{\tilde{s}}^{cl}$,
where $E_{\bar{x}}$ is the cokernel of the map of log inertia groups
\[
I_{x}\to I_{s}.\]
Indeed, $\beta^{-1}(\bar{x})$ consists of $\#\mbox{coker }(I_{x}\to I_{s})$
points, which follows from the fact that $X_{\tilde{s}}^{cl}$ is
the normalization of $(X_{\bar{s}}\times_{\bar{s}}\tilde{s})^{cl}$.
The higher derived functors $R^{n-k}\beta_{*}\mathcal{F}$ are all
trivial, since $\beta_{*}$ is exact. Therefore, the spectral sequence
becomes\[
\wedge^{k}\bar{M}_{rel,\bar{x}}^{gp}\otimes\Lambda[E_{\bar{x}}]\otimes\Lambda(-k)=(R^{k}\tilde{\epsilon}_{*}\Lambda)_{\bar{x}}.\]
The tame inertia acts on the stalks of these sheaves through $I^{t}\cong I_{\bar{s}}\mapsto\Lambda[I_{\bar{s}}]\to\Lambda[E_{\bar{x}}]$. 

In our particular case, it is easy to compute $R^{k}\tilde{\epsilon}_{*}\Lambda$
globally. Let \[
\hat{\mathbb{Z}}^{'}(1)=\lim_{\substack{\longleftarrow\\
(m,p)=1}
}\mu_{m}.\]
 We have \[
I_{x}=Hom(\bar{M}_{x}^{gp},\hat{\mathbb{Z}}^{'}(1))\]
 and \[
I_{s}=Hom(\bar{N}_{s}^{gp},\hat{\mathbb{Z}}^{'}(1)).\]
The map of inertia groups is induced by the map $\bar{M}_{s}^{gp}\to\bar{M}_{x}^{gp}$,
which is determined by $1\mapsto(1,\dots,1,0\dots,0)$ , where the
first $n$ terms are nonzero. Any homomorphism of $\bar{M}_{s}^{gp}\cong\mathbb{Z}\to\hat{\mathbb{Z}}^{'}(1)$
can be obtained from some homomorphism $\bar{M}_{x}^{gp}\to\hat{\mathbb{Z}}^{'}(1)$.
Thus $E_{\bar{x}}$ is trivial for all log geometric points $\bar{x}$
and $I^{t}$ acts trivially on the stalks of the sheaves of nearby
cycles. Moreover, in our situation we can check that $\beta$ is an
isomorphism, since $(X_{\bar{s}}\times_{\bar{s}}\tilde{s})^{cl}$
is already normal, which follows from the fact that $X_{\bar{s}}$
is reduced, which can be checked etale locally. Thus we have the global
isomorphism \[
R^{k}\tilde{\epsilon}_{*}\Lambda\simeq\wedge^{k}\bar{M}_{rel}^{gp}\otimes\Lambda(-k).\]

The above discussion also allows us to determine the sheaves of nearby
cycles. Indeed, we have $R^{k}\psi\Lambda\simeq\wedge^{k}\bar{M}_{rel}^{gp}\otimes\Lambda(-k)$
and $\bar{M}_{rel}^{gp}$ can be computed explicitly on neighborhoods.
If $U$ is a neighborhood of $X$ with $U$ etale over $X_{r,s}$
then the log structure on $U$ is induced from the log structure on
$X_{r,s}$. Let $I,$$J\subseteq\{1,\dots,n\}$ be sets of indices
with cardinalities $r$ and $s$ respectively, corresponding to sets
of divisors $Y_{1,i}$ and $Y_{2,j}$ which intersect $U$. 
\begin{prop}
\label{localnearbycycles}If $x_{i}\in\mathcal{O}_{X}$ is the image
of $1\in\mathcal{O}_{Y_{1,i}}$ pushed forward under the closed immersion
$a_{1,i}:Y_{1,i}\hookrightarrow X$ and $y_{j}\in\mathcal{O}_{X}$
is the image of $1\in\mathcal{O}_{Y_{2,j}}$ pushed forward under
$a_{2,j}:Y_{2,j}\hookrightarrow X,$ then we have 
\end{prop}
\[
R^{k}\psi\Lambda(k)|_{U}\simeq\wedge^{k}[(\oplus_{i\in I}x_{i}\Lambda/\sum_{i\in I}x_{i})\oplus(\oplus_{j\in J}y_{j}\Lambda/\sum_{j\in J}y_{j})]|_{U}.\]

We can define a global map of sheaves \[
\wedge^{k}[(\oplus_{i=1}^{n}x_{i}\Lambda)\oplus(\oplus_{i=1}^{n}y_{j}\Lambda)]\to\wedge^{k}\bar{M}_{rel}^{gp}\simeq R^{k}\psi\Lambda(k)\]
by sending $x_{i}$ to the generator of the divisor $Y_{1,i}$ and
$y_{j}$ to the generator of the divisor $Y_{2,j}$. The image of
$x_{i}$ in $\mathcal{O}_{X}$ will be a unit over $X\backslash Y_{i,1}$,
so in particular it will be a unit outside the special fiber $Y$
of $X$. Since the log structure on $X$ is the canonical log structure
associated to the special fiber, the image of $x_{i}$ will be a nontrivial
element of $\bar{M}_{rel}^{gp}$. The same holds true for the $y_{j}$.
We see from the local computation in proposition \ref{localnearbycycles}
that the above map of sheaves is surjective and that the kernel is
generated by the two sections $\sum_{i=1}^{n}x_{i}$ and $\sum_{j=1}^{n}y_{j}.$ 
\begin{cor}
\label{globalnearby}There is a global isomorphism
\end{cor}
\[
\bigwedge^{k}[(\oplus_{i=1}^{n}x_{i}\Lambda/\sum_{i=1}^{n}x_{i})\oplus(\oplus_{j=1}^{n}y_{j}\Lambda/\sum_{j=1}^{n}y_{j})]\simeq R^{k}\psi\Lambda(k).\]

Let $\mathcal{L}_{1}=\oplus x_{i}\Lambda/\sum_{i=1}^{n}x_{i}$ and
$\mathcal{L}_{2}=\oplus_{j=1}^{n}y_{j}\Lambda/\sum_{j=1}^{n}y_{j}$.
From the above corollary, we see that $R^{k}\psi\Lambda(k)$ can be
decomposed as $\sum_{l=0}^{k}\wedge^{l}\mathcal{L}_{1}\otimes\wedge^{k-l}\mathcal{L}_{2}$.
If $X$ was actually a product of semistable schemes, $X=X_{1}\times_{S}X_{2}$,
then the sheaves $\wedge^{l}\mathcal{L}_{1}$ and $\wedge^{k-l}\mathcal{L}_{2}$
would have an interpretation as pullbacks of the nearby cycles sheaves
$R^{l}\psi\Lambda$ and $R^{k-l}\psi\Lambda$ associated to $X_{1}$
and $X_{2}$ respectively. Corollary \ref{globalnearby} would then
look like a Künneth-type formula computing the sheaves of nearby cycles
for a product of semistable schemes. In fact, in such a situation,
the computation of the sheaves of nearby cycles reflects the stronger
relation between the actual complexes of nearby cycles\[
R\psi\Lambda_{X_{1}\times_{S}X_{2}}\simeq R\psi\Lambda_{X_{1}}\otimes_{s}^{L}R\psi\Lambda_{X_{2}}\]
which takes place in the derived category of constructible sheaves
of $\Lambda$-modules on $(X_{1}\times_{S}X_{2})_{\bar{s}}$. This
result was proven in \cite{I2} for a product of schemes of finite
type. The isomorphism is stated in the case when $\Lambda$ is torsion,
however the analogue morphism for $\Lambda$ a finite extension of
$\mathbb{Z}_{l}$ or $\mathbb{Q}_{l}$ can be defined by passage to
the limit (see the formalism in \cite{E}) and it will still be an
isomorphism. We would like to give here a different proof of this
result in the case of the product of two semistable schemes. We will
use log schemes, specifically Nakayama's computation of log vanishing
cycles for log smooth schemes.

Recall that the scheme $S$ has generic point $\eta$ and closed point
$s$. We will freely use the notations $\bar{S},\tilde{S}$ and $\bar{s},\tilde{s}$,
and also the corresponding notations for a scheme $X$ fixed in the
begining of this subsection. We first need a preliminary result. 
\begin{lem}
\label{flatness of sheaves}Let $X_{1}$ be a strictly semistable
scheme over $S.$ Then the sheaves $R^{k}\psi\Lambda$ are flat over
$\Lambda$. \end{lem}
\begin{proof}
By Proposition 1.1.2.1 of \cite{Saito}, we have an exact sequence
of sheaves on $X_{1,\bar{s}}$ \[
0\to R^{k}\psi\Lambda\to i^{*}R^{k+1}j_{*}\Lambda(1)\to R^{k+1}\psi\Lambda(1)\to0.\]
We will prove by induction on $k$ that $R^{n-k}\psi\Lambda$ is flat
over $\Lambda$. Indeed, $R^{n}\psi\Lambda=0$ so the induction hypothesis
is true for $k=0$. For the induction step, note that we can compute
$i^{*}R^{n-k+1}j_{*}\Lambda$ using log etale cohomology. If we let
$x_{i}$ be a generator of the ideal defining the irreducible component
$Y_{1,i}$ of $X_{1,s}$ then we can endow $X_{1,s}$ with the log
structure given by $\oplus_{i\in I}x_{i}\mathbb{N}$$ $. We know
that $i^{*}R^{n-k+1}j_{*}\Lambda\simeq R^{n-k+1}\epsilon_{*}\Lambda$,
where $\epsilon$ is the morphism which forgets the log structure.
We can compute $R^{n-k+1}\epsilon_{*}\Lambda$ as above to get \[
i^{*}R^{n-k+1}j_{*}\simeq\wedge^{n-k+1}(\oplus_{i\in I}x_{i}\Lambda)\otimes_{\Lambda}\Lambda(-n+k-1),\]
which is free over $\Lambda$. In the short exact sequence\[
0\to R^{n-k}\psi\Lambda\to i^{*}R^{n-k+1}j_{*}\Lambda(1)\to R^{n-k+1}\psi\Lambda(1)\to0\]
 the middle term is free, the right term is flat by the induction
hypothesis, so the left term must be flat as well. \end{proof}
\begin{prop}
\label{product}Let $X_{1}$ and $X_{2}$ be strictly semistable schemes
over $S$. Then we have the following equality in the derived category
of constructible $\Lambda[I_{s}]$-modules on $(X_{1}\times_{S}X_{2})_{s}$:
\[
R\psi(\Lambda_{X_{1,\eta}})\otimes_{s}^{L}R\psi(\Lambda_{X_{2,\eta}})\simeq R\psi(\Lambda_{(X_{1}\times_{S}X_{2})_{\eta}}),\]
where the external tensor product of a complexes is obtained by taking
$pr_{1}^{*}\otimes pr_{2}^{*}$ and where the superscript $L$ refers
to left derived tensor product. \end{prop}
\begin{proof}
We've seen from the above discussion that in the case of a log smooth
scheme with vertical log structure the complex of vanishing cycles
depends only on the special fiber endowed with the canonical log structure.
In other words, for $i=1,2$, we have $R\psi\Lambda_{X_{i,\eta}}\simeq R\tilde{\epsilon}_{i,*}\Lambda_{X_{i,s}}$
as complexes on $X_{i,s}$, where $\tilde{\epsilon}_{i}:\tilde{X}_{i,\tilde{s}}\to\bar{X}_{i,\bar{s}}$
is the identity morphism on the underlying schemes and forgets the
log structure. Analogously, we also have $R\psi\Lambda_{(X_{1}\times_{S}X_{2})_{\eta}}=R\tilde{\epsilon}_{*}\Lambda$,
where \[
\tilde{\epsilon}:(\tilde{X}_{1}\times_{\tilde{S}}\tilde{X}_{2})_{\tilde{s}}\to(\bar{X}_{1}\times_{\bar{S}}\bar{X}_{2})_{\bar{s}}\]
is the morphism which forgets the log structure. (Here we've used
the fact that the fiber product of log smooth schemes with vertical
log structure is log smooth with vertical log structure and that the
underlying scheme of the fiber product of log schemes $\tilde{X}_{1}\times_{\tilde{S}}\tilde{X}_{2}$
is just $\bar{X}_{1}\times_{\bar{S}}\bar{X}_{2}).$ Therefore, it
suffices to prove that we have an isomorphism \[
R\tilde{\epsilon}_{*}\Lambda_{(\tilde{X}_{1}\times_{\tilde{S}}\tilde{X}_{2})_{\tilde{s}}}\simeq R\tilde{\epsilon}_{1,*}\Lambda_{\tilde{X}_{1,\tilde{s}}}\otimes_{\bar{s}}^{L}R\tilde{\epsilon}_{2,*}\Lambda_{\tilde{X}_{2,\tilde{s}}}\]
in the derived category of constructible sheaves of $\Lambda[I_{s}]$-modules
on $(\bar{X}_{1}\times_{\bar{S}}\bar{X}_{2})_{\bar{s}}$. 

It is enough to show that the Künneth map \[
\mathcal{C}=R\tilde{\epsilon}_{1,*}\Lambda_{\tilde{X}_{1},s}\otimes_{\bar{s}}^{L}R\tilde{\epsilon}_{2,*}\Lambda_{\tilde{X}_{2,\tilde{s}}}\to R\tilde{\epsilon}_{*}\Lambda_{(\tilde{X}_{1}\times_{\tilde{S}}\tilde{X}_{2})_{\tilde{s}}}=\mathcal{D},\]
which is defined as in \cite{SGA4} XVII 5.4.1.4, induces an isomorphism
on the cohomology of the two complexes above, for then the map itself
will be a quasi-isomorphism. The cohomology of the product complex
can be computed using a Künneth formula as $H^{n}(\mathcal{C})=\bigoplus_{k=0}^{n}R^{k}\tilde{\epsilon}_{1,*}\Lambda\otimes_{\bar{s}}R^{n-k}\tilde{\epsilon}_{2,*}\Lambda$.
In general, the Künneth formula involves a spectral sequence with
terms \[
E_{2}^{l,n-l}=\sum_{k=0}^{n-l}Tor_{l}^{\Lambda[I_{s}]}(R^{k}\tilde{\epsilon}_{1,*}\Lambda,R^{n-l-k}\tilde{\epsilon}_{2,*}\Lambda)\Rightarrow H^{n}(\mathcal{C}),\]
see \cite{EGA3} XVII 6.5.4.2 for a statement using homology. In our
case the cohomology sheaves $R^{k}\tilde{\epsilon}_{i,*}\Lambda$
are flat $\Lambda$-modules with trivial $I_{s}$-action by Lemmas
\ref{trivial inertia} and \ref{flatness of sheaves}, so for $l>0$
all the $E_{2}^{l,n-l}$ terms vanish. (Alternatively, one can prove
the formula for $H^{n}(\mathcal{C})$ by taking flat resolutions for
both of the factor complexes and using the fact that the cohomology
sheaves of the flat complexes are flat as well.) 

In order to prove that the induced map $H^{n}(\mathcal{C})\to H^{n}(\mathcal{D})$
is an isomorphism, it suffices to check that it induces an isomorphism
on stalks at geometric points. Let $x$ be a geometric point of $X_{1}\times_{S}X_{2}$
above the geometric point $\bar{s}$ of $S$. The point $x$ will
project to geometric points $x_{1}$ and $x_{2}$ of $X_{1}$ and
$X_{2}$. From \cite{I1} it follows that there is an isomorphism
on stalks\[
R^{k}\tilde{\epsilon}_{i,*}\Lambda_{x_{i}}\simeq H^{k}(J_{i},\Lambda)\]
 for $0\leq k\leq n$ and $i=1,2$, where $J_{i}$ is the relative
log inertia group \[
\ker(\pi_{1}^{log}(X_{i},x_{i})\to\pi_{1}^{log}(S,s)).\]
A similar statement holds for the stalks at $x$\[
R^{n}\tilde{\epsilon}_{*}\Lambda_{x}\simeq H^{n}(J,\Lambda),\]
where $J$ is the relative log inertia group $\ker(\pi_{1}^{log}(X,x)\to\pi_{1}^{log}(S,s))$.
Directly from the definition of the log fundamental group we can compute
$J=J_{1}\times J_{2}$. We have the following commutative diagram\[
\xymatrix{H^{n}(\mathcal{C})_{x}\ar[r]\ar^{\cong}[d] & H^{n}(\mathcal{D})_{x}\ar^{\cong}[d]\\
\bigoplus_{k=0}^{n}H^{k}(J_{1},\Lambda)\otimes_{\Lambda}H^{n-k}(J_{2},\Lambda)\ar[r] & H^{n}(J_{1}\times J_{2},\Lambda)}
\]
where the bottom arrow is the Künneth map in group cohomology and
is also an isomorphism. (Again, the Künneth spectral sequence \[
E_{2}^{l,n-l}=\sum_{k=0}^{n-l}Tor_{l}^{\Lambda}(H^{k}(J_{1},\Lambda),H^{n-k}(J_{2},\Lambda))\]
degenerates at $E_{2}$ and all terms outside the vertical line $l=0$
vanish because these cohomology groups are flat $\Lambda$-modules.)
Therefore the top arrow $H^{n}(\mathcal{C})_{x}\to H^{n}(\mathcal{D})_{x}$
has to be an isomorphism for all geometric points $x$ of $X$ which
means it comes from a global isomorphism of sheaves on $X$. 
\end{proof}

\section{The monodromy filtration}

\subsection{Overview of the semistable case}

In this section, we will explain a way of writing down explicitly
the monodromy filtration on the complex of nearby cycles $R\psi\Lambda$,
in the case of a semistable scheme. Our exposition will follow that
of \cite{Saito}, which constructs the monodromy filtration using
perverse sheaves. We let $\Lambda=\mathbb{Z}/l^{r}\mathbb{Z},\mathbb{Z}_{l},\mathbb{Q}_{l}$
or $\bar{\mathbb{Q}}_{l}.$ 

Let $X_{1}/\mathcal{O}_{K}$ be a strictly semistable scheme of relative
dimension $n-1$ with generic fiber $X_{1,\eta}$ and special fiber
$Y_{1}=X_{1,s}$. Let $R\psi\Lambda=\bar{i}^{*}R\bar{j}_{*}\Lambda$
be the complex of nearby cycles over $Y_{1,\bar{\mathbb{F}}}$. Let
$D_{1},\dots,D_{m}$ be the irreducible components of $Y_{1}$ and
for each index set $I\subseteq\{1,\dots,m\}$ let $Y_{I}=\cap_{i\in I}D_{i}$
and $a_{I}:Y_{I}\to Y_{1}$ be the immersion. The scheme $Y_{I}$
is smooth of dimension $n-1-k$ if $\#I=k+1$. For all $0\leq k\leq m-1$
we set \[
Y_{1}^{(k)}=\bigsqcup_{I\subseteq\{1,\dots,m\},\#I=k+1}Y_{I}\]
and let $a_{k}:Y_{1}^{(k)}\to Y_{1}$ be the projection. We identify
$a_{k*}\Lambda=\wedge^{k+1}a_{0*}\Lambda$. 

We will work in the derived category of bounded complexes of constructible
sheaves of $\Lambda$-modules on $Y_{1,\bar{\mathbb{F}}}$. We will
denote this category by $D_{c}^{b}(Y_{1,\bar{\mathbb{F}}},\Lambda)$. 

Let $\partial[\pi]$ be the boundary of $\pi$ with respect to the
Kummer sequence obtained by applying $i^{*}Rj_{*}$ to the exact sequence
of etale sheaves on $X_{1,\eta}$ \[
0\to\Lambda(1)\to\mathcal{O}_{X_{1,\eta}}^{*}\to\mathcal{O}_{X_{1,\eta}}^{*}\to0\]
for $\Lambda=\mathbb{Z}/l^{r}\mathbb{Z}$. Taking an inverse limit
over $r$ and tensoring we get an element $\partial[\pi]\in i^{*}R^{1}j_{*}\Lambda(1)$
for $\Lambda=\mathbb{Q}_{l}$ or $\bar{\mathbb{Q}}_{l}$. Let $\theta:\Lambda_{Y_{1}}\to i^{*}R^{1}j_{*}\Lambda(1)$
be the map sending $1$ to $\partial[\pi]$. Let $\delta:\Lambda_{Y_{1}}\to a_{0*}\Lambda$
be the canonical map. The following result appears as Corollary 1.3
of \cite{Saito}. 
\begin{prop}
\label{isom exact seq}1. There is an isomorphism of exact sequences

\[
\xymatrix{\Lambda_{Y_{1}}\ar[r]\sp-{\delta}\ar[d] & a_{0*}\Lambda\ar[r]\sp-{\delta\wedge}\ar[d] & \dots\ar[r]\sp-{\delta\wedge}\ar[d] & a_{n-1*}\Lambda\ar[r]\ar[d] & 0\\
\Lambda_{Y_{1}}\ar[r]\sp-{\theta} & i^{*}R^{1}j_{*}\Lambda(1)\ar[r]\sp-{\theta\cup} & \dots\ar[r]\sp-{\theta\cup} & i^{*}R^{n}j_{*}\Lambda(n)\ar[r] & 0}
,\]
 where the first vertical arrow is the identity and all the other
vertical arrows are isomorphisms. 

2. For $k\geq0$ we have an exact sequence \[
0\to R^{k}\psi\Lambda\to i^{*}R^{k+1}j_{*}\Lambda(1)\to\dots\to i^{*}R^{n}j_{*}\Lambda(n-k)\to0,\]
where all the horizontal maps are induced from $\theta\cup$. \end{prop}
\begin{note*}
1. The vertical isomorphisms in the first part of Proposition \ref{isom exact seq}
come from the Kummer sequence corresponding to each of the $D_{i}$
for $i=1,\dots,m$. The maps $\theta_{i}:\Lambda_{D_{i}}\to i^{*}R^{1}j_{*}\Lambda(1)$
are defined by sending $1$ to $\partial[\pi_{i}]$, where $\pi_{i}$
is the generator of the ideal defining $D_{i}$ and $\partial$ is
the connecting differential in the Kummer sequence. The isomorphism
$a_{0*}\Lambda\stackrel{\sim}{\to}i^{*}R^{1}j_{*}\Lambda(1)$ is the
direct sum of the $\theta_{i}$ for $i=1,\dots,m$. 

2. Putting together the two isomorphisms, we get a quasi-isomorphism
of complexes \begin{equation}
R^{k}\psi\Lambda(k)[-k]\stackrel{\sim}{\to}[a_{k*}\Lambda\to\dots\to a_{n-1*}\Lambda\to0],\label{eq:resolution}\end{equation}
where $R^{k}\psi\Lambda(k)$ is put in degree $k$ and $a_{n-1*}\Lambda$
is put in degree $n-1$. \end{note*}
\begin{lem}
\label{lem:perversity}The complex $a_{l*}\Lambda[-l]$ is a $-(n-1)$-shifted
perverse sheaf for all $0\leq l\leq n-1$ and so is the complex $R^{k}\psi\Lambda(k)[-k]$
for all $0\leq k\leq n-1$.\end{lem}
\begin{proof}
Since $Y_{1}^{(l)}$ is smooth of dimension $n-1-l$, we know that
$\Lambda[-l]$ is a $-(n-1)$-shifted perverse sheaf on $Y_{1}^{(l)}$.
The map $a_{l}:Y^{(l)}\to Y$ is finite and since the direct image
for a finite map is exact for the perverse $t$-structure, we deduce
that $a_{l*}\Lambda[-l]$ is a $-(n-1)$-shifted perverse sheaf on
$Y$. This is true for each $0\leq l\leq n-1$. The complex $R^{k}\psi\Lambda(k)[-k]$
is a successive extension of terms of the form $a_{l*}\Lambda[-l]$
(as objects in the triangulated category $D_{c}^{b}(Y_{\bar{\mathbb{F}}},\Lambda)$.
Because the category of $-(n-1)$-shifted perverse sheaves is stable
under extensions, we conclude that $R^{k}\psi\Lambda(k)[-k]$ is also
a $-(n-1)$-shited perverse sheaf. 
\end{proof}
Let $\mathcal{L}\in D_{c}^{b}(Y_{1,\bar{\mathbb{F}}},\Lambda)$ be
represented by the complex \[
\dots\to\mathcal{L}^{k-1}\to\mathcal{L}^{k}\to\mathcal{L}^{k+1}\to\dots.\]

\begin{defn}
We define $\tau_{\leq k}\mathcal{L}$ to be the standard truncation
of $\mathcal{L}$, represented by the complex \[
\dots\to\mathcal{L}^{k-1}\to\ker(\mathcal{L}^{k}\to\mathcal{L}^{k+1})\to0.\]
Then $\tau_{\leq k}$ is a functor on $D_{c}^{b}(Y_{1,\bar{\mathbb{F}}},\Lambda)$.
We also define $\tilde{\tau}_{\leq k}\mathcal{K}$ to be represented
by the complex \[
\dots\to\mathcal{L}^{k-1}\to\mathcal{L}^{k}\to\mbox{im }(\mathcal{L}^{k}\to\mathcal{L}^{k+1})\to0.\]

\end{defn}
For every $k$ we have a quasi-isomorphism $\tau_{\leq k}\mathcal{L}\stackrel{\sim}{\to}\tilde{\tau}_{\leq k}\mathcal{L}$,
which is given degree by degree by the inclusion map. 
\begin{cor}
\label{perversity of filtration}The complex $R\psi\Lambda$ is a
$-(n-1)$-shifted perverse sheaf and the truncations $\tau_{\leq k}R\psi\Lambda$
make up a decreasing filtration of $R\psi\Lambda$ by $-(n-1)$-shifted
perverse sheaves. \end{cor}
\begin{proof}
Since the cohomology of $R\psi\Lambda$ vanishes in degrees greater
than $n-1$, we have $R\psi\Lambda\simeq\tau_{\leq n-1}R\psi\Lambda$
so it suffices to prove by induction that each $\tau_{\leq k}R\psi\Lambda$
is a $-(n-1)$-shifted perverse sheaf. For $k=0$, we have $\tau_{\leq0}R\psi\Lambda\simeq R^{0}\psi\Lambda$,
which is a $-(n-1)$-shifted perverse sheaf by Lemma \ref{lem:perversity}.
For $k\geq1$ we have a distinguished triangle \[
(\tau_{\leq k-1}R\psi\Lambda,\tau_{\leq k}R\psi\Lambda,R^{k}\psi\Lambda[-k])\]
 and assuming that $\tau_{\leq k-1}R\psi\Lambda$ is a $-(n-1)$-shifted
perverse sheaf, we conclude that $\tau_{\leq k}R\psi\Lambda$ is as
well. The distinguished triangles become short exact sequences in
the abelian category of perverse sheaves, from which we deduce that
the $\tau_{\leq k}R\psi\Lambda$ make up a decreasing filtration of
$R\psi\Lambda$ and that the graded pieces of this filtration are
the $R^{k}\psi\Lambda$. 
\end{proof}
The complex $R\psi\Lambda$ has an action of $I_{s}$, which acts
trivially on the cohomology sheaves $R^{k}\psi\Lambda$. From this,
it follows that the action of $ $$I_{s}$ factors through the action
of its tame pro-$l$-quotient. Let $T$ be a generator of pro-$l$-part
of the tame inertia (i.e. such that $t_{l}(T)$ is a generator of
$\mathbb{Z}_{l}(1))$. We are interested in understanding the action
of $T$ on $R\psi\Lambda$. In fact, we're interested in understanding
the action of $N=\log T$, by recovering its kernel and image filtration.
We've seen that $T$ acts trivially on the $R^{k}\psi\Lambda$, which
means that $N$ sends $\tau_{\leq k}R\psi\Lambda\to\tilde{\tau}_{\leq k-1}R\psi\Lambda\stackrel{\sim}{\to}\tau_{\leq k-1}R\psi\Lambda$.
We get an induced map \[
\bar{N}:R^{k}\psi\Lambda[-k]\to R^{k-1}\psi\Lambda[-k+1].\]

The next result appears as part $4$ of Lemma 2.2.1 of \cite{Saito},
except that it is stated and proved for the map $\bar{\nu}:R^{k}\psi\Lambda[-k]\to R^{k-1}\psi\Lambda[-k+1]$
induced from $\nu=T-1$. However, $\log T\equiv T-1\pmod{(T-1)^{2}}$,
and $(T-1)^{2}$ sends $\tau_{\leq k}R\psi\Lambda\to\tau_{\leq k-2}R\psi\Lambda$.
Since $R^{k-1}\psi\Lambda[-k+1]\simeq\tau_{\leq k-1}R\psi\Lambda/\tau_{\leq k-2}R\psi\Lambda$,
we deduce that the two maps $\bar{N}$ and $\bar{\nu}$ coincide,
i.e. \[
\bar{N}=\bar{\nu}:R^{k}\psi\Lambda[-k]\to R^{k-1}\psi\Lambda[-k+1].\]
 Therefore, we can rewrite part 4 of Lemma 2.5 of \cite{Saito} as
follows. 
\begin{lem}
\label{twisting}The map $\bar{N}$ and the isomorphisms of Note \ref{eq:full resolution}
make a commutative diagram

\[
\xymatrix{R^{k+1}\psi\Lambda[-(k+1)]\ar[r]\sp-{\sim}\ar[d]^{\bar{N}} & [0\ar[r]\ar[d] & a_{k+1*}\Lambda(-(k+1))\ar[r]\sp-{\delta\wedge}\ar[d]^{\otimes t_{l}(T)} & \dots\ar[r]\sp-{\delta\wedge} & a_{n-1*}\Lambda(-(k+1))]\ar[d]\\
R^{k}\psi\Lambda[-k]\ar[r]\sp-{\sim} & [a_{k*}\Lambda(-k)\ar[r]\sp-{\delta\wedge} & a_{k+1*}\Lambda(-k)\ar[r]\sp-{\delta\wedge} & \dots\ar[r]\sp-{\delta\wedge} & a_{n-1*}\Lambda(-k)]}
,\]
where the sheaves $a_{n-1*}\Lambda(-(k+1))$ and $a_{n-1*}\Lambda(-k)$
are put in degree $n-1$. 
\end{lem}
From the above commutative diagram, it is easy to see that the map
$\bar{N}$ is injective, since we can just compute the cone of the
map of complexes on the right. In general, to compute the kernel and
cokernel of a map of perverse sheaves, we have to compute the cone
$C$ of that map, then the perverse truncation $\tau_{\geq0}^{p}C$
will be the cokernel and $\tau_{\leq-1}^{p}C[-1]$ will be the kernel
(see the proof of Theorem 1.3.6 of \cite{BBD}). It is straightforward
to check that the cone of $\bar{N}$ is quasi-isomorphic to $a_{k*}\Lambda(-k)[-k]$,
which is a $-(n-1)$-shifted perverse sheaf. We deduce that $\bar{N}$
has kernel $0$ and cokernel $a_{k*}\Lambda(-k)[-k]$. 

The fact that $\bar{N}$ is injective means that the canonical filtration
$\tau_{\leq k}R\psi\Lambda$ coincides with the kernel filtration
of $N$ on $R\psi\Lambda$ and that the $R^{k}\psi\Lambda[-k]$ for
$0\leq k\leq n-1$ are the graded pieces of the kernel filtration.
Moreover, the graded pieces of the induced image filtration of $N$
on the $R^{k}\psi\Lambda$ are $a_{k+l*}\Lambda(-l)[-(k+l)]$ for
$0\leq l\leq n-1-k$. This information suffices to reconstruct the
graded pieces of the monodromy filtration on $R\psi\Lambda$. 
\begin{prop}
There is an isomorphism \[
\bigoplus_{l-k=r}a_{(k+l)*}\Lambda(-l)[-(k+l)]\to Gr_{r}^{M}R\psi\Lambda.\]

\end{prop}
This isomorphism, together with the spectral sequence associated to
the monodromy filtration induces the weight spectral sequence (see
Corollary 2.2.4 of \cite{Saito}).

\subsection{The product of semistable schemes}

Let $X_{1}$ and $X_{2}$ be semistable schemes of relative dimension
$n-1$ over $\mathcal{O}_{K}$, and let $\Lambda=\mathbb{Q}_{l}$
or $\bar{\mathbb{Q}}_{l}$. Let $R\psi\Lambda_{X_{i}}$ be the complex
of vanishing cycles on $X_{i,\bar{s}}$ for $i=1,2$ and let $R\psi\Lambda_{X_{1}\times X_{2}}$
be the complex of vanishing cycles on $(X_{1}\times_{S}X_{2})_{\bar{s}}$.
By Proposition \ref{product}, we have \[
R\psi\Lambda_{X_{1}\times X_{2}}\simeq R\psi\Lambda_{X_{1}}\otimes_{\Lambda}R\psi\Lambda_{X_{2}}\]
and notice that this isomorphism is compatible with the action of
the inertia $I$ in $G_{K}$. From Proposition \ref{trivial inertia},
the action of $I$ is trivial on the cohomology sheaves of $R\psi\Lambda_{X_{1}\times X_{2}}$,
so only the pro-$l$ part of $I$ acts nontrivially on $R\psi\Lambda_{X_{1}\times X_{2}}$.
Let $T$ be a generator of the pro-$l$ part of $I$ and set $N=\log T$.
Let $N,N_{1},N_{2}$ denote the action of $N$ on $R\psi\Lambda_{X_{1}\times X_{2}}$,
$R\psi\Lambda_{X_{1}}$ and $R\psi\Lambda_{X_{2}}$ respectively.
Since the above isomorphism is compatible with the action of $T$,
we deduce that $T$ acts on $R\psi\Lambda_{X_{1}}\otimes_{\Lambda[I]}R\psi\Lambda_{X_{2}}$
via $T\otimes T$. From this, we conclude that $N$ acts on $R\psi\Lambda_{X_{1}}\otimes_{\Lambda[I]}R\psi\Lambda_{X_{2}}$
as $N_{1}\otimes1+1\otimes N_{2}$. 

As in the proof of Proposition \ref{localnearbycycles}, we have a
decomposition\[
R^{k}\psi\Lambda\simeq\bigoplus_{l=0}^{k}R^{l}\psi\Lambda_{X_{1}}\otimes R^{k-l}\psi\Lambda_{X_{2}}.\]
We shall see that $N$ induces a map \[
\bar{N}:R^{k}\psi\Lambda_{X_{1}\times X_{2}}[-k]\to R^{k-1}\psi\Lambda_{X_{1}\times X_{2}}[-k+1]\]
 which acts on $R^{l}\psi\Lambda_{X_{1}}\otimes_{\Lambda}R^{k-l}\psi\Lambda_{X_{2}}[-k]$
by $\bar{N}_{1}\otimes1+1\otimes\bar{N}_{2}$. First we prove a few
preliminary results. 

For $i=1,2$ and $0\leq l\leq n$ define the following schemes:
\begin{itemize}
\item Let $Y_{i}/\mathbb{F}$ be the special fiber of $X_{i}$
\item Let $D_{i,1},\dots,D_{i,m_{i}}$ be the irreducible components of
$X_{i}$ 
\item For $J\subseteq\{1,\dots,m_{i}\}$ let $Y_{i,J}$ be $\cap_{j\in J}D_{i,j}$
and let $a_{J}^{i}:Y_{i,J}\to Y_{i}$ be the immersion. Note that
if the cardinality of $J$ is $l+1$, then the scheme $Y_{i,J}$ is
smooth of dimension $n-l-1$. 
\item For all $0\leq l\leq m-1$ set $Y_{i}^{(l)}=\bigsqcup_{\#J=l+1}Y_{i,J}$
and let $a_{l}^{i}:Y_{i}^{(l)}\to Y_{i}$ be the projection.
\end{itemize}
Then for each $i=1,2$ we have a resolution of $R^{l}\psi\Lambda_{X_{i}}[-l]$
in terms of the sheaves $a_{j*}^{i}\Lambda$:\[
R^{l}\psi\Lambda_{X_{i}}[-l]\stackrel{\sim}{\to}[a_{l*}^{i}\Lambda(-l)\to\dots\to a_{n-1*}^{i}\Lambda(-l)],\]
where $a_{n-1*}^{i}\Lambda(-l)$ is put in degree $n-1$. 

Now let $Y/\mathbb{F}$ be the special fiber of $X_{1}\times X_{2}$.
Let \[
Y_{J_{1},J_{2}}=\bigcap_{j_{1}\in J_{1},j_{2}\in J_{2}}(D_{j_{1}}\times_{\mathbb{F}}D_{j_{2}}).\]
Set $Y^{(l_{1},l_{2})}=\bigsqcup_{\#J_{1}=l_{1}+1,\#J_{2}=l_{2}+1}Y_{J_{1},J_{2}}$
and let $a_{l_{1},l_{2}}:Y^{(l_{1},l_{2})}\to Y$ be the projection.
The scheme $Y^{(l_{1},l_{2})}$ is smooth of dimension $2n-2-l_{1}-l_{2}$.
Note that $Y^{(l_{1},l_{2})}=Y_{1}^{(l_{1})}\times Y_{2}^{(l_{2})}$
and that $a_{l_{1},l_{2}*}\Lambda\simeq a_{l_{1}*}^{1}\Lambda\otimes a_{l_{2}*}^{2}\Lambda$,
where the tensor product of sheaves is an external tensor product. 
\begin{lem}
\label{resolution}We have the following resolution of $R^{l}\psi\Lambda_{X_{1}}\otimes R^{k-l}\psi\Lambda_{X_{2}}[-k]$
as the complex \[
a_{l,k-l*}\Lambda(-k)\to a_{l,k-l+1*}\Lambda(-k)\oplus a_{l+1,k-l*}\Lambda(-k)\to\dots\to a_{n-1,n-1*}\Lambda(-k),\]
where the sheaf $a_{n-1,n-1*}\Lambda(-k)$ is put in degree $2n-2$.
The general term of the complex which appears in degree $l_{1}+l_{2}$
is \[
\bigoplus_{\substack{l_{1}\geq l\\
l_{2}\geq k-l}
}a_{l_{1},l_{2}*}\Lambda(-k)\]
For each $l_{1},l_{2}$ the complexes $a_{l_{1},l_{2}*}\Lambda(-k)[-l_{1}-l_{2}]$
are $-(2n-2)$-shifted perverse sheaves, so the complex $R^{k}\psi\Lambda_{X_{1}\times X_{2}}[-k]$
is also $-(2n-2)$-shifted perverse sheaf. \end{lem}
\begin{proof}
Each of the complexes $R^{l}\psi\Lambda_{X_{1}}$ and $R^{k}\psi\Lambda_{X_{2}}$
have resolutions in terms of $a_{1,l_{1}*}\Lambda(-l)$ and $a_{2,l_{2}*}\Lambda(-k+l)$
respectively, where $l\leq l_{1}\leq n-1$ and $k-l\leq l_{2}\leq n-1$.
We form the double complex associated to the product of these resolutions
and the single complex associated to it is a resolution of $R^{l}\psi\Lambda_{X_{1}}\otimes R^{k-l}\psi\Lambda_{X_{2}}[-k]$
of the following form: \[
\xymatrix{a_{l*}^{1}\Lambda(-l)\otimes a_{k-l*}^{2}\Lambda(-k+l)\ar[d]\\
a_{l+1*}^{1}\Lambda(-l)\otimes a_{k-l*}^{2}\Lambda(-k+l)\oplus a_{l*}^{1}\Lambda\otimes a_{k-l+1*}^{2}\Lambda(-k+l)\ar[d]\\
\dots\ar[d]\\
a_{n*}^{1}\Lambda(-l)\otimes a_{n*}^{2}\Lambda(-k+l).}
\]
In the above complex, the sheaf $a_{n-1*}^{1}\Lambda(-l)\otimes a_{n-1*}^{2}\Lambda(-k+l)$
is put in degree $2n-2$. Now we use the formula\[
a_{l_{1},l_{2}*}\Lambda(-k)=a_{l_{1}*}^{1}\Lambda(-l)\otimes a_{l_{2}*}^{1}\Lambda(-k+l)\]
to conclude the first part of the lemma. The complex $a_{l_{1},l_{2}*}\Lambda(-k)[-l_{1}-l_{2}]$
is the direct image via $a_{l_{1},l_{2}*}$ of the complex $\Lambda(-k)[-l_{1}-l_{2}]$
on $Y^{(l_{1},l_{2})}$. Since $Y^{(l_{1},l_{2})}$ is smooth of dimension
$2n-2-l_{1}-l_{2}$, we know that $\Lambda(-k)[-l_{1}-l_{2}]$ is
a $-(2n-2)$-shifted perverse sheaf, so its direct image under the
finite map $a_{l_{1},l_{2}*}$ is also a $-(2n-2)$-shifted perverse
sheaf. We've just seen that each $R^{l}\psi\Lambda_{X_{1}}\otimes R^{k-l}\psi\Lambda_{X_{2}}$
can be obtained from successive extensions of factors of the form
$a_{l_{1},l_{2}*}\Lambda(-k)[-l_{1}-l_{2}]$ and since the category
of $-(2n-2)$-shifted perverse sheaves is stable under extensions
we deduce that $R^{l}\psi\Lambda_{X_{1}}\otimes R^{k-l}\psi\Lambda_{X_{2}}[-k]$
is a $-(2n-2)$-shifted perverse sheaf. Now $R^{k}\psi\Lambda_{X_{1}\times X_{2}}[-k]=\bigoplus_{l=0}^{k}R^{l}\psi\Lambda_{X_{1}}\otimes R^{k-l}\psi\Lambda_{X_{2}}[-k]$,
so it is also a $-(2n-2)$-shifted perverse sheaf.\end{proof}
\begin{cor}
$R\psi\Lambda_{X_{1}\times X_{2}}$ is a $-(2n-2)$-shifted perverse
sheaf. The canonical truncation $\tau_{\leq k}R\psi\Lambda_{X_{1}\times X_{2}}$
is a filtration by $-(2n-2)$-shifted perverse sheaves and the graded
pieces of this filtration are the $R^{k}\psi\Lambda_{X_{1}\times X_{2}}[-k]$.\end{cor}
\begin{proof}
The proof is exactly the same as that of Corollary \ref{perversity of filtration}.
It suffices to show that each $\tau_{\leq k}R\psi\Lambda$ is a $-(2n-2)$-shifted
perverse sheaf and we can do this by induction, using the distinguished
triangle \[
(\tau_{\leq k-1}R\psi\Lambda_{X_{1}\times X_{2}},\tau_{\leq k}R\psi\Lambda_{X_{1}\times X_{2}},R^{k}\psi\Lambda_{X_{1}\times X_{2}}[-k]).\]
Once everything is proven to be in an abelian category, the distinguished
triangle becomes a short exact sequence and we get a filtration on
$R\psi\Lambda_{X_{1}\times X_{2}}$ with its desired graded pieces. 
\end{proof}
Now we can deduce that there is a map \[
\bar{N}:R^{k}\psi\Lambda_{X_{1}\times X_{2}}[-k]\to R^{k-1}\psi\Lambda_{X_{1}\times X_{2}}[-(k-1)].\]
Indeed, since $T$ acts trivially on the cohomology sheaves of $R\psi\Lambda_{X_{1}\times X_{2}}$,
we deduce that $N$ sends $\tau_{\leq k}R\psi\Lambda_{X_{1}\times X_{2}}$
to $\tau_{\leq k-1}R\psi\Lambda_{X_{1}\times X_{2}}$, which induces
$\bar{N}$. It remains to check that this induced map $\bar{N}$ restricted
to $R^{l}\psi\Lambda_{X_{1}}\otimes R^{k-l}\psi\Lambda_{X_{2}}$ is
the same map as $\bar{N}_{1}\otimes1+1\otimes\bar{N}_{2}$, sending
\[
R^{l}\psi\Lambda_{X_{1}}\otimes R^{k-l}\psi\Lambda_{X_{2}}[-k]\to(R^{l-1}\psi\Lambda_{X_{1}}\otimes R^{k}\psi\Lambda_{X_{1}}\oplus R^{l}\psi\Lambda_{X_{1}}\otimes R^{k-l-1}\psi\Lambda_{X_{2}})[-(k-1)].\]

To see this, first notice that for each $0\leq l\leq k\leq n-1$ the
complex $\tau_{\leq l}R\psi\Lambda_{X_{1}}\otimes\tau_{\leq k-l}R\psi\Lambda_{X_{2}}$
is a $-(2n-2)$-shifted perverse sheaf, because it is the external
tensor product of $-(n-1)$-shifted perverse sheaves on $X_{1}$ and
on $X_{2}$. (See proposition 4.2.8 of \cite{BBD}). Let \[
\tau_{\leq l-1}R\psi\Lambda_{X_{1}}\otimes\tau_{\leq k-l}R\psi\Lambda_{X_{2}}+\tau_{\leq l}R\psi\Lambda_{X_{1}}\otimes\tau_{\leq k-l-1}R\psi\Lambda_{X_{2}}\]
 be the image of \[
\tau_{\leq l-1}R\psi\Lambda_{X_{1}}\otimes\tau_{\leq k-l}R\psi\Lambda_{X_{2}}\oplus\tau_{\leq l}R\psi\Lambda_{X_{1}}\otimes\tau_{\leq k-l-1}R\psi\Lambda_{X_{2}}\to\tau_{\leq k-1}R\psi\Lambda_{X_{1}\times X_{2}}.\]
 We have a commutative diagram of $-(2n-2)$-shifted perverse sheaves\[
\xymatrix{\tau_{\leq l}R\psi\Lambda_{X_{1}}\otimes\tau_{\leq k-l}R\psi\Lambda_{X_{2}}\ar[r]\ar[d]^{N_{1}\otimes1+1\otimes N_{2}} & \tau_{\leq k}R\psi\Lambda_{X_{1}\times X_{2}}\ar[d]^{N}\\
\tau_{\leq l-1}R\psi\Lambda_{X_{1}}\otimes\tau_{\leq k-l}R\psi\Lambda_{X_{2}}+\tau_{\leq l}R\psi\Lambda_{X_{1}}\otimes\tau_{\leq k-l-1}R\psi\Lambda_{X_{2}}\ar[r] & \tau_{\leq k-1}R\psi\Lambda_{X_{1}\times X_{2}},}
\]
where the horizontal maps are the natural maps of complexes. 
\begin{lem}
The image of $R_{l,k-l}=\tau_{\leq l}R\psi\Lambda_{X_{1}}\otimes\tau_{\leq k-l}R\psi\Lambda_{X_{2}}$
in $R^{k}\psi\Lambda[-k]$ is $R^{l}\psi\Lambda_{X_{1}}\otimes R^{k-l}\psi\Lambda_{X_{2}}[-k]$. \end{lem}
\begin{proof}
The map of perverse sheaves $R_{l,k-l}\to\tau_{\leq k}R\psi\Lambda_{X_{1}\times X_{2}}\to R^{k}\psi\Lambda[-k]$
factors through \[
R^{l}\psi\Lambda_{X_{1}}\otimes R^{k-l}\psi\Lambda_{X_{2}}[-k]\hookrightarrow R^{k}\psi\Lambda[-k].\]
This can be checked on the level of complexes. We only need to know
that the natural map \[
R_{l,k-l}\stackrel{g}{\to}R^{l}\psi\Lambda_{X_{1}}\otimes R^{k-l}\psi\Lambda_{X_{2}}[-k]\]
is a surjection. This follows once we know that the triangle \[
R_{l-1,k-l}+R_{l,k-l-1}\stackrel{f}{\to}R_{l,k-l}\stackrel{g}{\to}R^{l}\psi\Lambda_{X_{1}}\otimes R^{k-l}\psi\Lambda_{X_{2}}[-k]\]
is distinguished, since then it has to be a short exact sequence of
$-(2n-2)$-shifted perverse sheaves, so $g$ would be a surjection.
To check that the triangle is distinguished, it suffices to compute
the fiber of $g$ and check that it is quasi-isomorphic to \[
\mathcal{M}=\tilde{\tau}_{\leq l-1}R\psi\Lambda_{X_{1}}\otimes\tilde{\tau}_{\leq k-l}R\psi\Lambda_{X_{2}}+\tilde{\tau}_{\leq l}R\psi\Lambda_{X_{1}}\otimes\tilde{\tau}_{\leq k-l-1}R\psi\Lambda_{X_{2}}.\]

Let $\mathcal{K}^{\cdot}$ be a representative for $R\psi\Lambda_{X_{1}}$
and $\mathcal{L}^{\cdot}$ be representative for $R\psi\Lambda_{X_{2}}$.
The degree $j<k$ term of $\mathcal{M}$ and of the fiber of $g$
are both equal to \[
\bigoplus_{i=j-k+l+1}^{l-1}\mathcal{K}^{i}\otimes\mathcal{L}^{j-i}\bigoplus\mathcal{K}^{j-k+l}\otimes\ker(\mathcal{L}^{k-l}\to\mathcal{L}^{k-l+1})\bigoplus\ker(\mathcal{K}^{l}\to\mathcal{K}^{l+1})\otimes\mathcal{L}^{j-l}\]
and the differentials are identical. The last non-zero term $\mathcal{M}^{k}$
in $\mathcal{M}$ appears in degree $k$ and is equal to \[
\ker(\mathcal{K}^{l}\to\mathcal{K}^{l+1})\otimes\mbox{im }(\mathcal{L}^{k-l-1}\to\mathcal{L}^{k-l})+\mbox{im }(\mathcal{K}^{l-1}\to\mathcal{K}^{l})\otimes\ker(\mathcal{L}^{k-l}\to\mathcal{L}^{k-l+1}).\]
The main problem is checking that the following map of complexes is
a quasi-isomorphism\[
\xymatrix{\mathcal{M}^{k}\ar[r]\ar[d]^{h} & 0\ar[d]\\
\ker(\mathcal{K}^{l}\to\mathcal{K}^{l+1})\otimes\ker(\mathcal{L}^{k-l}\to\mathcal{L}^{k-l+1})\ar[r] & H^{l}(\mathcal{K})\otimes H^{k-l}(\mathcal{L}),}
\]
where the left vertical arrow $h$ is the natural inclusion. It is
equivalent to prove that the object in the lower right corner is the
cokernel of $h$. This follows from the Künneth spectral sequence,
when computing the cohomology of the product of the two complexes
\[
\tilde{\mathcal{K}}:=[\mbox{im }(\mathcal{K}^{l-1}\to\mathcal{K}^{l})\to\ker(\mathcal{K}^{l}\to\mathcal{K}^{l+1})]\mbox{ and }\]
\[
\tilde{\mathcal{L}}:=[\mbox{im }(\mathcal{L}^{k-l-1}\to\mathcal{L}^{k-l})\to\ker(\mathcal{L}^{k-l}\to\mathcal{L}^{k-l+1})]\]
Indeed, since $H^{1}(\tilde{\mathcal{K}})=R^{l}\psi\Lambda_{X_{1}}$
and $H^{1}(\tilde{\mathcal{L}})=R^{k-l}\psi\Lambda_{X_{2}}$ are both
flat over $\Lambda$ the Künneth spectral sequence degenerates. We
get $H^{2}(\tilde{\mathcal{K}}\otimes\tilde{\mathcal{L}})=H^{1}(\tilde{\mathcal{K}})\otimes H^{1}(\tilde{\mathcal{L}})$
and this is exactly the statement that $H^{l}(\mathcal{K})\otimes H^{k-l}(\mathcal{L})$
is the cokernel of $h$. 
\end{proof}
Putting together the above discussion, we conclude the following result. 
\begin{prop}
\label{formula for N bar}The action of $N$ on $R\psi\Lambda_{X_{1}\times X_{2}}$
induces a map \[
\bar{N}:R^{k}\psi\Lambda_{X_{1}\times X_{2}}[-k]\to R^{k-1}\psi\Lambda_{X_{1}\times X_{2}}[-(k-1)]\]
 which coincides with $\bar{N}_{1}\otimes1+1\otimes\bar{N}_{2}$ when
restricted to $R^{l}\psi\Lambda_{X_{1}}\otimes R^{k-l}\psi\Lambda_{X_{2}}[-k]$
for each $0\leq l\leq k$. 
\end{prop}
We now use the decomposition of $R^{k}\psi\Lambda_{X_{1}\times X_{2}}[-k]$
in terms of $R^{l}\psi\Lambda_{X_{1}}\otimes R^{k-l}\psi\Lambda_{X_{2}}[-k]$
for $0\leq l\leq k$ and the resolution of $R^{l}\psi\Lambda_{X_{1}}\otimes R^{k-l}\psi\Lambda_{X_{2}}[-k]$
in terms of $a_{l_{1},l_{2}*}\Lambda(-k)[-(l_{1}+l_{2})]$ to get
a resolution of $R^{k}\psi\Lambda_{X_{1}\times X_{2}}[-k]$, of the
form\begin{equation}
\bigoplus_{l_{1}+l_{2}=k}a_{l_{1},l_{2}*}\Lambda(-k)^{\oplus c_{l_{1},l_{2}}^{k}}\to\dots\to\bigoplus_{l_{1}+l_{2}=k+j}a_{l_{1},l_{2}*}\Lambda(-k)^{\oplus c_{l_{1},l_{2}}^{k}}\to\dots,\label{eq:full resolution}\end{equation}
where the first term is put in degree $k$ and the coefficients $c_{l_{1},l_{2}}^{k}$
count how many copies of $a_{l_{1},l_{2}*}\Lambda(-k)$ show up in
the direct sum. 
\begin{lem}
\label{coefficients}Let $c_{l_{1},l_{2}}^{k}$ be the coefficient
of $a_{l_{1},l_{2}*}\Lambda(-k)[-(l_{1}+l_{2})]$ in the resolution
of $R^{k}\psi\Lambda_{X_{1}\times X_{2}}[-k]$. Then \[
c_{l_{1},l_{2}}^{k}=\min(\min(l_{1},l_{2})+1,l_{1}+l_{2}-k+1,k+1).\]
\end{lem}
\begin{proof}
The coefficient $c_{l_{1},l_{2}}^{k}$ counts for how many values
of $0\leq l\leq k$ the resolution of $R^{l}\psi\Lambda_{X_{1}}\otimes R^{k-l}\psi\Lambda_{X_{2}}$
contains the term $a_{l_{1},l_{2}*}\Lambda(-k)[-(l_{1}+l_{2})]$.
This count is clearly bounded by $k+1$, because there are $k+1$
possible values of $l$. When $l_{1}+l_{2}-k+1\leq k+1$, the count
is \[
\min(\min(l_{1},l_{2})+1,l_{1}+l_{2}-k+1),\]
because $a_{l_{1},l_{2}*}\Lambda(-k)[-(l_{1}+l_{2})]$ will show up
in the resolution of $R^{l_{1}-j}\psi\Lambda_{X_{1}}\otimes R^{k-l_{1}+j}\psi\Lambda_{X_{2}}$
for all $0\leq j\leq l_{1}+l_{2}-k+1$ which satisfy $0\leq l_{1}-j\leq k$.
When both $l_{1}$ and $l_{2}$ are less than $k$, all the $j\in0,\dots,l_{1}+l_{2}-k+1$
satisfy the requirement. When $l_{2}\geq k$, there are exactly $l_{1}+1$
values of $j$ which satisfy the requirement and we can treat the
case $l_{1}\geq k$ analogously to get $l_{2}+1$ values of $j$.
This covers the case $l_{1}+l_{2}\leq2k$. In the case $l_{1}+l_{2}\geq2k$,
we need to count all $0\leq j\leq k$ which satisfy $0\leq l_{1}-j\leq k$.
The result is \[
\min(\min(l_{1},l_{2})+1,k+1).\]
This completes the determination of $c_{l_{1},l_{2}}^{k}$. 
\end{proof}
Note that for all $l_{1}+l_{2}\leq2k-2$, we have $c_{l_{1},l_{2}}^{k}\leq c_{l_{1},l_{2}}^{k-1}$.
For $l_{1}+l_{2}=2k-1$ we always have $\min(l_{1},l_{2})+1\leq k<k+1$,
so that \[
c_{l_{1},l_{2}}^{k}=c_{l_{1},l_{2}}^{k-1}=\min(l_{1},l_{2})+1.\]
However, $c_{k,k}^{k}=k+1>k=c_{k,k}^{k-1}$ and for $l_{1}+l_{2}\geq2k$
we have $c_{l_{1},l_{2}}^{k}\geq c_{l_{1},l_{2}}^{k-1}$. 

We now have an explicit description of \[
\bar{N}:R^{k}\psi\Lambda_{X_{1}\times X_{2}}[-k]\to R^{k-1}\psi\Lambda_{X_{1}\times X_{2}}[-k]\]
as a map of complexes with terms of the form $\bigoplus_{l_{1}+l_{2}=k+j}c_{l_{1},l_{2}}^{k}\cdot a_{l_{1},l_{2}*}\Lambda(-k)$,
which are put in degree $k+j$. Writing $\bar{N}=\bar{N}_{1}\otimes1+1\otimes\bar{N}_{2}$
as a map of complexes, we will be able to compute both the kernel
and cokernel of $\bar{N}$. First we need a few preliminary results. 
\begin{lem}
\label{distinguished}Let $\mathcal{C},\mathcal{D}$ be complexes
in the derived category of an abelian category. Let $f:\mathcal{C}\to\mathcal{D}$
be a map of complexes which is injective degree by degree, $f^{k}:\mathcal{C}^{k}\hookrightarrow\mathcal{D}^{k}$.
Let $\bar{\mathcal{D}}^{k}=\mbox{coker }(f^{k})$ and let $\bar{\mathcal{D}}$
be the complex with terms $\bar{\mathcal{D}}^{k}$ and differential
$\bar{d}$ induced by the differential $d$ of $\mathcal{D}$. Then
there exists a quasi-isomorphism $\mbox{Cone}(f)\simeq\mathcal{\bar{D}}$. \end{lem}
\begin{proof}
The proof is essentially a diagram chase. Since we are in an abelian
category, we can pretend that our objects have elements and perform
diagram chases. For an explanation of why we can do this, see \cite{Ryan}.
Alternatively, the abelian category we want to apply this to is that
of sheaves of $\Lambda$-modules on $Y$, so we can talk about sections. 

Note that $Cone(f)$ is the complex which has the degree $k$ term
equal to $\mathcal{D}^{k}\oplus\mathcal{C}^{k+1}$. The differential
sends $(x,y)\in\mathcal{D}^{k}\oplus\mathcal{C}^{k+1}$ to $(dx+f(y),-dy)\in\mathcal{D}^{k+1}\oplus\mathcal{C}^{k+2}$.
There is a natural map of complexes $Cone(f)\to\bar{\mathcal{D}}$
which is defined degree by degree as $\mathcal{D}^{k}\oplus\mathcal{C}^{k+1}\to\bar{\mathcal{D}}^{k}$,
where the map $\mathcal{D}^{k}\to\bar{\mathcal{D}}^{k}$ is the natural
projection and where $\mathcal{C}^{k+1}$ is sent to $0$. In order
to check that this map is a quasi-isomorphism, we we just need to
check that it induces an isomorphism on cohomology in degree $k$.
We have the following diagram\[
\xymatrix{\dots\ar[r] & \mathcal{D}^{k-1}\oplus\mathcal{C}^{k}\ar[d]\ar[r] & \mathcal{D}^{k}\oplus\mathcal{C}{}^{k+1}\ar[d]\ar[r] & \mathcal{D}^{k+1}\oplus\mathcal{C}^{k+2}\ar[d]\ar[r] & \dots\\
\dots\ar[r] & \bar{\mathcal{D}}^{k-1}\ar[r] & \bar{\mathcal{D}}^{k}\ar[r] & \bar{\mathcal{D}}^{k+1}\ar[r] & \dots}
.\]
Let $(x,y)\in\ker(\mathcal{D}^{k}\oplus\mathcal{C}^{k+1}\to\mathcal{D}^{k+1}\oplus\mathcal{C}^{k+2})$.
This means that $dx=-f^{k+1}(y)$. Assume that $(x,y)\mapsto0\in\bar{\mathcal{D}}^{k}/\mbox{im}\bar{\mathcal{D}}^{k-1}$.
Then $x-dx'=f^{k}(y')$ for some $x'\in\mathcal{D}^{k-1},y'\in\mathcal{C}^{k}$.
We find that $f^{k+1}(dy')=df^{k}(y')=dx=-f^{k+1}(y)$ and since $f^{k+1}$
is injective we conclude $-dy'=y$. So $(x,y)=(f(y')+dx',-dy')$ which
is the image of the element $(x',y')\in\mathcal{D}^{k-1}\oplus\mathcal{C}^{k}$.
So the induced map on cohomology is injective. The proof that the
map is surjective is even easier: take $x\in\mathcal{D}^{k}$ such
that $dx\in\mbox{im }(\mathcal{C}^{k+1}\to\mathcal{D}^{k+1})$. Then
$dx=-f(y)$ for a unique $y\in\mathcal{C}^{k+1}$. Then $(x,y)\mapsto\bar{x}\in\ker(\bar{\mathcal{D}}^{k}\to\bar{\mathcal{D}}^{k+1})$
and $(x,y)\in\ker(\mathcal{D}^{k}\oplus\mathcal{C}^{k+1}\to\mathcal{D}^{k+1}\oplus\mathcal{C}^{k+2})$.
To see this, note that $dx+f(y)=0$ by the choice of $y$ and we also
have $dy=0$ because $d^{2}x=-f(dy)=0$ and because $f$ is injective. \end{proof}
\begin{cor}
\label{ker and coker}Assume that we are now working in the category
$D_{c}^{b}(Y,\Lambda)$ and that we have a map of bounded complexes
$f:\mathcal{C}\to\mathcal{D}$ which is injective degree by degree.
We can form the complex $\bar{\mathcal{D}}$ as in Lemma \ref{distinguished}.
Assume that the short exact sequence of sheaves \[
0\to\mathcal{C}^{k}\stackrel{f^{k}}{\to}\mathcal{D}^{k}\to\bar{\mathcal{D}}^{k}\to0\]
is splittable. Assume also that $\mathcal{C}^{k}[-k]$ and $\mathcal{D}^{k}[-k]$
are $-(2n-2)$-shifted perverse sheaves. 

Then $\bar{\mathcal{D}}^{k}[-k]$ is a $-(2n-2)$-shifted perverse
sheaf and thus so is $\mathcal{\bar{D}}$ (since it is an extension
of $\bar{\mathcal{D}}^{k}[-k]$ for finitely many $k$). Moreover,
the following is an exact sequence of $-(2n-2)$-shifted perverse
sheaves\[
0\to\mathcal{C}\to\mathcal{D}\to\bar{\mathcal{D}}\to0\]
\end{cor}
\begin{proof}
$\bar{\mathcal{D}}^{k}[-k]$ is a $-(2n-2)$-shifted perverse sheaf
because it is a direct factor of $\mathcal{D}^{k}[-k]$ and so $\bar{\mathcal{D}}$
is also a $-(2n-2)$-shifted perverse sheaf. If $\Lambda$ was torsion,
then we could identify the category $D_{c}^{b}(Y,\Lambda)$ with the
derived category of the category of constructible sheaves of $\Lambda$-modules
and the corollary would follow from Lemma \ref{distinguished}. However,
the cases we are most interested in are$\Lambda=\mathbb{Q}_{l}$ or
$\bar{\mathbb{Q}}_{l}$. It is possible that by checking the definition
of the category $D_{c}^{b}(Y,\Lambda)$ carefully, we could ensure
that a version of Lemma \ref{distinguished} applies to our case.
However, an alternative approach uses Beilinson's result which identifies
$D_{c}^{b}(Y,\Lambda)$ with the derived category of perverse sheaves
on $Y$, see \cite{Be}. 

We see that the map $f:\mathcal{C}\to\mathcal{D}$ is injective, since
we can think of it as a map of filtered objects, which is injective
on the $k$th graded pieces for each $k$. Indeed $\mathcal{C}$ is
a successive extension of the $-(2n-2)$-shifted perverse sheaves
$\mathcal{C}^{k}[-k]$ and $\mathcal{D}$ is a successive extension
of $\mathcal{D}^{k}[-k]$ and the fact that $f$ is a map of complexes
implies that $f$ respects these extension. Let $k$ be the largest
integer for which either of $\mathcal{C}^{k}$ and $\mathcal{D}^{k}$
is non-zero. We have the commutative diagram of exact sequences\[
\xymatrix{0\ar[r] & \mathcal{C}^{k}[-k]\ar[d]^{f^{k}[-k]}\ar[r] & \mathcal{C}'\ar[d]\ar[r] & \mathcal{C}^{k-1}[-k+1]\ar[d]^{f^{k-1}[-k+1]}\ar[r] & 0\\
0\ar[r] & \mathcal{D}^{k}[-k]\ar[r] & \mathcal{D}'\ar[r] & \mathcal{D}^{k-1}[-k+1]\ar[r] & 0}
,\]
where the arrows on the left and on the right are injective. The fact
that the middle map is also injective follows from a standard diagrm
chase. (Note that we are working in the category of $-(2n-2)$-shifted
perverse sheaves, which is abelian, so we can perform diagram chases
by \cite{Ryan}.) The injectivity of $f$ follows by induction. 

By a repeated application of the snake lemma in the abelian category
of $-(2n-2)$-shifted perverse sheaves, we see that the cokernel of
$f$ is a succesive extension of terms of the form $\mathcal{\bar{D}}^{k}[-k]$.
In order to identify this cokernel with $\bar{\mathcal{D}}$, it suffices
to check that the differential of $\bar{\mathcal{D}}$ coincides in
$\mathrm{Ext}^{1}(\bar{\mathcal{D}}^{k}[-k],\bar{\mathcal{D}}^{k-1}[-k+1])$
with the extension class which defines the cokernel. To check this,
it is enough to see that the following square is commutative \[
\xymatrix{\mathcal{D}^{k-1}[-k+1]\ar[d]^{f^{k-1}[-k+1]}\ar[r] & \mathcal{D}^{k}[-k+1]\ar[d]^{f^{k}[-k+1]}\\
\bar{\mathcal{D}}^{k-1}[-k+1]\ar[r] & \bar{\mathcal{D}}^{k}[-k+1]}
,\]
where the top (resp. bottom) horizontal map is the boundary map obtained
from considering the distinguished triangle $(\mathcal{D}^{k}[-k],\mathcal{D}',\mathcal{D}^{k-1}[-k+1])$
(resp. $(\bar{\mathcal{D}}^{k}[-k],\mathcal{\bar{D}}',\bar{\mathcal{D}}^{k-1}[-k+1])$)
in $D_{c}^{b}(Y,\Lambda)$. The top boundary map is the differential
of $\mathcal{D}$ and if the square is commutative, then the bottom
map must be the differential of $\bar{\mathcal{D}}.$ The commutativity
can be checked by hand, by making the boundary maps explicit using
the construction of the cone. (There is a natural map \[
\xymatrix{\mathcal{D}^{k}[-k]\ar[d]\ar[r] & \mathcal{D}'\ar[d]\\
0\ar[r] & \mathcal{D}^{k-1}[-k+1]}
,\]
which is a quasi-isomorphism in $D_{c}^{b}(Y,\Lambda)$. The boundary
map of the distinguished triangle is obtained by composing the inverse
of this quasi-isomorphism with the natural map \[
\xymatrix{\mathcal{D}^{k}[-k]\ar[d]\ar[r] & \mathcal{D}'\ar[d]\\
\mathcal{D}^{k}[-k]\ar[r] & 0}
.\]
The same construction works for $\bar{\mathcal{D}}$ and it is straightforward
to check the commutativity now.) \end{proof}
\begin{lem}
\label{ker of N bar}Let $k\geq1$. Consider the map \[
\bar{N}:R^{k}\psi\Lambda_{X_{1}\times X_{2}}[-k]\to R^{k-1}\psi\Lambda_{X_{1}\times X_{2}}[-(k-1)].\]
Define the complex \[
\mathcal{P}_{k}=\left[a_{k,k*}\Lambda(-k)\stackrel{\wedge\delta}{\to}a_{k,k+1*}\Lambda(-k)\bigoplus a_{k+1,k*}\Lambda(-k)\to\dots\stackrel{\wedge\delta}{\to}a_{n-1,n-1*}\Lambda(-k)\right],\]
where $a_{k,k*}\Lambda(-k)$ is put in degree $2k$. The factor $a_{l_{1},l_{2}*}\Lambda(-k)$
appears in the resolution of $\mathcal{P}$ in degree $l_{1}+l_{2}$
whenever $k\leq l_{1},l_{2}\leq n-1$. Also define the complex \[
\mathcal{R}_{k}=\left[\bigoplus_{j=0}^{k-1}a_{j,k-1-j*}\Lambda(-(k-1))\to\dots\to a_{k-1,k-1*}\Lambda((-(k-1))\right],\]
where the first term is put in degree $k-1$ and the term $a_{l_{1},l_{2}*}\Lambda(-(k-1))$
appears in degree $l_{1}+l_{2}$ whenever $0\leq l_{1},l_{2}\leq k-1$.

Then $\mathcal{P}_{k}\simeq\ker(\bar{N})$ and $\mathcal{R}_{k}\simeq\mbox{coker}(\bar{N})$. \end{lem}
\begin{proof}
Note that both $\mathcal{P}_{k}$ and $\mathcal{R}_{k}$ are $-(2n-2)$-shifted
perverse sheaves, by the same argument we've used before. The proof
will go as follows: we will first define a map $\mathcal{P}_{k}\to R^{k}\psi\Lambda_{X_{1}\times X_{2}}[-k]$
and check that $\bar{N}$ kills the image of $\mathcal{P}_{k}$. We
use Corollary \ref{ker and coker} to check that the map $\mathcal{P}_{k}\to R^{k}\psi\Lambda_{X_{1}\times X_{2}}[-k]$
is an injection and to compute its cokernel $\mathcal{Q}_{k}$. Then
we check using Corollary \ref{ker and coker} again that the induced
map $\mathcal{Q}_{k}\to R^{k-1}\psi\Lambda[-(k-1)]$ is an injection
and we identify its cokernel with $\mathcal{R}.$

For the first step, note that it suffices to define the maps \[
f^{l_{1},l_{2}}:a_{l_{1},l_{2}*}\Lambda(-k)\to a_{l_{1},l_{2}*}\Lambda(-k)^{\oplus(k+1)}\]
for all $l_{1},l_{2}\geq k$ and we do so by $x\mapsto(x,-x,\dots,(-1)^{k}x)$.
These maps are clearly compatible with the differentials $\wedge\delta$,
so they induce a map $f:\mathcal{P}_{k}\to R^{k}\psi\Lambda_{X_{1}\times X_{2}}[-k]$
(this is a map of complexes between $\mathcal{P}$ and the standard
representative of $R^{k}\psi\Lambda_{X_{1}\times X_{2}}[-k]$). Moreover,
we can check that the restriction \[
\bar{N}:a_{l_{1},l_{2}*}\Lambda(-k)^{\oplus(k+1)}\to a_{l_{1},l_{2}*}\Lambda(-k)^{\oplus k}\]
sends $(x,-x,\dots,(-1)^{k}x)\mapsto(0,\dots,0)$. 

Indeed, the $j$th factor $a_{l_{1},l_{2}*}\Lambda(-k)$ appears in
the resolution of $R^{j}\psi\Lambda_{X_{1}}\otimes R^{k-j}\psi\Lambda_{X_{2}}[-k]$.
The latter object is sent by $\bar{N}_{1}\otimes1$ to $R^{j-1}\psi\Lambda_{X_{1}}\otimes R^{k-j}\psi\Lambda_{X_{2}}[-(k-1)]$
for $1\leq j\leq k$ and by $1\otimes\bar{N}_{2}$ to $R^{j}\psi\Lambda_{X_{1}}\otimes R^{k-1-j}\psi\Lambda_{X_{2}}[-(k-1)]$
for $0\leq j\leq k-1$. We also know that $\bar{N}_{1}\otimes1$ kills
$R^{0}\psi\Lambda_{X_{1}}\otimes R^{k}\psi\Lambda_{X_{2}}[-k]$ and
similarly $1\otimes\bar{N}_{2}$ kills $R^{k}\psi\Lambda_{X_{1}}\otimes R^{0}\psi\Lambda_{X_{2}}[-k]$.
By Lemma \ref{twisting}, we find that for $1\leq j\leq k-1$ \[
(0,\dots,0,x,0\dots,0)\mapsto(0,\dots,x\otimes t_{l}(T),x\otimes t_{l}(T),0\dots,0),\]
where the term $x$ is put in position $j$ and the terms $x\otimes t_{l}(T)$
are put in positions $j-1$ and $j$. We also have \[
(x,0,\dots,0)\mapsto(x\otimes t_{l}(T),0,\dots,0)\mbox{ and }(0,\dots,0,x)\mapsto(0,\dots,0,x\otimes t_{l}(T)).\]
Thus, we find that $\bar{N}$ sends \[
(x,-x,\dots,(-1)^{k}x)\mapsto(x\otimes t_{l}(T)-x\otimes t_{l}(T),\dots,(-1)^{k-1}x\otimes t_{l}(T)+(-1)^{k}x\otimes t_{l}(T)),\]
and the term on the right is $(0,\dots,0)$. Since we have exhibited
$\bar{N}\circ f$ as a chain map and we've checked that it vanishes
degree by degree, we conclude that $\bar{N}\circ f=0$. Thus, $f(\mathcal{P}_{k})\subseteq\ker\bar{N}$. 

Note that for all $l_{1},l_{2}$ we can identify the quotient of $a_{l_{1},l_{2}*}\Lambda(-k)^{\oplus(k+1)}$
by $f^{l_{1},l_{2}}(a_{l_{1},l_{2}*}\Lambda(-k))$ with $a_{l_{1},l_{2}*}\Lambda(-k)^{\oplus k}$.
The resulting exact sequence \[
0\to a_{l_{1},l_{2}*}\Lambda(-k)\stackrel{f^{l_{1},l_{2}}}{\to}a_{l_{1},l_{2}*}\Lambda(-k)^{\oplus(k+1)}\stackrel{}{\to}a_{l_{1},l_{2}*}\Lambda(-k)^{\oplus k}\to0\]
is splittable, because the third term is free over $\Lambda$. By
Corollary \ref{ker and coker}, the map $f:\mathcal{P}_{k}\to R^{k}\psi\Lambda_{X_{1}\times X_{2}}[-k]$
is injective and we can identify degree by degree the complex $\mathcal{Q}_{k}$
representing the cokernel of $f$. In degrees less than $2k-1$, the
terms of $\mathcal{Q}_{k}$ are the same as those of $R^{k}\psi\Lambda_{X_{1}\times X_{2}}[-k]$
and in degrees at least $2k-1$, they are the terms of $R^{k-1}\psi\Lambda_{X_{1}\times X_{2}}[-k+1]$. 

To prove that the induced map $\mathcal{Q}_{k}\to R^{k-1}\psi\Lambda_{X_{1}\times X_{2}}[-(k-1)]$
is injective it suffices to check degree by degree and the proof is
analogous to the one for $f:\mathcal{P}_{k}\to R^{k}\psi\Lambda_{X_{1}\times X_{2}}[-k]$.
The cokernel is identified with $\mathcal{R}_{k}$ degree by degree,
via the exact sequence \[
0\to a_{l_{1},l_{2}*}\Lambda(-k)^{\oplus(k-1)}\stackrel{\bar{N}^{l_{1},l_{2}}}{\to}a_{l_{1},l_{2}*}\Lambda(-(k-1))^{\oplus k}\to a_{l_{1},l_{2}*}\Lambda(-(k-1))\to0\]
 for $0\leq l_{1},l_{2}\leq k-1$. \end{proof}
\begin{note*}
1. The complex $\mathcal{P}_{k}$ has as its factors exactly the terms
$a_{l_{1},l_{2}*}\Lambda(-k)[-(l_{1}+l_{2})]$ for which $c_{l_{1},l_{2}}^{k}-c_{l_{1},l_{2}}^{k-1}=1$,
while $\mathcal{R}_{k}$ has as its factors the terms $a_{l_{1},l_{2}*}\Lambda(-(k-1))[-(l_{1}+l_{2})]$
for which $c_{l_{1},l_{2}}^{k-1}-c_{l_{1},l_{2}}^{k}=1$. 

2. Another way to express the kernel of $\bar{N}$ is as the image
of $R^{2k}\psi\Lambda_{X_{1}\times X_{2}}[-2k]$ in $R^{k}\psi\Lambda_{X_{1}\times X_{2}}[-k]$
under the map \[
\bar{N}_{1}^{k}\otimes1-\bar{N}_{1}^{k-1}\otimes\bar{N}_{2}+\dots+(-1)^{k}1\otimes\bar{N}_{2}^{k}.\]
This follows from Lemmas \ref{twisting} and \ref{ker of N bar}.\end{note*}
\begin{cor}
\label{kernel of N}The filtration of $R\psi\Lambda_{X_{1}\times X_{2}}$
by $\tau_{\leq k}R\psi\Lambda_{X_{1}\times X_{2}}$ induces a filtration
on $\ker N$. The first graded piece of this filtration $gr^{1}\ker N$
is $R^{0}\psi\Lambda_{X_{1}\times X_{2}}$. The graded piece $gr^{k+1}\ker N$
of this filtration is $\mathcal{P}_{k}$. \end{cor}
\begin{proof}
We've already seen that $N$ maps all of $R^{0}\psi\Lambda_{X_{1}\times X_{2}}$
to $0$, since $T$ acts trivially on the cohomology of $R\psi\Lambda_{X_{1}\times X_{2}}$.
This identifies the first graded piece to be $R^{0}\psi\Lambda_{X_{1}\times X_{2}}$. 

In order to identify the $(k+1)$st graded piece, we will once more
pretend that our shifted perverse sheaves have elements. We can do
this since the $(2-2n)$-shifted perverse sheaves form an abelian
category and we only need to do this in order to simplify the exposition.
First notice that $gr^{k}\ker N\subseteq\mathcal{P}_{k}$, since anything
in the kernel of $N$ reduces to something in the kernel of $\bar{N}$. 

So it suffices to show that any $x\in\mathcal{P}_{k}$ lifts to some
$\tilde{x}\in\ker N$. Pick any $\tilde{x}\in\tau_{\leq k}R\psi\Lambda_{X_{1}\times X_{2}}$
lifting $x$. Since $\bar{N}$ sends $x$ to $0$, we conclude that
$N$ maps $\tilde{x}$ to $\tau_{\leq k-2}R\psi\Lambda_{X_{1}\times X_{2}}$.
The image of $N\tilde{x}$ in $R^{k-2}\psi\Lambda_{X_{1}\times X_{2}}[-k+2]$
depends on our choice of the lift $\tilde{x}$. However, the image
of $N\tilde{x}$ in $\mathcal{R}_{k-1}$ only depends on $x$. If
we can show that that image is $0$, we conclude that we can pick
a lift $\tilde{x}$ such that $N\tilde{x}\in\tau_{\leq k-3}R\psi\Lambda$.
We can continue applying the same argument while modifying our choice
of lift $\tilde{x}$, such that $N\tilde{x}\in\tau_{\leq k-j}R\psi\Lambda_{X_{1}\times X_{2}}$
for larger and larger $j$. In the end we see that $N\tilde{x}=0$. 

It remains to check that the map $\mathcal{P}_{k}\to\mathcal{R}_{k-1}$
sending $x\in\mathcal{P}_{k}$ to the image of $N\tilde{x}$ in $\mathcal{R}_{k-1}$
is $0$. We can see this by checking that any map $\mathcal{P}_{k}\to\mathcal{R}_{k-1}$
is $0$. Indeed, we have the following decompositions of $\mathcal{P}_{k}$
and $\mathcal{R}_{k-1}$ as $(2-2n)$-shifted perverse sheaves:\[
\mathcal{P}_{k}=\left[a_{k,k*}\Lambda(-k)\stackrel{\wedge\delta}{\to}a_{k,k+1*}\Lambda(-k)\bigoplus a_{k+1,k*}\Lambda(-k)\to\dots\stackrel{\wedge\delta}{\to}a_{n-1,n-1*}\Lambda(-k)\right]\]
\[
\mbox{and }\mathcal{R}_{k-1}=\left[\bigoplus_{j=0}^{k-1}a_{j,k-2-j*}\Lambda(-(k-2))\to\dots\to a_{k-2,k-2*}\Lambda((-(k-2))\right].\]
 Each of the factors $a_{l_{1},l_{2}*}\Lambda$ is a direct sum of
factors of the form $a_{J_{1},J_{2}*}\Lambda$, where $\mbox{card}J_{i}=l_{i}$
for $i=1,2$ and $a_{J_{1},J_{2}}:Y_{J_{1},J_{2}}\hookrightarrow Y$
is a closed immersion. Each factor $a_{J_{1},J_{2}*}\Lambda$ is a
simple $(2-2n)$-shifted perverse sheaf, so we have decompositions
into simple factors for both $\mathcal{P}_{k}$ and $\mathcal{R}_{k-1}$.
It is straightforward to see that $\mathcal{P}_{k}$ and $\mathcal{R}_{k-1}$
have no simple factors in common. Thus, any map $\mathcal{P}_{k}\to\mathcal{R}_{k-1}$
must vanish. The same holds true for any map $\mathcal{P}_{k}\to\mathcal{R}_{k-j}$
for any $2\leq j\leq k$. 
\end{proof}
The filtration with graded pieces $\mathcal{P}_{k}$ on $\ker N$
induces a filtration on $\ker N/\mbox{im}N\cap\ker N$ whose graded
pieces are $\mathcal{P}_{k}/\mbox{im}\bar{N}$. Indeed, it suffices
to check that the image of $\mbox{im}N$ in $\mathcal{P}_{k}$ coincides
with $\mbox{im}\bar{N}$. The simplest way to see this is again by
using a diagram chase. First, it is obvious that for \[
\bar{N}:R^{k}\psi\Lambda_{X_{1}\times X_{2}}\to R^{k-1}\psi\Lambda_{X_{1}\times X_{2}}\]
 we have $\mbox{im}\bar{N}\subseteq gr^{k}\mbox{im}N$. Now let $x\in gr^{k}\mbox{im}N$.
This means that there exists a lift $\tilde{x}\in\tau_{\leq k-1}R\psi\Lambda_{X_{1}\times X_{2}}$
of $x$ and an element $\tilde{y}\in\tau_{\leq k+j}R\psi\Lambda_{X_{1}\times X_{2}}$
with $0\leq j\leq2n-k$ such that $\tilde{x}=N\tilde{y}$. In order
to conclude that $x\in\mbox{im}\bar{N}$, it suffices to show that
we can take $j=0$. In the case $j\geq1$, let $y\in R^{k+j}\psi\Lambda_{X_{1}\times X_{2}}$
be the image of $\tilde{y}$. We have $\bar{N}y=0$ and in this case
we've seen in the proof of Corollary \ref{kernel of N} that we can
find $\tilde{y}^{(1)}\in\tau_{\leq k+j-1}R\psi\Lambda_{X_{1}\times X_{2}}$
such that $N(\tilde{y}-\tilde{y}^{(1)})=0$. In other words, $\tilde{x}=N\tilde{y}^{(1)}$
and we can replace $j$ by $j-1$. After finitely many steps, we can
find $\tilde{y}^{(j)}\in\tau_{\leq k}R\psi\Lambda_{X_{1}\times X_{2}}$
such that $\tilde{x}=N\tilde{y}^{(j)}$ . Thus, $x\in\mbox{im}\bar{N}$. 
\begin{lem}
\label{base case}The filtration of $R\psi\Lambda_{X_{1}\times X_{2}}$
by $\tau_{\leq k}R\psi\Lambda$ induces a filtration on $\ker N/\mbox{im}N\cap\ker N$
with the $(k+1)$-st graded piece $a_{k,k*}\Lambda(-k)[-2k]$ for
$0\leq k\leq n-1$. \end{lem}
\begin{proof}
First, we need to compute the quotient $R^{0}\psi\Lambda_{X_{1}\times X_{2}}/\mbox{im}N$,
which is the same as $R^{0}\psi\Lambda_{X_{1}\times X_{2}}/\mathcal{Q}_{1}=\mathcal{R}_{1}$
and $\mathcal{R}_{1}\simeq a_{0,0*}\Lambda$ by Lemma \ref{ker and coker}.

Now we must compute for each $k\geq0$ the quotient of $(2-2n)$-shifted
perverse sheaves $\mathcal{P}_{k}/\mbox{im}\bar{N}$. This is the
same as $\mathcal{P}_{k}/\mathcal{Q}_{k+1}$, which is also the image
of $\mathcal{P}_{k}$ in $\mathcal{R}_{k+1}$ via \[
\mathcal{P}_{k}\hookrightarrow R^{k}\psi\Lambda[-k]\twoheadrightarrow\mathcal{R}_{k+1}.\]
Recall that we have decompositions for both $\mathcal{P}_{k}$ and
$\mathcal{R}_{k+1}$ in terms of simple objects in the category of
$(2-2n)$-shifted perverse sheaves,\[
\mathcal{P}_{k}=\left[a_{k,k*}\Lambda(-k)\stackrel{\wedge\delta}{\to};a_{k,k+1*}\Lambda(-k)\bigoplus a_{k+1,k*}\Lambda(-k)\to\dots\stackrel{\wedge\delta}{\to}a_{n-1,n-1*}\Lambda(-k)\right]\]
 \[
\mbox{ and }\mathcal{R}_{k+1}=\left[\bigoplus_{j=0}^{k+1}a_{j,k-j*}\Lambda(-k)\to\dots\to a_{k,k*}\Lambda(-k)\right].\]
The only simple factors that show up in both decompositions are those
that show up in $a_{k,k*}\Lambda(-k)[-2k]$, so these are the only
factors that may have non-zero image in $\mathcal{R}_{k+1}$. Thus,
$\mathcal{P}_{k}/\mbox{im}\bar{N}$ is a quotient of $a_{k,k*}\Lambda(-k)[-2k]$
and it remains to see that it is the whole thing. As seen in Lemma
\ref{ker of N bar}, the map $\mathcal{P}_{k}\to\mathcal{R}_{k+1}$
can be described as a composition of chain maps. The composition in
degree $2k$ is the map \[
a_{k,k*}\Lambda(-k)\hookrightarrow a_{k,k*}\Lambda(-k)^{\oplus k+1}\twoheadrightarrow a_{k,k*}\Lambda(-k)\]
where the inclusion sends $x\mapsto(x,-x,\dots,(-1)^{k+1}x)$ and
the surjection is a quotient by $(x,x,0,\dots,0)$, $(0,x,x,0\dots,0),\dots,$$(0,\dots,0,x,x)$
for $x\in a_{k,k*}\Lambda(-k)$. It is elementary to check that the
composition of these two maps is an isomorphism, so we are done. 
\end{proof}
$ $Analogously, we can compute the kernel and cokernel of \[
\bar{N}^{j}:R^{k}\psi\Lambda_{X_{1}\times X_{2}}[-k]\to R^{k-j}\psi\Lambda_{X_{1}}[-k+j]\]
for $2\leq j\leq k\leq2n-2$ and use this to recover the graded pieces
of a filtration on $\ker N^{j}/\ker N^{j-1}$ and on $(\ker N^{j}/\ker N^{j-1})/(\mbox{im}N\cap\ker N^{j})$. 
\begin{lem}
\label{graded pieces}Let $2\leq j\leq2n-2$. The filtration of $R\psi\Lambda_{X_{1}\times X_{2}}$
by $\tau_{\leq k}R\psi\Lambda$ induces a filtration on \[
(\ker N^{j}/\ker N^{j-1})/(\mbox{im}N\cap\ker N^{j}).\]
The first graded piece of this filtration is isomorphic to \[
\bigoplus_{i=0}^{j-1}a_{i,j-1-i*}\Lambda(-j+1)[-j+1].\]
For $k\geq1$, the $(k+1)$-st graded piece is isomorphic to \[
(\ker\bar{N}^{j}/\ker\bar{N}^{j-1})/(\mbox{im}\bar{N}\cap\ker\bar{N}^{j})\]
 where \[
\bar{N}^{j}:R^{k+j-1}\psi\Lambda[-(k+j-1)]\to R^{k-1}\psi\Lambda[-k+1].\]
More explicitly, the $(k+1)$-st graded piece is isomorphic to \[
\bigoplus_{i=1}^{j}a_{k+i-1,k+j-i*}\Lambda(-(k+j-1))[-2k-j+1].\]
\end{lem}
\begin{proof}
We will prove the lemma by induction on $j$. The base case $j=1$
is proven in Corollary \ref{kernel of N} and Lemma \ref{base case}.
Assume it is true for $j-1$. 

To prove the first claim, note that the first graded piece of \[
(\ker N^{j}/\mbox{im}N\cap\ker N^{j})/(\ker N^{j-1}/\mbox{im}N\cap\ker N^{j-1})\]
has to be a quotient of \[
R^{j-1}\psi\Lambda_{X_{1}\times X_{2}}[-j+1]/\mathcal{Q}_{j}\simeq\mathcal{R}_{j}.\]
This is true because $\tau_{\leq j-1}R\psi\Lambda_{X_{1}\times X_{2}}\subseteq\ker N^{j}$
and $\tau_{\leq j-2}R\psi\Lambda_{X_{1}\times X_{2}}\subseteq\ker N^{j-1}$
and \[
R^{j-1}\psi\Lambda_{X_{1}\times X_{2}}[-j+1]=\tau_{\leq j-1}R\psi\Lambda_{X_{1}\times X_{2}}/\tau_{\leq j-2}R\psi\Lambda_{X_{1}\times X_{2}}.\]
More precisely, the first graded piece has to be a quotient of \[
\mathcal{R}_{j}/(\ker N^{j-2}/\mbox{im}N\cap\ker N^{j-2})\]
by the second graded piece of \[
(\ker N^{j-1}/\ker N^{j-2})/\mbox{(im}N\cap\ker N^{j-1}).\]
(Here, we abusively write \[
\ker N^{j-2}/\mbox{im}N\cap\ker N^{j-2}\]
 where we mean the image of this object in $R_{j}$.) By the induction
hypothesis, this second graded piece is\[
\bigoplus_{i=1}^{j-1}a_{i,j-i*}\Lambda(-j+1)[-j].\]
Continuing this argument, we see that in order to get the first graded
piece of $(\ker N^{j}/\ker N^{j-1})/\mbox{(im}N\cap\ker N^{j})$ we
must quotient $\mathcal{R}_{j}$ successively by \[
\bigoplus_{i=1}^{j-k}a_{k+i-1,j-i*}\Lambda(-j+1)[-k-j+1],\]
with $k$ going from $j-1$ down to $1$. (This corresponds to quotienting
out successively by the $j$th graded piece of $\ker N/\mbox{(im}N\cap\ker N)$,
the $(j-1)$st graded piece of $(\ker N^{2}/\ker N)/(\mbox{im}N\cap\ker N^{2})$
down to the second graded piece of $(\ker N^{j-1}/\ker N^{j-2})/(\mbox{im}N\cap\ker N^{j-1})$.)
We know that \[
\mathcal{R}_{j}=\left[\bigoplus_{i=0}^{j-1}a_{i,j-1-i*}\Lambda(-(j-1))\to\dots\to a_{j-1,j-1*}\Lambda((-(j-1))\right],\]
with general term in degree $k+j-1$ equal to \[
\bigoplus_{i=1}^{j-k}a_{k+i-1,j-i*}\Lambda(-(j-1).\]
After quotienting out successively, we are left with only the degree
$j-1$ term, which is \[
\bigoplus_{i=0}^{j-1}a_{i,j+1-i*}\Lambda(-(j-1))[-(j-1)],\]
as desired. 

In order to identify the $(k+1)$-st graded piece of \[
(\ker N^{j}/\ker N^{j-1})/(\mbox{im}N\cap\ker N^{j})\]
for $k\geq1$, we first identify the kernel of $\bar{N}^{j}:R^{k+j-1}\psi\Lambda_{X_{1}\times X_{2}}\to R^{k-1}\psi\Lambda_{X_{1}\times X_{2}}$
as a map of perverse sheaves, as in Lemma \ref{ker of N bar}. Then
we can identify it with the $(k+1)$-st graded piece of $\ker N^{j}$
as in Lemma \ref{base case} and quotient by $\mathcal{Q}_{k+j}$.
Finally, we can use induction as above to compute the $(k+1)$-st
graded piece of $(\ker N^{j}/\ker N^{j-1})/(\mbox{im}N\cap\ker N^{j})$.
\end{proof}
Let \[
\mbox{Gr}^{q}\mbox{Gr}_{p}R\psi\Lambda=(\ker N^{p}\cap\mbox{im}N^{q})/(\ker N^{p-1}\cap\mbox{im}N^{q})+(\ker N^{p}\cap\mbox{im}N^{q+1}).\]
The monodromy filtration $M_{r}R\psi\Lambda$ has graded pieces $\mbox{Gr}_{r}^{M}R\psi\Lambda$
isomorphic to \[
\mbox{Gr}_{r}^{M}R\psi\Lambda\simeq\bigoplus_{p-q=r}\mbox{Gr}^{q}\mbox{Gr}_{p}R\psi\Lambda\]
by Lemma 2.1 of \cite{Saito}, so to understand the graded pieces
of the monodromy filtration it suffices to understand the $\mbox{Gr}^{q}\mbox{Gr}_{p}R\psi\Lambda$.
Lemma \ref{graded pieces} exhibits a filtration on $\mbox{Gr}^{0}\mbox{Gr}_{p}R\psi\Lambda$
with the $(k+1)$-st graded piece isomorphic to \[
\bigoplus_{i=1}^{p}a_{k+i-1,k+p-i*}\Lambda(-(k+p-1))[-2k-p+1].\]
The action of $N^{q}$ induces an isomorphism of $\mbox{Gr}^{0}\mbox{Gr}_{p+q}R\psi\Lambda$
with $\mbox{Gr}^{q}\mbox{Gr}_{p}R\psi\Lambda$, so there is a filtration
on the latter with the $(k+1)$-st graded piece isomorphic to \[
\bigoplus_{i=1}^{p+q}a_{k+i-1,k+p+q-i*}\Lambda(-(k+p-1))[-2k-p-q+1].\]
We can use the spectral sequence associated to a filtration to compute
the terms in the monodromy spectral sequence \[
E_{1}^{r,m-r}=H^{m}(Y_{\bar{\mathbb{F}}},\mbox{ gr}_{-r}^{M}R\psi\Lambda)=\bigoplus_{p-q=-r}H^{m}(Y_{\bar{\mathbb{F}}},\mbox{Gr}^{q}\mbox{Gr}_{p}R\psi\Lambda)\]
\[
\Rightarrow H^{m}(Y_{\bar{\mathbb{F}}},R\psi\Lambda)=H^{m}(X_{\bar{K}},\Lambda).\]

\begin{cor}
There is a spectral sequence \[
E_{1}^{k+1,m-k-1}=\bigoplus_{i=1}^{p+q}H^{m}(Y_{\bar{\mathbb{F}}},a_{k+i-1,k+p+q-i*}\Lambda(-(k+p-1))[-2k-p-q+1])\]
\[
\Rightarrow H^{m}(Y_{\bar{\mathbb{F}}},\mbox{Gr}^{q}\mbox{Gr}_{p}R\psi\Lambda)\]
 compatible with the action of $G_{\mathbb{F}}$. This can be rewritten
as \[
E_{1}^{k+1,m-k-1}=\bigoplus_{i=1}^{p+q}H^{m-2k-p-q+1}(Y_{\bar{\mathbb{F}}}^{(k+i-1,k+p+q-i)},\Lambda(-(k+p-1)))\]
\[
\Rightarrow H^{m}(Y_{\bar{\mathbb{F}}},\mbox{Gr}^{q}\mbox{Gr}_{p}R\psi\Lambda).\]
\end{cor}
\begin{proof}
This follows from Lemma 5.2.18 of \cite{MSaito}, since all filtrations
we are working with are filtrations as $(2-2n)$-shifted perverse
sheaves and these correspond to quasi-filtrations in the sense of
M. Saito in the derived category. 
\end{proof}

\subsection{More general schemes}

In this subsection, we will explain how the results of the previous
section concerning products of semistable schemes apply to more general
schemes, in particular to the Shimura varieties $X_{U}/\mathcal{O}_{K}$. 

Let $X'/\mathcal{O}_{K}$ be a scheme such that the completions of
the strict henselizations $\mathcal{O}_{X,'s}^{\wedge}$ at closed
geometric points $s$ are isomorphic to \[
W[[X_{1},\dots,X_{n},Y_{1},\dots Y_{n}]]/(X_{i_{1}}\cdot\dots\cdot X_{i_{r}}-\pi,Y_{j_{1}}\cdot\dots\cdot Y_{j_{s}}-\pi)\]
for some indices $i_{1},\dots,i_{r},j_{1},\dots,j_{s}\in\{1,\dots n\}$
and some $1\leq r,s\leq n$. Also assume that the special fiber $Y'$
is a union of closed subschemes $Y'_{1,j}$ with $j\in\{1,\dots n\}$,
which are cut out by one local equation, such that if $s$ is a closed
geometric point of $Y'_{1,j}$, then $j\in\{i_{1},\dots,i_{r}\}$
and $Y'_{1,j}$ is cut out in $\mathcal{O}_{X',s}^{\wedge}$ by the
equation $X_{j}=0$. Similarly, assume that $Y'$ is a union of closed
subschemes $Y'_{2,j}$ with $j\in\{1,\dots,n\}$, which are cut out
by one local equation such that if $s$ is a closed geometric point
of $Y'_{2,j}$ then $j\in\{j_{1},\dots,j_{r}\}$ and $Y'_{2,j}$ is
cut out in $\mathcal{O}_{X',s}^{\wedge}$ by the equation $Y_{j}=0$. 

Let $X/X'$ be smooth of dimension $m$ and let $Y$ be the special
fiber of $X$ and $Y_{i,j}=Y'_{i,j}\times_{X'}X$ for $i=1,2$ and
$j=1,\dots,n$. As in Lemma \ref{locally etale over}, $X'$ is locally
etale over\[
X_{r,s}=\mbox{Spec }\mathcal{O}_{K}[X_{1},\dots,X_{n},Y_{1},\dots Y_{n}]/(\prod_{i=1}^{r}X_{i}-\pi,\prod_{j-1}^{s}Y_{j}-\pi),\]
 so $X$ is locally etale over

\[
X_{r,s,m}=\mbox{Spec }\mathcal{O}_{K}[X_{1},\dots,X_{n},Y_{1},\dots Y_{n},Z_{1},\dots,Z_{m}]/(\prod_{i=1}^{r}X_{i}-\pi,\prod_{j-1}^{s}Y_{j}-\pi),\]
which is a product of semistable schemes. The results of Section 3
apply to $X'$ and it is easy to check that they also apply to $X$.
In particular, we know that the inertia $I_{K}$ acts trivially on
the sheaves of nearby cycles $R^{k}\psi\Lambda$ of $X$ and we have
a description of the $R^{k}\psi\Lambda$ in terms of the log structure
we put on $X/\mbox{Spec }\mathcal{O}_{K}$. Let $a_{j}^{i}:Y_{i,j}\to Y$
denote the closed immersion for $i=1,2$ and $j\in\{1,\dots n\}$.
Then by Corollary \ref{globalnearby}, we have an isomorphism \[
R^{k}\psi\Lambda(k)\simeq\wedge^{k}((\oplus_{j=1}^{n}a_{j*}^{1}\Lambda)/\Lambda+(\oplus_{j=1}^{n}a_{j*}^{2}\Lambda)/\Lambda)\]

For $i=1,2$ and $J_{i}\subseteq\{1,\dots,n\}$, let \[
Y_{J_{1},J_{2}}=(\cap_{j_{1}\in J_{1}}Y_{1,j_{1}})\bigcap(\cap_{j_{2}\in J_{2}}Y_{2,j_{2}})\]
and let $a_{J_{1},J_{2}}:Y_{J_{1},J_{2}}\to Y$ be the closed immersion.
Set \[
Y^{(l_{1},l_{2})}=\bigsqcup_{\#J_{1}=l_{1}+1,\#J_{2}=l_{2}+1}Y_{J_{1},J_{2}}\]
 and let $a_{l_{1},l_{2}}:Y^{(l_{1},l_{2})}\to Y$ be the projection.
The scheme $Y^{(l_{1},l_{2})}$ is smooth of dimension $\dim Y-l_{1}-l_{2}$
(we can see this from the strict local rings). 

We can write \[
R^{k}\psi\Lambda\simeq\bigoplus_{l=0}^{k}\wedge^{l}((\oplus_{j=1}^{n}a_{j*}^{1}\Lambda)/\Lambda)\otimes\wedge^{k-l}((\oplus_{j=1}^{n}a_{j*}^{2}\Lambda)/\Lambda)(-k)\]
and then define the map of sheaves on $Y$ \[
\theta_{k}:R^{k}\psi\Lambda\to\sum_{l=0}^{k}a_{l,k-l*}\Lambda(-k)\]
as a sum of maps for $l$ going from $0$ to $k$. First, define for
$i=1,2$ \[
\delta_{l_{i}}:\bigwedge^{l_{i}}((\oplus_{j=1}^{n}a_{j*}^{i}\Lambda)/\Lambda)\to\bigwedge^{l_{i}+1}(\oplus_{j=1}^{n}a_{j*}^{i}\Lambda)\]
by sending \[
a_{j_{1}*}^{i}\Lambda\wedge\dots\wedge a_{j_{l_{i}}*}^{i}\Lambda\to\oplus_{j\not=j_{1},\dots,j_{l_{i}}}a_{j_{1}*}^{i}\Lambda\wedge\dots\wedge a_{j_{l_{i}}*}^{i}\Lambda\wedge a_{j*}^{i}\Lambda\]
 via $1\mapsto(1,\dots,1).$ (It is easy to check that the above map
is well-defined.) Then notice that \[
\wedge^{l+1}(\oplus_{j=1}^{n}a_{j*}^{1}\Lambda)\otimes\wedge^{k+1-l}(\oplus a_{j*}^{2}\Lambda)\simeq a_{l,k-l*}\Lambda.\]
Indeed, for $J_{1},J_{2}\subseteq\{1,\dots,n\}$ with $\#J_{1}=l+1,\#J_{2}=k+1-l$
we have \[
(\bigwedge_{j_{1}\in J_{1}}a_{j_{1}*}^{1}\Lambda)\otimes(\bigwedge_{j_{2}\in J_{2}}a_{j_{2}*}^{2}\Lambda)\simeq a_{J_{1},J_{2}*}\Lambda\]
because $Y_{J_{1},J_{2}}=(\cap_{j_{1}\in J_{1}}Y_{1,j_{1}})\times_{Y}(\cap_{j_{2}\in J_{2}}Y_{2,j_{2}})$
and we can sum the above identity over all $J_{1},J_{2}$ of the prescribed
cardinality. 
\begin{lem}
\label{resolution shimura}The following sequence is exact \[
R^{k}\psi\Lambda\stackrel{\theta_{k}}{\to}\bigoplus_{l=0}^{k}a_{l,k-l*}\Lambda(-k)^{\oplus c_{l,k-l}^{k}}\stackrel{}{\to}\bigoplus_{l=0}^{k+1}a_{l,k+1-l*}\Lambda(-k)^{\oplus c_{l,k+1-l}^{k}}\stackrel{}{\to}\dots\]
\[
\stackrel{}{\to}\bigoplus_{l=0}^{2n-2}a_{l,2n-2-l*}\Lambda(-k)^{\oplus c_{l,2n-2-l}^{k}}\to0\]
where the first map is the one defined above and the coefficients
$c_{l_{1},l_{2}}^{k}$ are defined in Lemma \ref{coefficients}. The
remaining maps in the sequence are global maps of sheaves corresponding
to $\wedge\delta_{1}\pm\wedge\delta_{2}$, where $\delta_{i}\in\oplus_{j=1}^{n}a_{j*}^{i}\Lambda$
is equal to $(1,\dots,1)$ for $i=1,2$. These maps are defined on
each of the $c_{l_{1},l_{2}}^{k}$ factors in the unique way which
makes them compatible with the maps in the resolution (\ref{eq:full resolution}). 

We can think of $\theta_{k}$ as a quasi-isomorphism of $R^{k}\psi\Lambda[-k]$
with the complex \[
\bigoplus_{l=0}^{k}a_{l,k-l*}\Lambda(-k)^{\oplus c_{l,k-l}^{k}}\stackrel{}{\to}\dots\stackrel{}{\to}\bigoplus_{l=0}^{2n-2}a_{l,2n-2-l*}\Lambda(-k)^{\oplus c_{l,2n-2-l}^{k}},\]
where the leftmost term is put in degree $k$. \end{lem}
\begin{proof}
It suffices to check exactness locally and we know that $X$ is locally
etale over products $X_{1}\times_{\mathcal{O}_{K}}X_{2}$ of semistable
schemes. Lemma \ref{resolution} proves the above statement in the
case of $X_{1}\times_{\mathcal{O}_{K}}X_{2}$ and the corresponding
sheaves on $Y$ are obtained by restriction (etale pullback) from
the special fiber $Y_{1}\times_{\mathbb{F}}Y_{2}$ of $X_{1}\times_{\mathcal{O}_{K}}X_{2}$. \end{proof}
\begin{cor}
The complex $R\psi\Lambda$ is a $-\dim Y$-shifted perverse sheaf
and the canonical filtration $\tau_{\leq k}R\psi\Lambda$ with graded
pieces $R^{k}\psi\Lambda[-k]$ is a filtration by $-\dim Y$-shifted
perverse sheaves. The monodromy operator $N$ sends $\tau_{\leq k}R\psi\Lambda$
to $\tau_{\leq k-1}R\psi\Lambda$ and this induces a map \[
\bar{N}:R^{k}\psi\Lambda[-k]\to R^{k-1}\psi\Lambda[-k+1].\]

\end{cor}
The next step is to understand the action of monodromy $\bar{N}$
and obtain an explicit description of $\bar{N}$ in terms of the resolution
of $R^{k}\psi\Lambda$ given by Lemma \ref{resolution shimura}. This
can be done etale locally, since on the nearby cycles for $X_{1}\times_{\mathcal{O}_{K}}X_{2}$
we know that $\bar{N}$ acts as $\bar{N}_{1}\otimes1+1\otimes\bar{N}_{2}$
from Proposition \ref{formula for N bar} and we have a good description
of $\bar{N}_{1}$ and $\bar{N}_{2}$ from Lemma \ref{twisting}. However,
we present here a different method for computing $\bar{N}$, which
works in greater generality. 
\begin{prop}
\label{extension class}The following diagram is commutative:\[
\xymatrix{R^{k+1}\psi\Lambda[-k-1]\ar[d]^{\bar{N}}\ar[r]\sp-{\sim} & [0\ar[r]\ar[d] & R^{k+1}\psi\Lambda\ar[d]^{\otimes t_{l}(T)}]\\
R^{k}\psi\Lambda[-k]\ar[r]\sp-{\sim} & [i^{*}R^{k+1}j_{*}\Lambda(1)\ar[r] & R^{k+1}\psi\Lambda(1)]}
\]
where in the right column the sheaves $R^{k+1}\psi\Lambda$ are put
in degree $k+1$. \end{prop}
\begin{rem}
The fact that the above formula could hold was suggested to us by
reading Ogus' paper \cite{Ogus}, which proves an analogous formula
for log smooth schemes in the complex analytic world. The proof is
identical to the proof of part 4 of Lemma 2.5 of \cite{Saito}, which
is meant for the semistable case but does not use semistability. The
same result should hold for any log smooth scheme $X/\mathcal{O}_{K}$
with vertical log structure and where the action of $I_{K}$ on $R^{k}\psi\Lambda$
is trivial for all $k$. 

For $0\leq k\leq2n-2$ define the complex \[
\mathcal{L}_{k}:=[\sum_{l=0}^{k}c_{l,k-l}^{k}a_{l,k-l*}\Lambda(-k)\to\dots\to\sum_{l=0}^{2n-2}c_{l,2n-2-l}^{k}a_{l,2n-2-l*}\Lambda(-k)],\]
where the sheaves $a_{l,k-l*}\Lambda(-k)$ are put in degree $k$.
We will define a map of complexes $\mathcal{L}_{k+1}\to\mathcal{L}_{k}(-1)$
degree by degree, as a sum over $l_{1}+l_{2}=k'$ of maps \[
f^{l_{1},l_{2}}:a_{l_{1},l_{2}}\Lambda^{\oplus c_{l_{1},l_{2}}^{k}}\to a_{l_{1},l_{2}*}\Lambda^{\oplus c_{l_{1},l_{2}}^{k}}.\]

Note that each coefficient $c_{l_{1},l_{2}}^{k}$ reflects for how
many $0\leq l'\leq k$ the term $a_{l_{1},l_{2}*}\Lambda$ appears
in the resolution of \[
\bigwedge^{l'}((\oplus_{j=1}^{n}a_{j*}^{1}\Lambda)/\Lambda)\otimes\bigwedge^{k-l'}((\oplus_{j=1}^{n}a_{j*}^{2}\Lambda)/\Lambda)(-k).\]
The set of such $l'$ has cardinallity $c_{l_{1},l_{2}}^{k}$ and
is always a subset of consecutive integers in $\{1,\dots,k\}$. Denote
the set of $l'$ by $C_{l_{1},l_{2}}^{k}$. Thus, we can order the
terms $a_{l_{1},l_{2}*}\Lambda$ by $l'$ and get a basis for $(a_{l_{1},l_{2}*}\Lambda)^{\oplus c_{l_{1},l_{2}}^{k}}$
over $a_{l_{1},l_{2}*}\Lambda$. It is easy to explain what $f^{l_{1},l_{2}}$
does to each element of $C_{l_{1},l_{2}}^{k+1}$: it sends \[
l'\in C_{l_{1},l_{2}}^{k+1}\mapsto\{l'-1,l'\}\cap C_{l_{1},l_{2}}^{k}.\]
When both $l'-1,l'\in C_{l_{1},l_{2}}^{k}$, the element of the basis
of $(a_{l_{1},l_{2}*}\Lambda)^{\oplus c_{l_{1},l_{2}}^{k+1}}$ given
by $(0,\dots0,1,0,\dots,0)$ where the $1$ appears in the position
corresponding to $l'$ is sent to the element of the basis of $(a_{l_{1},l_{2}*}\Lambda)^{\oplus c_{l_{1},l_{2}}^{k}}$
given by $(0,\dots,0,1,1,0,\dots,0)$ where the two $1$'s are in
positions corresponding to $l'-1$ and $l'$. If $l'-1\not\in C_{l_{1},l_{2}}^{k}$
but $l'\in C_{l_{1},l_{2}}^{k}$ then $l'=0$ and $(1,0,\dots,0)\mapsto(1,0\dots,0)$.
If $l'-1\in C_{l_{1},l_{2}}^{k}$ but $l'\not\in C_{l_{1},l_{2}}^{k}$
then $l'=k+1$ and $(0,\dots,0,1)\mapsto(0,\dots,0,1)$. This completes
the definition of $f^{(l_{1},l_{2})}$. \end{rem}
\begin{cor}
The following diagram is commutative \[
\xymatrix{R^{k+1}\psi\Lambda[-k-1]\ar[d]^{\bar{N}}\ar[r]\sp-{\sim} & \mathcal{L}_{k+1}\ar[d]^{f}\\
R^{k}\psi\Lambda[-k]\ar[r]\sp-{\sim} & \mathcal{L}_{k}}
,\]
 where the sheaves on the right are put in degree $2n-2$. The map
$f$ is a map of complexes, which acts degree by degree as \[
f^{k'}=\sum_{l_{1}+l_{2}=k'}f^{l_{1},l_{2}}[-k-1]\otimes t_{l}(T),\]
 where $f^{l_{1},l_{2}}:a_{l_{1},l_{2}}\Lambda^{\oplus c_{l_{1},l_{2}}^{k}}\to a_{l_{1},l_{2}*}\Lambda^{\oplus c_{l_{1},l_{2}}^{k}}$
was defined above. \end{cor}
\begin{proof}
This can be checked etale locally, using Proposition \ref{formula for N bar},
which states that $\bar{N}=\bar{N}_{1}\otimes1+1\otimes\bar{N}_{2}$
over a product $X_{1}\times_{\mathcal{O}_{K}}X_{2}$ of semistable
schemes and using the fact that each of the $\bar{N}_{i}$ can be
described as \[
\xymatrix{0\ar[r]\ar[d] & a_{k+1*}^{i}\Lambda(-(k+1))\ar[r]\sp-{\delta\wedge}\ar[d]^{\otimes t_{l}(T)} & \dots\ar[r]\sp-{\delta\wedge} & a_{n-1*}^{i}\Lambda(-(k+1))\ar[d]\\
a_{k*}^{i}\Lambda(-k)\ar[r]\sp-{\delta\wedge} & a_{k+1*}^{i}\Lambda(-k)\ar[r]\sp-{\delta\wedge} & \dots\ar[r]\sp-{\delta\wedge} & a_{n-1*}^{i}\Lambda(-k)}
,\]
 for $i=1,2$. 

This can also be checked globally, by using Proposition \ref{extension class}
to replace the leftmost column of our diagram by \[
\xymatrix{0\ar[d]\ar[r] & R^{k+1}\psi\Lambda\ar[d]^{\otimes t_{l}(T)}\\
i^{*}R^{k+1}j_{*}\Lambda(1)\ar[r] & R^{k+1}\psi\Lambda(1)}
,\]
where the left column is put in degree $k$. In fact, it suffices
to understand the map of complexes \[
\xymatrix{0\ar[d]\ar[r] & R^{k+1}\psi\Lambda\ar[d]^{\mathrm{id}}\\
i^{*}R^{k+1}j_{*}\Lambda\ar[r] & R^{k+1}\psi\Lambda}
,\]
and check that it is compatible with the map \[
\xymatrix{0\ar[r]\ar[d]^{f^{k}\otimes t_{l}(T)^{-1}} & \dots\ar[r] & \sum_{l=0}^{2n-2}c_{l,2n-2-l}^{k+1}a_{l,2n-2-l*}\Lambda(-k-1)\ar[d]^{f^{2n-2}\otimes t_{l}(T)^{-1}}\\
\sum_{l=0}^{k}c_{l,k-l}^{k}a_{l,k-l*}\Lambda(-k-1)\ar[r] & \dots\ar[r] & \sum_{l=0}^{2n-2}c_{l,2n-2-l}^{k}a_{l,2n-2-l*}\Lambda(-k-1)}
.\]
Let $\mathcal{K}=\mbox{Cone}(f\otimes t_{l}(T)^{-1}:\mathcal{L}_{k+1}\to\mathcal{L}_{k}(-1))$.
The triangle 

\[
\xymatrix{R^{k}\psi\Lambda[-k-1]\ar[r] & i^{*}R^{k+1}j_{*}\Lambda[-k-1]\ar[r] & R^{k+1}\psi\Lambda[-k-1]\ar[rr]\sp-{\bar{N}\otimes t_{l}(T)^{-1}} &  & R^{k}\psi\Lambda[-k]}
\]
is distinguished. It suffices to see that we can define a map $g:i^{*}R^{k+1}j_{*}\Lambda[-k]\to\mathcal{K}$
which makes the first two squares of the following diagram commute:

\[
\xymatrix{R^{k}\psi\Lambda[-k-1]\ar[r]\ar[d]^{\theta_{k}[-1]} & i^{*}R^{k+1}j_{*}\Lambda[-k-1]\ar[r]\ar[d]^{g[-1]} & R^{k+1}\psi\Lambda[-k-1]\ar[rr]\sp-{\bar{N}\otimes t_{l}(T)^{-1}}\ar[d]^{\theta_{k+1}} &  & R^{k}\psi\Lambda[-k]\ar[d]^{\theta_{k}}\\
\mathcal{L}_{k}(-1)[-1]\ar[r] & \mathcal{K}[-1]\ar[r] & \mathcal{L}_{k+1}\ar[rr]\sp-{f\otimes t_{l}(T)^{-1}} &  & \mathcal{L}_{k}(-1)}
\]
If the middle square is commutative, then there must exist $\theta':R^{k}\psi\Lambda[-k-1]\to\mathcal{L}_{k}(-1)[-1]$
making the diagram a morphism of distinguished triangles. Then $\theta'$
would make the first square commutative, so $\theta'$ and $\theta_{k}[-1]$
coincide once they are pushed forward to $\mathcal{K}[-1]$. However,
\[
\mbox{Hom}(R^{k}\psi\Lambda[-k-1],\mathcal{L}_{k+1}[-1])\simeq\mbox{Hom}(R^{k}\psi\Lambda[-k],R^{k+1}\psi\Lambda[-k-1])=0,\]
so the Hom exact sequence associated to the bottom distinguished triangle
implies that $\theta'=\theta_{k}[-1]$. The diagram above is a morphism
of distinguished triangles with $\theta_{k}[-1]$ as the leftmost
morphism. This tells us that the third triangle in the diagram is
also commutative, which is what we wanted to prove. 

We can compute $i^{*}R^{k+1}j_{*}\Lambda$ using the log structure
on $X$:\[
i^{*}R^{k+1}j_{*}\Lambda\simeq\wedge^{k+1}((\bigoplus_{j=1}^{n}a_{j*}^{1}\Lambda\oplus\bigoplus_{j=1}^{n}a_{j*}^{2}\Lambda)/(1,\dots1,0,\dots0)=(0,\dots,0,1,\dots,1)).\]
We can also compute $\mathcal{K}$ explicitly, since we have an explicit
description of each $f^{k',l_{1},l_{2}}$. The first non-zero term
of $\mathcal{K}$ appears in degree $k$ and it is isomorphic to \[
\sum_{l=0}^{k}a_{l,k-l*}\Lambda.\]
There is a natural map of complexes $i^{*}R^{k+1}j_{*}\Lambda[-k]\to\mathcal{K}$,
which sends \[
a_{J_{1},J_{2}*}\Lambda\to\bigoplus_{J_{1}'\supset J_{1},\#J_{1}'=\#J_{1}+1}a_{J_{1}',J_{2}*}\Lambda\oplus\bigoplus_{J_{2}'\supset J_{2},\#J_{2}'=\#J_{2}+1}a_{J_{1},J'_{2}*}\Lambda,\]
when $J_{1},J_{2}$ are both non-zero. The map sends \[
a_{J_{1},\emptyset}{}_{*}\Lambda\to\bigoplus_{\#J_{2}'=1}a_{J_{1},J'_{2}*}\Lambda\mbox{ and }a_{\emptyset,J_{2}*}\Lambda\to\bigoplus_{\#J_{1}'=1}a_{J'_{1},J_{2}*}\Lambda.\]
It is easy to see that the above map is well-defined on $i^{*}R^{k+1}j_{*}\Lambda[-k]$
and that it is indeed a map of complexes. It remains to see that the
above map of complexes $i^{*}R^{k+1}j_{*}\Lambda[-k]\to\mathcal{K}$
makes the first two squares of the diagram commute. This is tedious,
but straightforward to verify. \end{proof}
\begin{rem}
Another way of proving the above corollary is to notice that Proposition
\ref{extension class} shows that the map \[
\bar{N}:R^{k+1}\psi\Lambda[-k-1]\to R^{k}\psi\Lambda[-k]\]
 is given by the cup product with the map $\gamma\otimes t_{l}(T):\bar{M}_{rel}^{gp}(-k-1)\to\Lambda(-k)[1]$
where $\gamma:\bar{M}_{rel}^{gp}\to\Lambda[1]$ is the map corresponding
to the class of the extension\[
0\to\Lambda\to\bar{M}^{gp}\to\bar{M}_{rel}^{gp}\to0\]
 of sheaves of $\Lambda$-modules on $Y$. Locally, $X$ is etale
over a product of semistable schemes $X_{1}\times_{\mathcal{O}_{K}}X_{2}$
and the extension $\bar{M}^{gp}$ is a Baire sum of extensions \[
0\to\Lambda\to\bar{M}_{1}^{gp}\to\bar{M}_{1,rel}^{gp}\to0\mbox{ and}\]
\[
0\to\Lambda\to\bar{M}_{2}^{gp}\to\bar{M}_{2,rel}^{gp}\to0,\]
which correspond to the log structures of $X_{1}$ and $X_{2}$ and
which by Proposition \ref{extension class} determine the maps $\bar{N}_{1}$
and $\bar{N}_{2}$. The Baire sum of extensions translates into $\bar{N}=\bar{N}_{1}\otimes1+1\otimes\bar{N}_{2}$
locally on $Y$. However, it is straightforward to check locally on
$Y$ that the map $f:\mathcal{L}_{k}\to\mathcal{L}_{k+1}$ is the
same as $\bar{N}_{1}\otimes1+1\otimes\bar{N}_{2}$. Thus, $f$ and
$\bar{N}$ are maps of perverse sheaves on $Y$ which agree locally
on $Y$, which means that $f$ and $\bar{N}$ agree globally. This
proves the corollary without appealing to Proposition \ref{extension class}.
(In fact, it suggests an alternate proof of Proposition \ref{extension class}.)\end{rem}
\begin{lem}
The map $\bar{N}:R^{k}\psi\Lambda[-k]\to R^{k-1}\psi\Lambda[-k+1]$
has kernel \[
\mathcal{P}_{k}\simeq\left[a_{k,k*}\Lambda(-k)\stackrel{\wedge\delta}{\to}a_{k,k+1*}\Lambda(-k)\bigoplus a_{k+1,k*}\Lambda(-k)\to\dots\stackrel{\wedge\delta}{\to}a_{n-1,n-1*}\Lambda(-k)\right],\]
where the first term is put in degree $2k$ and cokernel \[
\mathcal{R}_{k}\simeq\left[\bigoplus_{j=0}^{k-1}a_{j,k-1-j*}\Lambda(-(k-1))\to\dots\to a_{k-1,k-1*}\Lambda((-(k-1))\right],\]
where the first term is put in degree $k-1$.\end{lem}
\begin{proof}
The proof is identical to the proof of Lemma \ref{ker of N bar},
since by Proposition \ref{extension class} we have a description
of $\bar{N}$ as a degree by degree map \[
f:\mathcal{L}_{k}\to\mathcal{L}_{k-1}.\]
\end{proof}
\begin{cor}
The filtration of $R\psi\Lambda$ by $\tau_{\leq k}R\psi\Lambda$
induces a filtration on $\ker N$. The first graded piece of this
filtration $gr^{1}\ker N$ is $R^{0}\psi\Lambda$. The graded piece
$gr^{k+1}\ker N$ of this filtration is $\mathcal{P}_{k}$.\end{cor}
\begin{proof}
This can be proved the same way as Corollary \ref{kernel of N}. The
only tricky part is seeing that we can identify a graded piece of
$\ker\bar{N}$ with a graded piece of $\ker N$. In other words, we
want to show that for $\bar{N}:R^{k}\psi\Lambda[-k]\to R^{k-1}\psi\Lambda[-k+1]$
and $x\in\ker\bar{N}$ we can find a lift $\tilde{x}\in\tau_{\leq k}R\psi\Lambda$
of $x$ such that $\tilde{x}\in\ker N$. As in the proof of Corollary
\ref{kernel of N}, we can define a map $\mathcal{P}_{k}\to\mathcal{R}_{k-1}$
sending $x$ to the image of $N\tilde{x}$ in $\mathcal{R}_{k-1}$,
which turns out to be independent of the lift $\tilde{x}$. We want
to see that this map vanishes but in fact any map $\mathcal{P}_{k}\to\mathcal{R}_{k-1}$
vanishes. Note that \[
a_{l_{1},l_{2}*}\Lambda[-l_{1}-l_{2}]\simeq\bigoplus_{\#S=l_{1}+1,\#T=l_{2}+1}a_{S,T*}\Lambda[-l_{1}-l_{2}].\]
The scheme $Y_{S,T}$ is smooth of pure dimension $\dim Y-l_{1}-l_{2}$
and so it is a disjoint union of its irreducible (connected) components
which are smooth of pure dimension $\dim Y-l_{1}-l_{2}$. Thus, each
$a_{S,T*}\Lambda[-l_{1}-l_{2}]$ is the direct sum of the pushforwards
of the $-\dim Y$-shifted perverse sheaves $\Lambda[-l_{1}-l_{2}]$
on the irreducible components of $Y_{S,T}$. Thus, we have a decomposition
of $a_{l_{1},l_{2}*}\Lambda[-l_{1}-l_{2}]$ in terms of simple objects
in the category of $-\dim Y$-shifted perverse sheaves. It is easy
to check that $\mathcal{P}_{k}$ and $\mathcal{R}_{k-j}$ for $k\geq j\geq1$
have no simple factors in common, so any map $\mathcal{P}_{k}\to\mathcal{R}_{k-j}$
must vanish. \end{proof}
\begin{rem}
The same techniques used in Section 4.2 apply in order to completely
determine the graded pieces of $(\ker N^{j}/\ker N^{j-1})/\mbox{(\mbox{im} }N\cap\ker N^{j})$
induced by the filtration of $R\psi\Lambda$ by $\tau_{\leq k}R\psi\Lambda$.
The only tricky part is seeing that we can also identify the $k$th
graded piece of $\mbox{im}N$ with \[
\mbox{im}(\bar{N}:R^{k+1}\psi\Lambda[-k-1]\to R^{k}\psi\Lambda[-k]),\]
but this can be proved in the same way as the corresponding statement
about the kernels of $N$ and $\bar{N}$. We get a complete description
of the graded pieces of $(\ker N^{j}/\ker N^{j-1})/\mbox{\mbox{im} }N$.\end{rem}
\begin{lem}
For $1\leq j\leq2n-2$, the filtration of $R\psi\Lambda$ by $\tau_{\leq k}R\psi\Lambda$
induces a filtration on $(\ker N^{j}/\ker N^{j-1})/\mbox{im}N$. For
$0\leq k\leq n-1-\frac{j-1}{2},$ the $(k+1)$-st graded piece of
this filtration is isomorphic to \[
\bigoplus_{i=1}^{j}a_{k+i-1,k+j-i*}\Lambda(-(k+j-1))[-2k-j+1].\]

\end{lem}
Let \[
\mbox{Gr}^{q}\mbox{Gr}_{p}R\psi\Lambda=(\ker N^{p}\cap\mbox{im}N^{q})/(\ker N^{p-1}\cap\mbox{im}N^{q})+(\ker N^{p}\cap\mbox{im}N^{q+1}).\]
The monodromy filtration $M_{r}R\psi\Lambda$ has graded pieces $\mbox{Gr}_{r}^{M}R\psi\Lambda$
isomorphic to \[
\mbox{Gr}_{r}^{M}R\psi\Lambda\simeq\bigoplus_{p-q=r}\mbox{Gr}^{q}\mbox{Gr}_{p}R\psi\Lambda,\]
and if we understand the cohomology of $Y_{\mathbb{\bar{F}}}$ with
coefficients in each $\mbox{Gr}^{q}\mbox{Gr}_{p}R\psi\Lambda$ we
can compute the cohomology of $Y_{\bar{\mathbb{F}}}$ with respect
to $R\psi\Lambda$. The next result tells us how to compute $H^{m}(Y_{\bar{\mathbb{F}}},\mbox{Gr}^{q}\mbox{Gr}_{p}R\psi\Lambda)$. 
\begin{cor}
\label{Spectral sequence}There is a spectral sequence \[
E_{1}^{k+1,m-k-1}=\bigoplus_{i=1}^{p+q}H^{m}(Y_{\bar{\mathbb{F}}},a_{k+i-1,k+p+q-i*}\Lambda(-(k+p-1))[-2k-p-q+1])\]
\[
\Rightarrow H^{m}(Y_{\bar{\mathbb{F}}},\mbox{Gr}^{q}\mbox{Gr}_{p}R\psi\Lambda)\]
 compatible with the action of $G_{\mathbb{F}}$. This can be rewritten
as \[
E_{1}^{k+1,m-k-1}=\bigoplus_{i=1}^{p+q}H^{m-2k-p-q+1}(Y_{\bar{\mathbb{F}}}^{(k+i-1,k+p+q-i)},\Lambda(-(k+p-1)))\]
\[
\Rightarrow H^{m}(Y_{\bar{\mathbb{F}}},\mbox{Gr}^{q}\mbox{Gr}_{p}R\psi\Lambda).\]
\end{cor}
\begin{rem}
The construction of the above spectral sequence is functorial with
respect to etale morphisms which preserve the stratification by $Y_{S,T}$
with $S,T\subset\{1,\dots,n\}$. The reason for this is that etale
morphisms preserve both the kernel and the image filtration of $N$
as well as the canonical filtration $\tau_{\leq k}R\psi\Lambda$. 
\end{rem}

\section{The cohomology of closed strata}

\subsection{Igusa varieties}

Let $q=p^{[\mathbb{F}:\mathbb{F}_{p}]}$. Fix $0\leq h_{1},h_{2}\leq n-1$
and consider the stratum $\bar{X}_{U_{0}}^{(h_{1},h_{2})}$ of the
Shimura variety $X_{U_{0}}$. Choose a compatible one-dimensional
formal $\mathcal{O}_{F,\mathfrak{p}_{1}}=\mathcal{O}_{K}$-module
$\Sigma_{1}$, of height $n-h_{1}$ and also a compatible one-dimensional
formal $\mathcal{O}_{F,\mathfrak{p}_{2}}\simeq\mathcal{O}_{K}$-module
$\Sigma_{2}$ of height $n-h_{2}$. Giving $\Sigma_{1}$ and $\Sigma_{2}$
is equivalent to giving a triple $(\Sigma,\lambda_{\Sigma},i_{\Sigma})$
where:
\begin{itemize}
\item $\Sigma$ is a Barsotti-Tate group over $\bar{\mathbb{F}}$. 
\item $\lambda_{\Sigma}:\Sigma\to\Sigma^{\vee}$ is a polarization. 
\item $i_{\Sigma}:\mathcal{O}_{F}\to\mathrm{End}(\Sigma)\otimes_{\mathbb{Z}}\mathbb{Z}_{(p)}$
such that $(\Sigma,i_{\Sigma})$ is compatible. 
\end{itemize}
Note that $(\Sigma[\mathfrak{p}_{i}^{\infty}])^{0}\simeq\Sigma_{i}$
for $i=1,2$ while $(\Sigma[\mathfrak{p}_{i}^{\infty}])^{\mathrm{et}}\simeq(K/\mathcal{O}_{K})^{n-h_{i}}.$

Assume that the level $U_{0}$ corresponds to the vector $(0,0,m_{3},\dots,m_{r})$.
Let \[
\vec{m}=((m_{i}^{0},m_{i}^{\mathrm{et}})_{i=1,2},m_{3},\dots,m_{r}).\]
The Igusa variety $\mbox{Ig}_{U^{p},\vec{m}}^{(h_{1},h_{2})}$ over
$\bar{X}_{U_{0}}^{(h_{1},h_{2})}\times_{\mathbb{F}}\bar{\mathbb{F}}$
is defined to be the moduli space of the set of the following isomorphisms
of finite flat group schemes for $i=1,2$:
\begin{itemize}
\item $\alpha_{i}^{0}:\Sigma_{i}[\mathfrak{p}_{i}^{m_{i}^{0}}]\stackrel{\sim}{\to}\mathcal{G}_{i}^{0}[\mathfrak{p}_{i}^{m_{i}^{0}}]$,
which extends etale locally to any $(m_{i}^{0})'\geq m_{i}^{0}$ and 
\item $\alpha_{i}^{\mbox{et}}:(\mathfrak{p}_{i}^{-m_{i}^{\mathrm{et}}}\mathcal{O}_{F,\mathfrak{p}_{i}}/\mathcal{O}_{F,\mathfrak{p}_{i}})^{h_{i}}\stackrel{\sim}{\to}\mathcal{G}_{i}^{\mathrm{et}}[\mathfrak{p}_{i}^{m_{i}^{\mathrm{et}}}].$
\end{itemize}
In other words, if $S/\bar{\mathbb{F}}$ is a scheme, then an $S$-point
of the Igusa variety $\mbox{Ig}_{U^{p},\vec{m}}^{(h_{1},h_{2})}$
corressponds to a tuple \[
(A,\lambda,i,\eta^{p},(\alpha_{i}^{\mbox{0}})_{i=1,2},(\alpha_{i}^{\mathrm{et}})_{i=1,2},(\alpha_{i})_{i\geq3}),\]
where 
\begin{itemize}
\item $A$ is an abelian scheme over $S$ with $\mathcal{G}_{A,i}=A[\mathfrak{p}_{i}^{\infty}]$;
\item $\lambda:A\to A^{\vee}$ is a prime-to-$p$ polarization;
\item $i:\mathcal{O}_{F}\hookrightarrow\mbox{End}(A)\otimes_{\mathbb{Z}}\mathbb{Z}_{(p)}$
such that $(A,i)$ is compatible and $\lambda\circ i(f)=i(f^{*})^{\vee}\circ\lambda,\forall f\in\mathcal{O}_{F}$;
\item $\bar{\eta}^{p}:V\otimes_{\mathbb{Q}}\mathbb{A}^{\infty,p}\to V^{p}A$
is a $\pi_{1}(S,s)$-invariant $U^{p}$-orbit of isomorphisms of $F\otimes_{\mathbb{Q}}\mathbb{A}^{\infty,p}$-modules,
sending the standard pairing on $V\otimes_{\mathbb{Q}}\mathbb{A}^{\infty,p}$
to an $(\mathbb{A}^{\infty,p})^{\times}$-multiple of the $\lambda$-Weil
pairing;
\item $\alpha_{i}^{0}:\Sigma^{0}[\mathfrak{p}_{i}^{m_{i}}]\stackrel{\sim}{\to}\mathcal{G}_{A,i}^{0}[\mathfrak{p}_{i}^{m_{i}}]$
is an $\mathcal{O}_{K}$-equivariant isomorphism of finite flat group
schemes which extends to any higher level $m'\geq m_{i}$, for $i=1,2$; 
\item $\alpha_{i}^{\mathrm{et}}:\Sigma^{\mathrm{et}}[\mathfrak{p}_{i}^{m_{i}}]\stackrel{\sim}{\to}\mathcal{G}_{A,i}^{\mathrm{et}}[\mathfrak{p}_{i}^{m_{i}}]$
is an $\mathcal{O}_{K}$-equivariant isomorphism of etale group schemes,
for $i=1,2$;
\item $\alpha_{i}:\Sigma[\mathfrak{p}_{i}^{m_{i}}]\stackrel{\sim}{\to}\mathcal{G}_{A,i}[\mathfrak{p}_{i}^{m_{i}}]$
is an $\mathcal{O}_{F,\mathfrak{p}_{i}}$-equivariant isomorphism
of etale group schemes, for $3\leq i\leq r$. 
\end{itemize}
Two such tuples are considered equivalent if there exists a prime-to-$p$
isogeny $f:A\to A'$ taking $(A,\lambda,i,\bar{\eta}^{p},\alpha_{i}^{0},\alpha_{i}^{\mathrm{et}},\alpha_{i})$
to $(A',\gamma\lambda',i',\bar{\eta}^{p'},\alpha_{i}^{0'},\alpha_{i}^{\mathrm{et}'},\alpha'_{i})$
for $\gamma\in\mathbb{Z}_{(p)}^{\times}$. 

The Igusa varieties $\mbox{Ig}_{U^{p},\vec{m}}^{(h_{1},h_{2})}$ form
an inverse system which has an action of $G(\mathbb{A}^{\infty,p})$
inherited from the action on $\bar{X}_{U_{0}}^{(h_{1},h_{2})}$. Let
\[
J^{(h_{1},h_{2})}(\mathbb{Q}_{p})=\mathbb{Q}_{p}^{\times}\times D_{K,n-h_{1}}^{\times}\times GL_{h_{1}}(K)\times D_{K,n-h_{2}}^{\times}\times GL_{h_{2}}(K)\times\prod_{i=3}^{r}GL_{n}(F_{\mathfrak{p}_{i}}),\]
which is the group of quasi-self-isogenies of $(\Sigma,\lambda_{\Sigma},i_{\Sigma})$
(to compute $J^{(h_{1},h_{2})}(\mathbb{Q}_{p})$ we use the duality
induced from the polarization). The automorphisms of $(\Sigma,\lambda_{\Sigma},i_{\Sigma})$
have an action on the right on $\mbox{Ig}_{U^{p},\vec{m}}^{(h_{1},h_{2})}$.
This can be extended to an action of a certain submonoid of $J^{(h_{1},h_{2})}(\mathbb{Q}_{p})$
on the inverse system of $\mbox{Ig}_{U^{p},\vec{m}}^{(h_{1},h_{2})}$
and furthermore to an action of the entire group $J^{(h_{1},h_{2})}(\mathbb{Q}_{p})$
on the directed system $H_{c}^{j}(\mbox{Ig}_{U^{p},\vec{m}}^{(h_{1},h_{2})},\mathcal{L}_{\xi})$.
For a definition of this action, see section 5 of \cite{Shin-1} and
section 4 of \cite{Man}.

We also define an Iwahori-Igusa variety of the first kind $I_{U}^{(h_{1},h_{2})}/\bar{X}_{U_{0}}$
as the moduli space of chains of isogenies for $i=1,2$\[
\mathcal{G}_{i}^{\mathrm{et}}=\mathcal{G}_{i,0}\to\mathcal{G}_{i,1}\to\dots\to\mathcal{G}_{i,h_{i}}=\mathcal{G}_{i}^{\mathrm{et}}/\mathcal{G}_{i}^{\mathrm{et}}[\mathfrak{p}_{i}]\]
of etale Barsotti-Tate $\mathcal{O}_{K}$-modules, each isogeny having
degree $\#\mathbb{F}$ and with composite equal to the natural map
$\mathcal{G}_{i}^{\mathrm{et}}\to\mathcal{G}_{i}^{\mathrm{et}}/\mathcal{G}_{i}^{\mathrm{et}}[\mathfrak{p}_{i}]$.
Then $I_{U}^{(h_{1},h_{2})}$ is finite etale over $\bar{X}_{U_{0}}^{(h_{1,}h_{2})}$
and naturally inherits the action of $G(\mathbb{A}^{\infty,p})$.
Moreover, for $m_{1}=m_{2}=1$ we know that $\mbox{Ig}_{U^{p},\vec{m}}^{(h_{1},h_{2})}/I_{U}^{(h_{1},h_{2})}\times_{\mathbb{F}}\bar{\mathbb{F}}$
is finite etale and Galois with Galois group $B_{h_{1}}(\mathbb{F})\times\mathcal{O}_{D_{K,n-h_{1}}}^{\times}\times B_{h_{2}}(\mathbb{F})\times\mathcal{O}_{D_{K,n-h_{2}}}^{\times}$.
(Here $B_{h_{i}}(\mathbb{F})\subseteq GL_{h_{i}}(\mathbb{F})$ is
the Borel subgroup.) 
\begin{lem}
For $S,T\subset\{1,\dots,n\}$ with $\#S=n-h_{1},\#T=n-h_{2}$ there
exists a finite map of $\bar{X}_{U_{0}}^{(h_{1},h_{2})}$-schemes\[
\varphi:Y_{U,S,T}^{0}\to I_{U}^{(h_{1},h_{2})}\]
which is bijective on the sets of geometric points. \end{lem}
\begin{proof}
The proof is a straightforward generalization of the proof of Lemma
4.1 of \cite{T-Y}. 
\end{proof}
Define\[
H_{c}^{j}(Y_{\mathrm{Iw}(m),S,T}^{0},\mathcal{L}_{\xi})=\lim_{\substack{\longrightarrow\\
U^{p}}
}H_{c}^{j}(Y_{U,S,T}^{0},\mathcal{L}_{\xi}),\]

\[
H^{j}(Y_{\mathrm{Iw}(m),S,T},\mathcal{L}_{\xi})=\lim_{\substack{\longrightarrow\\
U^{p}}
}H^{j}(Y_{U,S,T},\mathcal{L}_{\xi}),\]

\[
H_{c}^{j}(\mbox{Ig}^{(h_{1},h_{2})},\mathcal{L}_{\xi})=\lim_{\substack{\longrightarrow\\
U^{p},\vec{m}}
}H_{c}^{j}(\mbox{Ig}_{U^{p},\vec{m}}^{(h_{1},h_{2})},\mathcal{L}_{\xi})\mbox{ and }\]
\[
H_{c}^{j}(I_{\mathrm{Iw}(m)}^{(h_{1},h_{2})},\mathcal{L}_{\xi})=\lim_{\substack{\longrightarrow\\
U^{p}}
}H_{c}^{j}(I_{U}^{(h_{1},h_{2})}\times_{\mathbb{F}}\bar{\mathbb{F}},\mathcal{L}_{\xi}).\]

For $m_{1}^{0}=m_{2}^{0}=0$, the Igusa variety $\mbox{Ig}_{U^{p},\vec{m}}^{(h_{1},h_{2})}$
is defined over $\mathbb{F}$. If in addition, $m_{1}^{\mbox{et}}=m_{2}^{\mbox{et}}=1$
then $\mbox{Ig}_{U^{p},\vec{m}}^{(h_{1},h_{2})}$ (over $\mathbb{F}$)
is a Galois cover of $I_{U}^{(h_{1},h_{2})}$ with Galois group $B_{h_{1}}(\mathbb{F})\times B_{h_{2}}(\mathbb{F})$. 
\begin{cor}
For every $S,T\subseteq\{1,\dots n-1\}$ with $\#S=n-h_{1},\#T=n-h_{2}$
and $j\in\mathbb{Z}_{\geq0}$ we have the following isomorphism \[
H_{c}^{j}(Y_{U,S,T}^{0}\times_{\mathbb{F}}\bar{\mathbb{F}},\mathcal{L}_{\xi})\simeq H_{c}^{j}(I_{U^{p},(0,0,1,1,m')}^{(h_{1},h_{2})}\times_{\mathbb{F}}\bar{\mathbb{F}},\mathcal{L}_{\xi})^{B_{h_{1}}(\mathbb{F})\times B_{h_{2}}(\mathbb{F})}.\]

\end{cor}
By taking a direct limit over $U^{p}$ and over $\vec{m}$ and considering
the definitions of the Igusa varieties, we get an isomorphism \[
H_{c}^{j}(Y_{\mbox{Iw},S,T}^{0},\mathcal{L}_{\xi})\simeq H_{c}^{j}(\mbox{Ig}^{(h_{1},h_{2})},\mathcal{L}_{\xi})^{U_{p}^{\mathfrak{p}_{1}\mathfrak{p}_{2}}\times\mbox{Iw}_{h_{1},\mathfrak{p}_{1}}\times\mathcal{O}_{D_{K,n-h_{1}}}^{\times}\times\mbox{Iw}_{h_{2},\mathfrak{p}_{2}}\times\mathcal{O}_{D_{K,n-h_{2}}}^{\times}}.\]
If we let \[
H_{c}^{j}(\mbox{Ig}_{0}^{(h_{1},h_{2})},\mathcal{L}_{\xi})=\lim_{\substack{\longrightarrow\\
U^{p},\vec{m}\\
m_{1}^{0}=m_{2}^{0}=0}
}H_{c}^{j}(\mbox{Ig}_{U^{p},\vec{m}}^{(h_{1},h_{2})}\times_{\mathbb{F}}\bar{\mathbb{F}},\mathcal{L}_{\xi})\]
then the above isomorphism becomes \[
H_{c}^{j}(Y_{\mbox{Iw},S,T}^{0},\mathcal{L}_{\xi})\simeq H_{c}^{j}(\mbox{Ig}_{0}^{(h_{1},h_{2})},\mathcal{L}_{\xi})^{U_{p}^{\mathfrak{p}_{1}\mathfrak{p}_{2}}\times\mbox{Iw}_{h_{1},\mathfrak{p}_{1}}\times\mbox{Iw}_{h_{2},\mathfrak{p}_{2}}}.\]

\begin{prop}
The action of $\mbox{Frob}_{\mathbb{F}}$ on $H_{c}^{j}(\mbox{Ig}_{0}^{(h_{1},h_{2})},\mathcal{L}_{\xi})$
coincides with the action of $(1,(p^{-[\mathbb{F}:\mathbb{F}_{p}]},-1,1,-1,1,1))\in G(\mathbb{A}^{\infty,p})\times J^{(h_{1},h_{2})}(\mathbb{Q}_{p})$.\end{prop}
\begin{proof}
Let $\mbox{Fr}:x\mapsto x^{p}$ be the absolute Frobenius on $\mathbb{F}_{p}$
and let $f=[\mathbb{F}:\mathbb{F}_{p}]$. To compute the action of
the geometric Frobenius $\mbox{Frob}_{\mathbb{F}}$ on $H_{c}^{j}(\mbox{Ig}_{0}^{(h_{1},h_{2})},\mathcal{L}_{\xi})$
we notice that the absolute Frobenius acts on each $H_{c}^{j}(\mbox{Ig}_{U^{p},\vec{m}}^{(h_{1},h_{2})}\times_{\mathbb{F}}\bar{\mathbb{F}},\mathcal{L}_{\xi})$
as $(\mbox{Fr}^{*}){}^{f}\times(\mbox{Frob}_{\mathbb{F}}^{*})^{-1}$.
However, the absolute Frobenius acts trivially on etale cohomology,
so the action of $\mbox{Frob}_{\mathbb{F}}$ coincides with the action
induced from \[
(\mbox{Fr}^{*})^{f}:\mbox{Ig}_{U^{p},\vec{m}}^{(h_{1},h_{2})}\to\mbox{Ig}_{U^{p},\vec{m}}^{(h_{1},h_{2})}\]
We claim that $(\mbox{Fr}^{*})^{f}$ acts the same as the element
$(1,p^{-[\mathbb{F}:\mathbb{F}_{p}]},-1,1,-1,1,1)$ of \[
G(\mathbb{A}^{\infty,p})\times\mathbb{Q}_{p}^{\times}/\mathbb{Z}_{p}^{\times}\times\mathbb{Z}\times GL_{h_{1}}(K)\times\mathbb{Z}\times GL_{h_{2}}(K)\times\prod_{i=3}^{r}GL_{n}(F_{\mathfrak{p}_{i}}),\]
where the two copies of $\mathbb{Z}$ are identified with $D_{K,n-h_{i}}^{\times}/\mathcal{O}_{D_{K,n-h_{i}}}^{\times}$
for $i=1,2$ via the valuation of the determinant. To verify this
claim, we will use the explicit description of the action of a submonoid
$J^{(h_{1},h_{2})}(\mathbb{Q}_{p})$ on the inverse system of Igusa
varieties $\mbox{Ig}_{U^{p},\vec{m}}^{(h_{1},h_{2})}$ found in {[}Man]
which generalizes that on page 122 of \cite{H-T}. First, it is easy
to see that $ $\[
(\mbox{Fr}^{*})^{f}:(A,\lambda,i,\bar{\eta}^{p},\alpha_{i}^{0},\alpha_{i}^{\mbox{et}},\alpha_{i})\mapsto(A^{(q)},\lambda^{(q)},i^{(q)},(\bar{\eta}^{p})^{(q)},(\alpha_{i}^{0})^{(q)},(\alpha_{i}^{\mbox{et}})^{(q)},\alpha_{i}^{(q)})\]
where $F^{f}:A\to A^{(q)}$ is the natural map and the structures
of $A^{(q)}$ are inherited from the structures of $A$ via $F^{f}$. 

On the other hand, the element $j=(1,p^{-[\mathbb{F}:\mathbb{F}_{p}]},-1,1,-1,1,1)$
acts via a quasi-isogeny of $\Sigma$. One can check that inverse
of the quasi-isogeny defined by $j$ is $j^{-1}:\Sigma\to\Sigma^{(q)}$,
which is a genuine isogeny. If we were working with points of $\mbox{Ig}^{(h_{1},h_{2})}$
(which are compatible systems of points of $\mbox{Ig}_{U^{p},\vec{m}}^{(h_{1},h_{2})}$
for all $U^{p}$ and $\vec{m}$) then $j$ should act by precomposing
all the isomorphism $\alpha_{i}^{0},\alpha_{i}^{\mbox{et}}$ for $i=1,2$
and $\alpha_{i}$ for $3\leq i\leq r$. Since $j|_{A[\mathfrak{p}_{i}^{\infty}]^{\mbox{et}}}=1$
for $i=1,2$ and $j|_{A[\mathfrak{p}_{i}^{\infty}]}=1$ for $3\leq i\leq r$
the isomorphisms $\alpha_{i}^{\mbox{et}}$ and $\alpha_{i}$ stay
the same. However, $\alpha_{i}^{0}\circ j$ is now only a quasi-isogeny
of Barsotti-Tate $\mathcal{O}_{K}$-modules and we need to change
the abelian variety $A$ by an isogeny in order to get back the isomorphisms.
Let $j_{i}=j|_{\Sigma[\mathfrak{p}_{i}^{\infty}]^{0}}$ for $i=1,2$.
Then $(j_{i})^{-1}:\Sigma[\mathfrak{p}_{i}^{\infty}]^{0}\to\Sigma[\mathfrak{p}_{i}^{\infty}]^{0}$
is a genuine isogeny induced by the action of $\pi_{i}\in D_{K,n-h_{i}}^{\times}$.
Let $\mathcal{K}_{i}\subset A[\mathfrak{p}_{i}^{[\mathbb{F}:\mathbb{F}_{p}]}]$
be the finite flat subgroup scheme $\alpha_{i}^{0}(\ker(j_{i})^{-1})$.
Let $\mathcal{K}=\mathcal{K}_{1}\oplus\mathcal{K}_{2}\subset A[u^{[\mathbb{F}:\mathbb{F}_{p}]}]$.
Let $\mathcal{K}^{\bot}\subset A[(u^{c})^{[\mathbb{F}:\mathbb{F}_{p}]}]$
be the annihilator of $\mathcal{K}$ under the $\lambda$-Weil pairing.
Let $\tilde{A}=A/\mathcal{K}\oplus\mathcal{K}^{\bot}$ and $f:A\to\tilde{A}$
be the natural projection map. Then \[
\beta_{i}^{0}=f\circ\alpha_{i}^{0}\circ j_{i}:(\Sigma[\mathfrak{p}_{i}^{\infty}])^{0}\to\tilde{A}[\mathfrak{p}_{i}^{\infty}]^{0}\]
is an isomorphism. The quotient abelian variety $\tilde{A}$ inherits
the structures of $A$ through the natural projection and it is easy
to see that $\tilde{A}=A^{(q)}$. Thus, the action of $j$ coincides
with the action of $(\mbox{Fr}^{*})^{f}$. This concludes the proof. \end{proof}
\begin{cor}
\label{cohomology of open strata}We have an isomorphism of admissible
$G(\mathbb{A}^{\infty,p})\times(\mbox{Frob}_{\mathbb{F}})^{\mathbb{Z}}$-modules\[
H_{c}^{j}(Y_{\mbox{Iw}(m),S,T}^{0},\mathcal{L}_{\xi})\simeq H_{c}^{j}(\mbox{Ig}_{0}^{(h_{1},h_{2})},\mathcal{L}_{\xi})^{U_{p}^{\mathfrak{p}_{1}\mathfrak{p}_{2}}\times\mbox{Iw}_{h_{1},\mathfrak{p}_{1}}\times\mbox{Iw}_{h_{2},\mathfrak{p}_{2}}},\]
where $\mbox{Frob}_{\mathbb{F}}$ acts as $(p^{-f},-1,1,-1,1,1)\in J^{(h_{1},h_{2})}(\mathbb{Q}_{p})$. 
\end{cor}

\subsection{Counting points on Igusa varieties}

We wish to apply the trace formula in order to compute the cohomology
of Igusa varieties. A key input of this is counting the $\bar{\mathbb{F}}-$points
of Igusa varieties. Most of this is worked out in \cite{Shin-1}.
The only missing ingredient is supplied by the main lemma in this
section, which is an analogue of Lemma V.4.1 in \cite{H-T} and of
{}``the vanishing of the Kottwitz invariant''. The $\bar{\mathbb{F}}$-points
of Igusa varieties are counted by counting $p$-adic types and other
data (e.g polarizations and level structure). We can keep track of
$p$-adic types via Honda-Tate theory; we need to check that these
$p$-adic types actually correspond to a point on one of our Igusa
varieties. 

A simple $p$-adic type over $F$ is a triple $(M,\eta,\kappa)$ where
\begin{itemize}
\item $M$ is a CM field, with $\mathfrak{P}$ being the set of places of
$M$ over $p$, 
\item $\eta=\sum_{x\in\mathfrak{P}}\eta_{x}x$ is an element of $\mathbb{Q}[\mathfrak{P}]$,
the $\mathbb{Q}$-vector space with basis $\mathfrak{P}$,
\item $\kappa:F\to M$ is a $\mathbb{Q}$-algebra homomorphism 
\end{itemize}
such that $\eta_{x}\geq0$ for all $x\in\mathfrak{P}$ and $\eta+c_{*}\eta=\sum_{x\in\mathfrak{P}}x(p)\cdot x$
in $\mathbb{Q}[\mathfrak{P}]$. Here $c$ is the complex conjugation
on $M$ and \[
c_{*}:\mathbb{Q}[\mathfrak{P}]\to\mathbb{Q}[\mathfrak{P}]\]
 is the $\mathbb{Q}$-linear map satisfying $x\mapsto$. See page
24 of \cite{Shin-1} for the general definition of a $p$-adic type.
As in \cite{Shin-1}, we will drop $\kappa$ from the notation, since
it is well understood as the $F$-algebra structure map of $M$. 

We can recover a simple $p$-adic type from the following data: 
\begin{itemize}
\item a CM field $M/F$; 
\item places $\tilde{\mathfrak{p}}_{i}$ of $M$ above $\mathfrak{p}_{i}$
such that $[M_{\tilde{\mathfrak{p}}_{i}}:F_{\mathfrak{p}_{i}}]n=[M:F](n-h_{i})$
for $i=1,2$ such that there is no intermediate field $F\subset N\subset M$
with $\tilde{\mathfrak{p}}_{i}|_{N}$ both inert in $M$. 
\end{itemize}
Using this data, we can define a simple $p$-adic type $(M,\eta)$,
where the coefficients of $\eta$ at places above $u$ are non-zero
only for $\tilde{\mathfrak{p}}_{1}$ and $\tilde{\mathfrak{p}}_{2}$.
The abelian variety $A/\bar{\mathbb{F}}$ corresponding to $(M,\eta)$
will have an action of $M$ via $i:M\hookrightarrow\mbox{End}^{0}(A)$.
By Honda-Tate theory, the pair $(A,i)$ will also satisfy
\begin{itemize}
\item $M$ is the center of $\mbox{End}_{F}^{0}(A)$,
\item $A[\mathfrak{p}_{i}^{\infty}]^{0}=A[\tilde{\mathfrak{p}}_{i}^{\infty}]$
has dimension $1$ and $A[\mathfrak{p}_{i}^{\infty}]^{e}$ has height
$h_{i}$ for $i=1,2$,
\item and $A[\mathfrak{p}_{i}^{\infty}]$ is ind-etale for $i>2$.\end{itemize}
\begin{lem}
\label{vanishing of kottwitz invariant}Let $M/F$ be a CM field as
above. Let $A/\bar{\mathbb{F}}$ be the corresponding abelian variety
equipped with $i:M\hookrightarrow\mbox{End}^{0}(A)$. Then we can
find 
\begin{itemize}
\item a polarization $\lambda_{0}:A\to A^{\vee}$ for which the Rosati involution
induces $c$ on $i(M)$, and 
\item a finitely-generated $M$-module $W_{0}$ together with a non-degenerate
Hermitian pairing \[
\langle\cdot,\cdot\rangle_{0}:W_{0}\times W_{0}\to\mathbb{Q}\]

\end{itemize}
such that the following are satisfied:
\begin{itemize}
\item there is an isomorphism of $M\otimes\mathbb{A}^{\infty,p}$-modules
\[
W_{0}\otimes\mathbb{A}^{\infty,p}\stackrel{\sim}{\to}V^{p}A\]
which takes $\langle\cdot,\cdot\rangle_{0}$ to an $(\mathbb{A}^{\infty,p})^{\times}$-multiple
of the $\lambda_{0}$ -Weil pairing on $V^{p}A$, and
\item there is an isomorphism of $F\otimes_{\mathbb{Q}}\mathbb{R}$-modules
\[
W_{0}\otimes_{\mathbb{Q}}R\stackrel{\sim}{\to}V\otimes_{\mathbb{Q}}\mathbb{R}\]
 which takes $\langle\cdot,\cdot\rangle_{0}$ to an $\mathbb{R}^{\times}$-multiple
of our standard pairing $\langle\cdot,\cdot\rangle$ on $V\otimes_{\mathbb{Q}}\mathbb{R}$. 
\end{itemize}
\end{lem}
\begin{proof}
By Lemma 9.2 of \cite{Kottwitz} there is a polarization $\lambda_{0}:A\to A^{\vee}$
such that the $\lambda_{0}$-Rosati involution preserves $M$ and
acts on it as $c$. The next step is to show that, up to isogeny,
we can lift $(A,i,\lambda_{0})$ from $\bar{\mathbb{F}}$ to $\mathcal{O}_{K^{ac}}$.
Using the results of \cite{Tat} we can find some lift of $A$ to
an abelian scheme $\tilde{A}/\mathcal{O}_{K^{ac}}$ in such a way
that $i$ lifts to an action $\tilde{i}$ of $M$ on $\tilde{A}$.
As in the proof of Lemma V.4.1 of \cite{H-T} we find a polarization
$\tilde{\lambda}$ of $\tilde{A}$ which reduces to $\lambda$. However,
we want to be more specific about choosing our lift $\tilde{A}$.
Indeed, for any lift, $\mbox{Lie}\tilde{A}\otimes_{\mathcal{O}_{K}^{ac}}K^{ac}$
is an $F\otimes K^{ac}\simeq(K^{ac})^{\mathrm{Hom}(F,K^{ac})}$-module,
so we have a decomposition \[
\mbox{Lie}\tilde{A}\otimes_{\mathcal{O}_{K^{ac}}}K^{ac}\simeq\bigoplus_{\tau\in\mathrm{Hom}(F,K^{ac})}(\mbox{Lie}\tilde{A})_{\tau}.\]
Let $\mathrm{Hom}(F,K^{ac})^{+}$ be the the set of places $\tau\in\mathrm{Hom}(F,K^{ac})$
which induce the place $u$ of $E$. We want to make sure that the
set of places $\tau\in\mathrm{Hom}(F,K^{ac})^{+}$ for which $(\mbox{Lie}\tilde{A})_{\tau}$
is non-trivial has exactly two elements $\tau'_{1}$ and $\tau'_{2}$
which differ by our distinguished element $\sigma\in\mbox{Gal}(F/\mathbb{Q})$,
i.e.\[
\tau'_{2}=\tau'_{1}\circ\sigma.\]
In order to ensure this, we need to go through Tate's original argument
for constructing lifts $\tilde{A}$ of $A$. 

First, let $\Phi=\sum_{\tilde{\tau}\in\mathrm{Hom}(M,K^{ac})}\Phi_{\tilde{\tau}}\cdot\tilde{\tau}$
with the $\Phi_{\tilde{\tau}}$ non-negative integers satisfying $\Phi_{\tilde{\tau}}+\Phi_{\tilde{\tau}^{c}}=n$.
For any such $\Phi$, we can construct an abelian variety $\tilde{A}_{\Phi}$
over $\mathcal{O}_{K^{ac}}$ such that \[
\mbox{Lie}\tilde{A}_{\Phi}\otimes_{\mathcal{O}_{K^{ac}}}K^{ac}\simeq\bigoplus_{\tau\in\mathrm{Hom}(F,K^{ac})}(\mbox{Lie}\tilde{A}_{\Phi})_{\tau}\]
satisfies $\dim(\mbox{Lie}\tilde{A}_{\Phi})_{\tau}=\Phi_{\tau}$.
This is done as in Lemma 4 of \cite{Tat}, which proves the case $n=1$.
We pick any $\tau_{i}'\in\mbox{Hom}(F,K^{ac})$ inducing the places
$\mathfrak{p}_{i}$ of $F$ for $i=1,2$ such that $\tau'_{2}=\tau'_{1}\circ\sigma$.
We lift the $\tau'_{i}$ to elements $\tilde{\tau}_{i}\in\mbox{Hom}(M,K^{ac})$
inducing $\tilde{\mathfrak{p}}_{i}$. We let $\Phi_{\tilde{\tau}_{i}}=1$
and $\Phi_{\tilde{\tau}}=0$ for any other $\tilde{\tau}\in\mbox{Hom}(M,K^{ac})^{+}$.
For $\tilde{\tau}\not\in\mbox{Hom}(M,K^{ac})^{+}$ we define $\Phi_{\tilde{\tau}}=n-\Phi_{\tilde{\tau}^{c}}$.
This determines $\Phi\in\mathbb{Q}[\mbox{Hom}(M,K^{ac})]$ entirely.
This $\Phi$ is not quite a $p$-adic type for $M$, however it is
easy to associate a $p$-adic type to it: we define $\eta=\sum_{x|p}\eta_{x}\cdot x$
by \[
\eta_{x}=\frac{e_{x/p}\cdot[M:F]}{n\cdot[M_{x}:\mathbb{Q}_{p}]}\cdot\sum\Phi_{\tilde{\tau}},\]
where the sum is over embeddings $\tilde{\tau}\in\mbox{Hom}(M,K^{ac})$
which induce the place $x$ of $M$. By Honda-Tate theory, the reduction
of the abelian scheme $\tilde{A}_{\Phi}/\mathcal{O}_{K^{ac}}$ associated
to $\Phi$ has $p$-adic type $\eta$. Indeed, the height of the $p$-divisible
group at $x$ of the reduction of $\tilde{A}_{\Phi}$ is $\frac{n\cdot[M_{x}:\mathbb{Q}_{p}]}{[M:F]}$
(see Proposition 8.4 of \cite{Shin-1} together with an expression
of $\dim A$ in terms of $M$). The dimension of the $p$-divisible
group at $x$ of the reduction is $\sum\Phi_{\tilde{\tau}}$, where
we're summing over all embeddings $\tilde{\tau}$ which induce $x$.

Now we set $\tilde{A}=\tilde{A}_{\Phi}$. It remains to check that
$\tilde{A}/\mathcal{O}_{K^{ac}}$ has special fiber isogenous to $A/\bar{\mathbb{F}}$
and this follows from the fact that the reduction of $\tilde{A}$
and $A$ are both associated to the same $p$-adic type $\eta$. Indeed,
it suffices to verify this for places $x$ above $u.$ We have \[
\eta_{x}=0=e_{x/p}\cdot\frac{\dim A[x^{\infty}]}{\mbox{height}A[x^{\infty}]}\]
 for all places $x\not=\tilde{\mathfrak{p}}_{i}$ for $i=1,2$. When
$x=\tilde{\mathfrak{p}}_{i}$ we have \[
\eta_{x}=e_{x/p}\cdot\frac{[M:F]}{[M_{x}:F_{\mathfrak{p}_{i}}]\cdot n\cdot[F_{\mathfrak{p}_{i}}:\mathbb{Q}_{p}]}=e_{x/p}\cdot\frac{1}{(n-h_{i})\cdot[F_{\mathfrak{p}_{i}}:\mathbb{Q}_{p}]}=e_{x/p}\cdot\frac{\dim A[x^{\infty}]}{\mbox{height}A[x^{\infty}]}.\]
Therefore, the $p$-adic type associated to $A$ is also $\eta$. 

There are exactly two distinct embeddings $\tau'_{1},\tau'_{2}\in\mbox{Hom}(F,K^{ac})^{+}$
such that $(\mbox{Lie}\tilde{A})_{\tau}\not=(0)$ only when $\tau=\tau'_{1}$
or $\tau'_{2}$. Moreover, these embeddings are related by $\tau'_{2}=\tau'_{1}\circ\sigma$.
Therefore, we can find an embedding $\kappa:K^{ac}\hookrightarrow\mathbb{C}$
such that $\kappa\circ\tau'_{i}=\tau_{i}$ for $i=1,2$. We set \[
W_{0}=H_{1}((\tilde{A}\times_{\mbox{Spec }\mathcal{O}_{K^{ac}},\kappa}\mbox{Spec }\mathbb{C})(\mathbb{C}),\mathbb{Q}).\]
From here on, the proof proceeds as the proof of Lemma V.4.1 of \cite{H-T}.
\end{proof}

\subsection{Vanishing of cohomology}

Let $\Pi^{1}$ be an automorphic representation of $GL_{1}(\mathbb{A}_{E})\times GL_{n}(\mathbb{A}_{F})$
and assume that $\Pi^{1}$ is cuspidal. Let $\varpi:\mathbb{A}_{E}^{\times}/E^{\times}\to\mathbb{C}$
be any Hecke character such that $\varpi|_{\mathbb{A}^{\times}/\mathbb{Q}^{\times}}$
is the composite of $\mbox{Art}_{\mathbb{Q}}$ and the natural surjective
character $W_{\mathbb{Q}}\twoheadrightarrow\mbox{Gal}(E/\mathbb{Q})\stackrel{\sim}{\to}\{\pm1\}$.

Also assume that $\Pi^{1}$ and $F$ satisfy 
\begin{itemize}
\item $\Pi^{1}\simeq\Pi^{1}\circ\theta$.
\item $\Pi_{\infty}^{1}$ is generic and $\Xi^{1}$-cohomological, for some
irreducible algebraic representation $\Xi^{1}$ of $\mathbb{G}_{n}(\mathbb{C})$,
which is the image of $\iota_{l}\xi$ under the base change from $G_{\mathbb{C}}$
to $\mathbb{G}_{n,\mathbb{C}}$. 
\item $\mbox{Ram}_{F/\mathbb{Q}}\cup\mbox{Ram}_{\mathbb{Q}}(\varpi)\cup\mbox{Ram}_{\mathbb{Q}}(\Pi)\subset\mbox{Spl}_{F/F_{2},\mathbb{Q}}$. 
\end{itemize}
Let $\mathfrak{S}=\mathfrak{S}_{fin}\cup\{\infty\}$ be a finite set
of places of $F$, which contains the places of $F$ above ramified
places of $\mathbb{Q}$ and the places where $\Pi$ is ramified. For
$l\not=p$, let $\iota:\bar{\mathbb{Q}}_{l}\stackrel{\sim}{\to}\mathbb{C}$
and let $\pi_{p}\in\mbox{\mbox{Irr}}_{l}(G(\mathbb{Q}_{p}))$ be such
that $BC(\iota_{l}\pi_{p})\simeq\Pi_{p}$. If we write $\Pi^{1}=\psi\otimes\Pi^{0}$
and $\pi_{p}=\pi_{p,0}\otimes\pi_{\mathfrak{p}_{1}}\otimes\pi_{\mathfrak{p}_{2}}\otimes(\otimes_{i=3}^{r}\pi_{\mathfrak{p}_{i}})$
then $\iota_{l}\pi_{p,0}\simeq\psi_{u}$ and $\iota_{l}\pi_{\mathfrak{p}_{i}}\simeq\Pi_{\mathfrak{p}_{i}}^{0}$
for all $1\leq i\leq r$. Under the identification $F_{\mathfrak{p}_{1}}\simeq F_{\mathfrak{p}_{2}}$,
assume that $\Pi_{\mathfrak{p}_{1}}^{0}\simeq\Pi_{\mathfrak{p}_{2}}^{0}$
(this condition will be satisfied in all our applications, since we
will choose $\Pi^{0}$ to be the base change of some cuspidal automorphic
representation $\Pi$ of $GL_{n}(\mathbb{A}_{F_{1}E})$). 

Define the following elements of $\mbox{Groth}(G(\mathbb{A}^{\infty,p})\times J^{(h_{1},h_{2})}(\mathbb{Q}_{p}))$
(the Grothendieck group of admissible representations):\[
[H_{c}(\mbox{Ig}^{(h_{1},h_{2})},\mathcal{L}_{\xi})]=\sum_{i}(-1)^{h_{1}+h_{2}-i}H_{c}^{i}(\mbox{Ig}^{(h_{1},h_{2})},\mathcal{L}_{\xi})\]

If $R\in\mbox{Groth}(G(\mathbb{A}^{\mathfrak{S}})\times G')$, we
can write $R=\sum_{\pi^{\mathfrak{S}}\otimes\rho}n(\pi^{\mathfrak{S}}\otimes\rho)[\pi^{\mathfrak{S}}][\rho]$,
where $\pi^{\mathfrak{S}}$ and $\rho$ run over $\mbox{Irr}_{l}(G(\mathbb{A}^{\mathfrak{S}}))$
and $\mbox{Irr}_{l}(G')$ respectively. We define \[
R\{\pi^{\mathfrak{S}}\}:=\sum_{\rho}n(\pi^{\mathfrak{S}}\otimes\rho)[\rho],\mbox{ }R[\pi^{\mathfrak{S}}]:=\sum_{\rho}n(\pi^{\mathfrak{S}}\otimes\rho)[\pi^{\mathfrak{S}}][\rho].\]
Also define \[
R\{\Pi^{1,\mathfrak{S}}\}:=\sum_{\pi^{\mathfrak{S}}}\{\pi^{\mathfrak{S}}\},R[\Pi^{1,\mathfrak{S}}]:=\sum_{\pi^{\mathfrak{S}}}R[\pi^{\mathfrak{S}}],\]
where each sum runs over $\pi^{\mathfrak{S}}\in\mbox{Irr}_{l}^{\mbox{ur}}(G(\mathbb{A}^{\mathfrak{S}}))$
such that $BC(\iota_{l}\pi^{\mathfrak{S}})\simeq\Pi^{\mathfrak{S}}$. 

Let $\mbox{Red}_{n}^{(h_{1},h_{2})}(\pi_{p})$ be the morphism from
$\mbox{Groth}(G(\mathbb{Q}_{p}))$ to $\mbox{Groth}(J^{(h_{1},h_{2})}(\mathbb{Q}_{p}))$
defined by \[
(-1)^{h_{1}+h_{2}}\pi_{p,0}\otimes\mbox{Red}^{n-h_{1},h_{1}}(\pi_{\mathfrak{p}_{1}})\otimes\mbox{Red}^{n-h_{2},h_{2}}(\pi_{\mathfrak{p}_{2}})\otimes(\otimes_{i>2}\pi_{\mathfrak{p}_{i}}),\]
where \[
\mbox{Red}^{n-h,h}:\mbox{ Groth}(GL_{n}(K))\to\mbox{Groth}(D_{K,\frac{1}{n-h}}^{\times}\times GL_{h}(K))\]
is obtained by composing the normalized Jacquet functor \[
J:\mbox{Groth}(GL_{n}(K))\to\mbox{Groth}(GL_{n-h}(K)\times GL_{h}(K))\]
with the Jacquet-Langlands map\[
LJ:\mbox{Groth}(GL_{n-h}(K))\to\mbox{Groth}(D_{K,\frac{1}{n-h}}^{\times})\]
defined by Badulescu in \cite{Bad}. Assume the following result,
which will be proved in section 6:
\begin{thm}
\label{igusa cohomology}We have the following equality in $\mbox{Groth}(G(\mathbb{A}_{S_{fin}\backslash\{p\}})\times J^{(h_{1},h_{2})}(\mathbb{Q}_{p})$:\[
BC_{\mathfrak{S}_{fin}\backslash\{p\}}(H_{c}(\mbox{Ig}^{(h_{1},h_{2})},\mathcal{L}_{\xi})\{\Pi^{1,\mathfrak{S}}\})\]
\[
=e_{0}(-1)^{h_{1}+h_{2}}C_{G}[\iota_{l}^{-1}\Pi_{\mathfrak{S}_{fin}\backslash\{p\}}^{1}][\mbox{Red}_{n}^{(h_{1},h_{2})}(\pi_{p})],\]
where $C_{G}$ is a positive integer and $e_{0}=\pm1$. 
\end{thm}
Let $S,T\subseteq\{1,\dots,n-1\}$ with $\#S=n-h_{1},\#T=n-h_{2}$.
From Theorem \ref{igusa cohomology} and Corollary \ref{cohomology of open strata}
and we obtain the equality \[
BC^{p}(H_{c}(Y_{\mbox{Iw}(m),S,T}^{0},\mathcal{L}_{\xi})[\Pi^{1,\mathfrak{S}}])\]
\[
=e_{0}C_{G}[\iota_{l}^{-1}\Pi^{\infty,p}][\mbox{Red}^{(h_{1},h_{2})}(\pi_{p,0}\otimes\pi_{\mathfrak{p}_{1}}\otimes\pi_{\mathfrak{p}_{2}})]\cdot\dim[(\otimes_{i=3}^{r}\pi_{\mathfrak{p}_{i}})^{U_{p}^{\mathfrak{p}_{1}\mathfrak{p}_{2}}}]\]
in $\mbox{Groth}(G(\mathbb{A}^{\infty,p})\times(\mbox{Frob}_{\mathbb{F}})^{\mathbb{Z}})$.
The group morphism \[
\mbox{Red}^{(h_{1},h_{2})}:\mbox{Groth}(\mathbb{Q}_{p}^{\times}\times GL_{n}(K)\times GL_{n}(K))\to\mbox{Groth}(\mbox{Frob}_{\mathbb{F}}^{\mathbb{Z}})\]
 is the composite of normalized Jacquet functors\[
J_{i}:\mbox{Groth}(GL_{n}(K))\to\mbox{Groth}(GL_{n-h_{i}}(K)\times GL_{h_{i}}(K))\]
for $i=1,2$ with the map \[
\mbox{Groth}(\mathbb{Q}_{p}^{\times}\times GL_{n-h_{1}}(K)\times GL_{h_{1}}(K)\times GL_{n-h_{2}}(K)\times GL_{h_{2}}(K))\to\mbox{Groth}(\mbox{Frob}_{\mathbb{F}}^{\mathbb{Z}})\]
 which sends $[\alpha_{1}\otimes\beta_{1}\otimes\alpha_{2}\otimes\beta_{2}\otimes\gamma]$
to \[
\sum_{\phi_{1},\phi_{2}}\mbox{\mbox{vol}}(D_{K,n-h_{1}}^{\times}/K^{\times})^{-1}\cdot\mbox{vol}(D_{K,n-h_{2}}^{\times}/K^{\times})\cdot\mbox{tr}\alpha_{1}(\varphi_{\mbox{Sp}_{n-h_{1}}(\phi_{1})})\cdot\mbox{tr}\alpha_{2}(\varphi_{\mbox{Sp}_{n-h_{2}}(\phi_{2})})\cdot\]
\[
\cdot(\mbox{dim}\beta_{1})^{\mbox{Iw}_{h_{1},\mathfrak{p}_{1}}}\cdot(\mbox{dim}\beta_{2})^{\mbox{Iw}_{h_{1},\mathfrak{p}_{2}}}\cdot[\mbox{rec}(\phi_{1}^{-1}\phi_{2}^{-1}|\mbox{\ }|^{1-n}(\gamma^{\mathbb{Z}_{p}^{\times}}\circ\mathbf{N}_{K/E_{u}})^{-1})].\]

\begin{lem}
\label{inclusion-exclusion}We have the following equality in $\mbox{Groth}(G(\mathbb{A}^{\infty,p})\times(\mbox{\mbox{Frob}}_{\mathbb{F}})^{\mathbb{Z}})$:\[
BC^{p}(H(Y_{\mbox{Iw}(m),S,T},\mathcal{L}_{\xi})[\Pi^{1,\mathfrak{S}}])=e_{0}C_{G}[\Pi^{1,\infty,p}]\dim[(\otimes_{i=3}^{r}\pi_{\mathfrak{p}_{i}})^{U_{p}^{\mathfrak{p}_{1}\mathfrak{p}_{2}}}]\times\]

\[
\left(\sum_{h_{1}=0}^{n-\#S}\sum_{h_{2}=0}^{n-\#T}(-1)^{2n-\#S-\#T-h_{1}-h_{2}}\left(\begin{array}{c}
n-\#S\\
h_{1}\end{array}\right)\left(\begin{array}{c}
n-\#T\\
h_{2}\end{array}\right)\cdot\right.\]

\[
\left.\mbox{Red}^{(h_{1},h_{2})}(\pi_{p,0}\otimes\pi_{\mathfrak{p}_{1}}\otimes\pi_{\mathfrak{p}_{2})}\right).\]
\end{lem}
\begin{proof}
The proof is a straightforward generalization of the proof of Lemma
4.3 of \cite{T-Y}.\end{proof}
\begin{thm}
\label{temperedness101}Assume that $\Pi_{\mathfrak{p}_{1}}^{0}\simeq\Pi_{\mathfrak{p}_{2}}^{0}$
has an Iwahori fixed vector. Then $\Pi_{\mathfrak{p}_{1}}^{0}\simeq\Pi_{\mathfrak{p}_{2}}^{0}$
is tempered. \end{thm}
\begin{proof}
By Corollary VII.2.18 of \cite{H-T}, $\iota_{l}\pi_{\mathfrak{p}_{i}}$
is tempered if and only if, for all $\sigma\in W_{K}$ every eigenvalue
$\alpha$ of $\mathcal{L}_{n,K}(\Pi_{\mathfrak{p}_{i}}^{0})(\sigma)$
(where $\mathcal{L}_{n,K}(\Pi_{\mathfrak{p}_{i}}^{0})$ is the image
of $\Pi_{\mathfrak{p}_{i}}^{0}$ under the local Langlands correspondence,
normalized as in \cite{Shin}) satisfies \[
|\iota_{l}\alpha|^{2}\in q^{\mathbb{Z}}.\]
We shall first use a standard argument to show that we can always
ensure that \[
|\iota_{l}\alpha|^{2}\in q^{\frac{1}{2}\mathbb{Z}}\]
and then we will use a classification of irreducible, generic, $\iota$-preunitary
representations of $GL_{n}(K)$ together with the cohomology of Igusa
varieties to show the full result. 

The space $H^{k}(X,\mathcal{L}_{\xi})$ decomposes as a $G(\mathbb{A}^{\infty})$-module
as \[
H^{k}(X,\mathcal{L}_{\xi})=\bigoplus_{\pi^{\infty}}\pi^{\infty}\otimes R_{\xi,l}^{k}(\pi^{\infty}),\]
where $\pi^{\infty}$ runs over $\mbox{Irr}_{l}(G(\mathbb{A}^{\infty}))$
and $R_{\xi.l}^{k}(\pi^{\infty})$ is a finite-dimensional $\mbox{Gal}(\bar{F}/F)$-representation.
Define the $\mbox{Gal}(\bar{F}/F)$-representation \[
\tilde{R}_{l}^{k}(\Pi^{1})=\sum_{\pi^{\infty}}R_{\xi,l}^{k}(\pi^{\infty}),\]
where the sum is over the $\pi^{\infty}\in\mbox{Irr}_{l}(G(\mathbb{A}^{\infty}))$
which are cohomological, unramified outside $\mathfrak{S}_{\mathrm{fin}}$
and such that $BC(\iota_{l}\pi^{\infty})=\Pi^{1,\infty}$. Also define
the element $\tilde{R}_{l}(\Pi^{1})\in\mbox{Groth}(\mbox{Gal}(\bar{F}/F))$
by \[
\tilde{R}_{l}(\Pi^{1})=\sum_{k}(-1)^{k}\tilde{R}_{l}^{k}(\Pi^{1}).\]

We claim that we have the following identity in $\mbox{Groth}(W_{K}):$\[
\tilde{R}_{l}(\Pi^{1})=e_{0}C_{G}\cdot[(\pi_{p,0}\circ\mbox{Art}_{\mathbb{Q}_{p}}^{-1})|_{W_{K}}\otimes\iota_{l}^{-1}\mathcal{L}_{K,n}(\Pi_{\mathfrak{p}_{1}}^{0})\otimes\iota_{l}^{-1}\mathcal{L}_{K,n}(\Pi_{\mathfrak{p}_{2}}^{0})].\]
This can be deduced from results of Kottwitz \cite{Kottwitz-1} or
by combining Theorem \ref{igusa cohomology} with Mantovan's formula
\cite{Man}. 

From the above identity, using the fact that $\Pi_{\mathfrak{p}_{1}}^{0}\simeq\Pi_{\mathfrak{p}_{2}}^{0}$,
we see that $|\iota_{l}(\alpha\beta)|^{2}\in q^{\mathbb{Z}}$ for
any eigenvalues $\alpha,\beta$ of any $\sigma\in W_{K}$, since $\tilde{R}_{l}(\Pi^{1})$
is found in the cohomology of some proper, smooth variety $X_{U}$
over $K$. In particular, we know that $|\iota_{l}\alpha|^{2}\in q^{\frac{1}{2}\mathbb{Z}}$.
Moreover, if one eigenvalue $\alpha$ of $\sigma$ satisfies $|\iota_{l}\alpha|^{2}\in q^{\mathbb{Z}}$
then all other eigenvalues of $\sigma$ would be forced to satisfy
it as well. A result of Tadic (\cite{Tad}, see also Lemma I.3.8 of
\cite{H-T}) says that if $\pi_{\mathfrak{p}_{i}}$ is a generic,
$\iota_{l}$-preunitary representation of $GL_{n}(K)$ with central
character $|\psi_{\pi_{\mathfrak{p}_{i}}}|\equiv1$ then $\pi_{\mathfrak{p}_{i}}$
is isomorphic to \[
\mbox{n-Ind}_{P(K)}^{GL_{n}(K)}(\pi_{1}\times\dots\times\pi_{s}\times\pi'_{1}|\det|^{a_{1}}\times\pi'_{1}|\det|^{-a_{1}}\times\dots\times\pi'_{t}|\det|^{a_{t}}\times\pi'_{t}|\det|^{-a_{t}}).\]
The $\pi_{1},\dots,\pi_{s},\pi_{1}',\dots\pi_{t}'$ are square integrable
representations of smaller linear groups with $|\psi_{\pi_{j}}|\equiv|\psi_{\pi'_{j'}}|\equiv1$
for all $j,j'$. Moreover, we must have $0<a_{j}<\frac{1}{2}$ for
$j=1,\dots,t$. If $s\not=0$ then for any $\sigma\in W_{K}$ there
is an eigenvalue $\alpha$ of $\mathcal{L}_{K,n}(\pi_{\mathfrak{p}_{i}})(\sigma)$
with $|\iota_{l}\alpha|^{2}\in q^{\mathbb{Z}}$, but then this must
happen for all eigenvalues of $\mathcal{L}_{K,n}(\pi_{\mathfrak{p}_{i}})(\sigma).$
So then $t=0$ and $\pi_{\mathfrak{p}_{i}}$ is tempered. If $s=0$
then every eigenvalue $\alpha$ of a lift of Frobenius $\sigma\in W_{K}$
must satisfy \[
|\iota_{l}\alpha|^{2}\in q^{\mathbb{Z}\pm2a_{j}}\]
 for some $j\in1,\dots,t$. Note that each $j$ corresponds to at
least one such eigenvalue $\alpha$, so we must have $a_{j}=\frac{1}{4}$
for all $j=1,\dots,t$. To summarize, $\pi_{\mathfrak{p}_{i}}$ is
either tempered or it is of the form \[
\mbox{n-Ind}_{P(K)}^{GL_{n}(K)}(\pi'_{1}|\det|^{\frac{1}{4}}\times\pi'_{1}|\det|^{-\frac{1}{4}}\times\dots\times\pi'_{t}|\det|^{\frac{1}{4}}\times\pi'_{t}|\det|^{-\frac{1}{4}}).\]

We shall now focus on the second case, in order to get a contradiction.
Since $\pi_{\mathfrak{p}_{i}}$ has an Iwahori fixed vector, each
$\pi_{j}'$ must be equal to $\mbox{Sp}_{s_{j}}(\chi_{j})$, where
$\chi_{j}$ is an unramified character of $K^{\times}$. We can compute
$\mbox{Red}^{(h_{1},h_{2})}(\pi_{p,0}\otimes\pi_{\mathfrak{p}_{1}}\otimes\pi_{\mathfrak{p}_{2}})$
explicitly and compare it to the cohomology of a closed stratum $Y_{\mbox{Iw},S,T}$
via Lemma \ref{inclusion-exclusion}.

We can compute $\mbox{Red}^{(h_{1},h_{2})}(\pi_{p,0}\otimes\pi_{\mathfrak{p}_{1}}\otimes\pi_{\mathfrak{p}_{2}})$
using an analogue of Lemma I.3.9 of \cite{H-T}, which follows as
well from Lemma 2.12 of \cite{BZ}. Indeed, \[
J_{i}\left(\mbox{n-Ind}_{P(K)}^{GL_{n}(K)}(\mbox{Sp}_{s_{1}}(\chi_{1})\cdot|\det|^{\frac{1}{4}}\times\mbox{Sp}_{s_{1}}(\chi_{1})\cdot|\det|^{-\frac{1}{4}}\times\dots\times\mbox{Sp}_{s_{t}}(\chi_{t})\cdot|\det|^{-\frac{1}{4}})\right)\]
 is equal to\[
\sum[\mbox{n-Ind}_{P_{i}'(K)}^{GL_{h_{i}(k)}}((\mbox{Sp}_{l_{1}}(\chi_{1}\otimes|\det|^{s_{1}-l_{1}+\frac{1}{4}})\times\dots\times\mbox{Sp}_{k_{t}}(\chi_{t}\otimes|\det|^{s_{t}-k_{t}-\frac{1}{4}}))]\]
\[
[\mbox{n-Ind}_{P_{i}''(K)}^{GL_{h_{i}(k)}}((\mbox{Sp}_{s_{1}-l_{1}}(\chi_{1}\otimes|\det|^{\frac{1}{4}})\times\dots\times\mbox{Sp}_{s_{t}-k_{t}}(\chi_{t}\otimes|\det|^{-\frac{1}{4}}))],\]
where the sum is over all non-negative integers $l_{j},k_{j}\leq s_{j}$
with $h_{i}=\sum_{j=1}^{t}(l_{j}+k_{j})$. 

Let $V_{j_{1}j_{2}}^{k}=\mbox{rec}\left(\chi_{j_{1}}^{-1}\chi_{j_{2}}^{-1}|\ |^{1-n+\epsilon_{k}}(\psi_{u}\circ\mathbf{N}_{K/E_{u}})^{-1}\right)$,
where \[
\epsilon_{k}=\begin{cases}
-\frac{1}{2}\mbox{ if } & k=1\\
0\mbox{ if } & k=2\\
\frac{1}{2}\mbox{ if } & k=3\end{cases}\]
 After we apply the functor \[
\mbox{Groth}(GL_{n-h_{1}}(K)\times GL_{h_{1}}(K)\times GL_{n-h_{2}}(K)\times GL_{h_{2}}(K)\times\mathbb{Q}_{p}^{\times})\to\mbox{Groth}(\mbox{Frob}_{\mathbb{F}}^{\mathbb{Z}}),\]
we get \[
\mbox{Red}^{(h_{1},h_{2})}(\pi_{p,0}\otimes\pi_{\mathfrak{p}_{1}}\otimes\pi_{\mathfrak{p}_{2}})=\sum_{j_{1},j_{2},k}\gamma_{j_{1}j_{2}}^{(h_{1},h_{2})}([V_{j_{1}j_{2}}^{1}]\oplus2[V_{j_{1}j_{2}}^{2}]\oplus[V_{j_{1}j_{2}}^{3}]),\]
where \[
\gamma_{j_{1},j_{2}}^{(h_{1},h_{2})}=\prod_{i=1}^{2}\dim\left(\mbox{n-Ind}_{P'_{i}(K)}^{GL_{h_{i}}(K)}\left(\mbox{Sp}_{s_{j_{i}}+h_{i}-n}(\chi_{j_{i}}|\ |^{n-h_{i}\pm\frac{1}{4}})\otimes\mbox{Sp}_{s_{j}}(\chi_{j_{i}}|\ |^{\mp\frac{1}{4}})\right.\right.\]
\[
\left.\left.\otimes\bigotimes_{j\not=j_{i}}\mbox{Sp}_{s_{j}}(\chi_{j}|\ |^{\frac{1}{4}})\otimes\bigotimes_{j\not=j_{i}}\mbox{Sp}_{s_{j}}(\chi_{j}|\ |^{-\frac{1}{4}})\right)\right)^{\mbox{Iw}_{h_{i},\mathfrak{p}_{i}}}\]
\[
=\prod_{i=1}^{2}\frac{h_{i}!}{(s_{j_{i}}+h_{i}-n)!s_{j_{i}}!\prod_{j\not=j_{i}}(s_{j}!)^{2}}\]
and where the sum is over the $j_{1},j_{2}$ for which $s_{j_{i}}\geq n-h_{i}$
for $i=1,2$. Here $P_{i}^{'}$ for $i=1,2$ are parabolic subgroups
of $GL_{h_{i}}(K)$. 

Let $D(\Pi^{1})=e_{0}C_{G}[\Pi^{1,\infty,p}]\dim[(\otimes_{i=3}^{r}\pi_{\mathfrak{p}_{i}})^{U_{p}^{\mathfrak{p}_{1}\mathfrak{p}_{2}}}]$.
Then \[
BC^{p}(H(Y_{\mbox{Iw}(m),S,T},\mathcal{L}_{\xi})[\Pi^{1,\mathfrak{S}}])=D(\Pi^{1})\cdot\]
\[
\left(\sum_{h_{1}=0}^{n-\#S}\sum_{h_{2}=0}^{n-\#T}(-1)^{2n-\#S-\#T-h_{1}-h_{2}}\left(\begin{array}{c}
n-\#S\\
h_{1}\end{array}\right)\left(\begin{array}{c}
n-\#T\\
h_{2}\end{array}\right)\cdot\right.\]
\[
\left.\sum_{j_{1},j_{2},k}\gamma_{j_{1}j_{2}}^{(h_{1},h_{2})}([V_{j_{1}j_{2}}^{1}]\oplus2[V_{j_{1}j_{2}}^{2}]\oplus[V_{j_{1}j_{2}}^{3}])\right).\]
We can compute the coefficient of $[V_{j_{1}j_{2}}^{k}]$ in $BC^{p}(H(Y_{\mbox{Iw},S,T},\mathcal{L}_{\xi}))[\Pi^{1,\mathfrak{S}}]$
by summing first over $j_{1},j_{2}$ and then over $h_{1},h_{2}$
going from $n-s_{j_{1}},n-s_{j_{2}}$ to $n-\#S$ and $n-\#T$ respectively.
Note that the coefficient of $[V_{j_{1}j_{2}}^{2}]$ is exactly twice
that of $[V_{j_{1}j_{2}}^{1}]$ and of $[V_{j_{1}j_{2}}^{3}]$. The
sum we get for $[V_{j_{1}j_{2}}^{1}]$ is \[
D(\Pi^{1})\frac{(n-\#S)!(n-\#T)!}{(s_{j_{1}}-\#S)!(s_{j_{2}}-\#T)!s_{j_{1}}!s_{j_{2}}!\prod_{j\not=j_{1}}(s_{j}!)^{2}\prod_{j\not=j_{2}}(s_{j}!)^{2}}\cdot\]
\[
\left(\sum_{h_{1}=n-s_{j_{1}}}^{n-\#S}\sum_{h_{2}=n-s_{j_{2}}}^{n-\#T}(-1)^{2n-\#S-\#T-h_{1}-h_{2}}\left(\begin{array}{c}
s_{j_{1}}-\#S\\
h_{1}+s_{j_{1}}-n\end{array}\right)\left(\begin{array}{c}
s_{j_{2}}-\#T\\
h_{2}+s_{j_{2}-n}\end{array}\right)\right)\]
The sum in parentheses can be decomposed as \[
\left(\sum_{h_{1}=n-s_{j_{1}}}^{n-\#S}(-1)^{n-\#S-h_{1}}\left(\begin{array}{c}
s_{j_{1}}-\#S\\
h_{1}+s_{j_{1}}-n\end{array}\right)\right)\cdot\]
\[
\left(\sum_{h_{2}=n-s_{j_{2}}}^{n-\#T}(-1)^{n-\#T-h_{2}}\left(\begin{array}{c}
s_{j_{2}}-\#T\\
h_{2}+s_{j_{2}}-n\end{array}\right)\right),\]
 which is equal to $0$ unless both $s_{j_{1}}=\#S$ and $s_{j_{2}}=\#T$.
So\[
BC^{p}(H(Y_{\mbox{Iw}(m),S,T},\mathcal{L}_{\xi})[\Pi^{1,\mathfrak{S}}])=D(\Pi^{1})\cdot\sum_{s_{j_{1}}=\#S,s_{j_{2}}=\#T}\frac{(n-\#S)!(n-\#T)!s_{j_{1}}!s_{j_{2}}!}{\prod_{j}(s_{j}!)^{4}}\cdot\]
\[
\left([V_{j_{1}j_{2}}^{1}]+2[V_{j_{1}j_{2}}^{2}]+[V_{j_{1}j_{2}}^{3}]\right).\]
Since each $Y_{U,S,T}$ is proper and smooth, $H^{j}(Y_{\mbox{Iw}(m),S,T},\mathcal{L}_{\xi})$
is strictly pure of weight $m_{\xi}-2t_{\xi}+j$. However, the $[V_{j_{1}j_{2}}^{k}]$
are strictly pure of weight $m_{\xi}-2t_{\xi}+2n-2-\epsilon_{k}-(\#S-1)-(\#T-1)=m_{\xi}-2t_{\xi}+2n-\#S-\#T-2\epsilon_{k}$.
So \[
BC^{p}(H^{j}(Y_{\mbox{Iw}(m),S,T},\mathcal{L}_{\xi})[\Pi^{1,\mathfrak{S}}])=0\]
unless $j=2n-\#S-\#T\pm1$ or $j=2n-\#S-\#T$. However, if the Igusa
cohomology is non-zero for some $j=2n-\#S-\#T\pm1$, then there exist
$j_{1},j_{2}$ with $s_{j_{1}}=\#S$ and $s_{j_{2}}=\#T$. Hence,
the cohomology must also be non-zero for $j=2n-\#S-\#T$. The coefficients
of $[V_{j_{1}j_{2}}^{k}]$ all have the same sign, so they are either
strictly positive or strictly negative only depending on $D(\Pi^{1})$.
However, $BC^{p}(H(Y_{\mbox{Iw}(m),S,T},\mathcal{L}_{\xi})[\Pi^{1,\mathfrak{S}}]$
is an alternating sum, so the weight $2n-\#S-\#T\pm1$ part of the
cohomology should appear with a different sign from the weight $2n-\#S-\#T$
part. This is a contradiction, so it must be the case that $\pi_{\mathfrak{p}_{1}}\simeq\pi_{\mathfrak{p}_{2}}$
is tempered. \end{proof}
\begin{cor}
\label{temperedness}Let $n\in\mathbb{Z}_{\geq2}$ be an integer and
$L$ be any CM field. Let $\Pi$ be a cuspidal automorphic representation
of $GL_{n}(\mathbb{A}_{L})$ satisfying 
\begin{itemize}
\item $\Pi^{\vee}\simeq\Pi\circ c$
\item $\Pi_{\infty}$ is cohomological for some irreducible algebraic representation
$\Xi$. 
\end{itemize}
Then $\Pi$ is tempered at every finite place $w$ of $L$. \end{cor}
\begin{proof}
By Lemma 1.4.3 of \cite{T-Y}, an irreducible smooth representation
$\Pi$ of $GL_{n}(K)$ is tempered if and only if $ $$ $$\mathcal{L}_{K,n}(\Pi)$
is pure of some weight. By Lemma 1.4.1 of \cite{T-Y}, purity is preserved
under a restriction to the Weil-Deligne representation of $W_{K'}$
for a finite extension $K'/K$ of fields. 

Fix a place $v$ of $L$ above $p$ where $p\not=l$. We will find
a CM field $F'$ such that
\begin{itemize}
\item $F'=EF_{1},$ where $E$ is an imaginary quadratic field in which
$p$ splits and $F_{1}=(F')^{c=1}$ has $[F_{1}:\mathbb{Q}]\geq2$, 
\item $F'$ is soluble and Galois over $L$,
\item $\Pi_{F'}^{0}=BC_{F'/L}(\Pi)$ is a cuspidal automorphic representation
of $GL_{n}(\mathbb{A}_{F'})$, and 
\item there is a place $\mathfrak{p}$ of $F$ above $v$ such that $\Pi_{F',\mathfrak{p}}^{0}$
has an Iwahori fixed vector, 
\end{itemize}
and a CM field $F$ which is a quadratic extension of $F'$ such that 
\begin{itemize}
\item $\mathfrak{p}=\mathfrak{p}_{1}\mathfrak{p}_{2}$ splits in $F$,
\item $\mbox{Ram}_{F/\mathbb{Q}}\cup\mbox{Ram}_{\mathbb{Q}}(\varpi)\cup\mbox{ Ram}_{\mathbb{Q}}(\Pi)\subset\mbox{Spl}_{F/F',\mathbb{Q}},$
and
\item $\Pi_{F}^{0}=BC_{F/F'}(\Pi_{F'}^{0})$ is a cuspidal automorphic representation
of $GL_{n}(\mathbb{A}_{F}).$
\end{itemize}
To find $F'$ and $F$ we proceed as follows, using the same argument
as on the last page of \cite{Shin}. For a CM field $F$, we shall
use the sets $\mathcal{E}(F)$ and $\mathcal{F}(F)$, which are defined
in the proof of Theorem 7.5 of \cite{Shin}. 

First we find a CM field $F_{0}$ which is soluble and Galois over
$L$ and a place $\mathfrak{p}_{0}$ above $v$ such that the last
two conditions for $F',\mathfrak{p}$ are satisfied for $F_{0},\mathfrak{p}_{0}$
instead. To see that the second to last condition for $F'$ only eliminates
finitely many choices for the CM field we can use the same argument
as Clozel in Section 1 of \cite{Cl2}. Indeed, if $BC_{F'/L}(\Pi)$
is not cuspidal, then we would have $\Pi\otimes\epsilon\simeq\Pi$
for $\epsilon$ the Artin character of $L$ associated to $F'$. But
then the character $\epsilon$ would occur in the semisimplification
of $R_{l}\otimes R_{l}\otimes\omega^{n-1}$, where $R_{l}$ is the
Galois representation associated to $\Pi$ by Chenevier and Harris
in \cite{CH} and $\omega$ is the cyclotomic character. Thus, there
are only finitely choices for $\epsilon$ and so for $F'/L$ which
are excluded. 

Next, we choose $E\in\mathcal{E}(F_{0})$ such that $p$ splits in
$E$. We take $F'=EF_{0}$ and $\mathfrak{p}$ any place of $F'$
above $\mathfrak{p}_{0}$. Let $F_{1}$ be the maximal totally real
subfield of $F'$ and let $w$ be the place of $F_{1}$ below $\mathfrak{p}$.
Next, we pick $F''\in\mathcal{F}(F')$ different from $F'$ and such
that $w$ splits in $F''$. Take $F=F''F'$. 

We can find a character $\psi$ of $\mathbb{A}_{E}^{\times}/E^{\times}$
such that $\Pi^{1}=\psi\otimes\Pi_{F}^{0}$ together with $F$ satisfy
the assumptions in the beginning of the section. (For the specific
conditions that $\psi$ must satsify, see Lemma \ref{conditions on psi}.)
We also know that $\Pi_{F,\mathfrak{p}_{1}}^{0}\simeq\Pi_{F,\mathfrak{p}_{2}}^{0}$
has an Iwahori fixed vector, thus we are in the situation of Theorem
\ref{temperedness101}. \end{proof}
\begin{prop}
\label{vanishing}Assume again that the conditions in the beginning
of this section are satisfied and that $\Pi_{\mathfrak{p}_{1}}^{0}\simeq\Pi_{\mathfrak{p}_{2}}^{0}$
has a nonzero Iwahori fixed vector. Then \[
BC^{p}(H^{j}(Y_{\mbox{Iw}(m),S,T},\mathcal{L}_{\xi})[\Pi^{1,\mathfrak{S}}])=0\]
 unless $j=2n-\#S-\#T$. \end{prop}
\begin{proof}
We will go through the same computation as in the proof of Theorem
\ref{temperedness101} except we will use the fact that $\pi_{\mathfrak{p}_{1}}\simeq\pi_{\mathfrak{p}_{2}}$
is tempered, so it is of the form \[
\mbox{n-Ind}_{P(K)}^{GL_{n}(K)}(\mbox{Sp}_{s_{1}}(\chi_{1})\times\dots\times\mbox{Sp}_{s_{t}}(\chi_{t})),\]
where the $\chi_{j}$ are unramified characters of $K^{\times}$. 

We can compute $\mbox{Red}^{(h_{1},h_{2})}(\pi_{p,0}\otimes\pi_{\mathfrak{p}_{1}}\otimes\pi_{\mathfrak{p}_{2}})$
as in the proof of Theorem \ref{temperedness101}. \[
J_{i}\left(\mbox{n-Ind}_{P(K)}^{GL_{n}(K)}(\mbox{Sp}_{s_{1}}(\chi_{1})\cdot|\det|^{\frac{1}{4}}\times\mbox{Sp}_{s_{1}}(\chi_{1})\cdot|\det|^{-\frac{1}{4}}\times\dots\times\mbox{Sp}_{s_{t}}(\chi_{t})\cdot|\det|^{-\frac{1}{4}})\right)\]
 is equal to\[
\sum[\mbox{n-Ind}_{P_{i}'(K)}^{GL_{h_{i}(k)}}(\mbox{Sp}_{k_{1}}(\chi_{1}\otimes|\det|^{s_{1}-k_{1}})\times\dots\times\mbox{Sp}_{k_{t}}(\chi_{t}\otimes|\det|^{s_{t}-k_{t}}))]\]
\[
[\mbox{n-Ind}_{P_{i}''(K)}^{GL_{h_{i}(k)}}(\mbox{Sp}_{s_{1}-k_{1}}(\chi_{1})\times\dots\times\mbox{Sp}_{s_{t}-k_{t}}(\chi_{t}\otimes|\det|^{-\frac{1}{4}}))],\]
where the sum is over all non-negative integers $k_{j}\leq s_{j}$
with $h_{i}=\sum_{j=1}^{t}k_{j}$. 

Let $V_{j_{1}j_{2}}=\mbox{rec}\left(\chi_{j_{1}}^{-1}\chi_{j_{2}}^{-1}|\ |^{1-n}(\psi_{u}\circ\mathbf{N}_{K/E_{u}})^{-1}\right)$.
After we apply the functor \[
\mbox{Groth}(GL_{n-h_{1}}(K)\times GL_{h_{1}}(K)\times GL_{n-h_{2}}(K)\times GL_{h_{2}}(K)\times\mathbb{Q}_{p}^{\times})\to\mbox{Groth}(\mbox{Frob}_{\mathbb{F}}^{\mathbb{Z}}),\]
we get \[
\mbox{Red}^{(h_{1},h_{2})}(\pi_{p,0}\otimes\pi_{\mathfrak{p}_{1}}\otimes\pi_{\mathfrak{p}_{2}})=\sum_{j_{1},j_{2},k}\gamma_{j_{1}j_{2}}^{(h_{1},h_{2})}[V_{j_{1}j_{2}}]\]
where \[
\gamma_{j_{1},j_{2}}^{(h_{1},h_{2})}=\prod_{i=1}^{2}\dim\left(\mbox{n-Ind}_{P'_{i}(K)}^{GL_{h_{i}}(K)}\left(\mbox{Sp}_{s_{j_{i}}}(\chi_{j_{i}}|\ |^{n-h_{i}})\otimes\bigotimes_{j\not=j_{i}}\mbox{Sp}_{s_{j}}(\chi_{j})\right)\right)^{\mbox{Iw}_{h_{i},\mathfrak{p}_{i}}}\]
\[
=\prod_{i=1}^{2}\frac{h_{i}!}{(s_{j_{i}}+h_{i}-n)!s_{j_{i}}!\prod_{j\not=j_{i}}(s_{j}!)^{2}}\]
and where the sum is over the $j_{1},j_{2}$ for which $s_{j_{i}}\geq n-h_{i}$
for $i=1,2$. Here $P_{i}^{'}$ for $i=1,2$ are parabolic subgroups
of $GL_{h_{i}}(K)$. 

Let $D(\Pi^{1})=e_{0}C_{G}[\Pi^{1,\infty,p}]\dim[(\otimes_{i=3}^{r}\pi_{\mathfrak{p}_{i}})^{U_{p}^{\mathfrak{p}_{1}\mathfrak{p}_{2}}}]$.
The same computation as in the proof of Theorem \ref{temperedness101}
gives us\[
BC^{p}(H(Y_{\mbox{Iw(m)},S,T},\mathcal{L}_{\xi})[\Pi^{1,\mathfrak{S}}])\]
\[
=D(\Pi^{1})\cdot\sum_{s_{j_{1}}=\#S,s_{j_{2}}=\#T}\frac{(n-\#S)!(n-\#T)!s_{j_{1}}!s_{j_{2}}!}{\prod_{j}(s_{j}!)^{2}}[V_{j_{1}j_{2}}].\]
Since $\pi_{\mathfrak{p}_{1}}\simeq\pi_{\mathfrak{p}_{2}}$ is tempered,
we know that $[V_{j_{1}j_{2}}]$ is strictly pure of weight $2n-\#S-\#T$.
The Weil conjectures tell us then that $BC^{p}(H^{j}(Y_{\mbox{Iw}(m),S,T},\mathcal{L}_{\xi})[\Pi^{1,\mathfrak{S}}])=0$
unless $j=2n-\#S-\#T$. 
\end{proof}

\section{The cohomology of Igusa varieties}

The goal of this section is to explain how to prove Theorem \ref{igusa cohomology}.
The proof will be a straightforward generalization of the proof of
Theorem 6.1 of \cite{Shin} and so we will follow closely the argument
and the notation of that paper. 

We will summarize without proof the results in \cite{Shin} on transfer
and on the twisted trace formula. We will emphasize the place $\infty$,
since that is the only place of $\mathbb{Q}$ where our group $G$
differs from the group $G$ considered in \cite{Shin}. 

We start by defining the notation we will be using throughout this
section, which is consistent with the notation of \cite{Shin}. If
$\vec{n}=(n_{i})_{i=1}^{r}$ with $n_{i},r\in\mathbb{Z}_{>0}$ define
\[
GL_{\vec{n}}:=\prod_{i=1}^{r}GL_{n_{i}}.\]
Let $i_{\vec{n}}:GL_{\vec{n}}\hookrightarrow GL_{N}$ ($N=\sum_{i}n_{i})$
be the natural map. Let \[
\Phi_{\vec{n}}=i_{\vec{n}}(\Phi_{n_{1}},\dots,,\Phi_{n_{j}}),\]
 where $\Phi_{n}$ is the matrix in $GL_{n}$ with entries $(\Phi_{n})_{ij}=(-1){}^{i+1}\delta_{i,n+1-j}$. 

Let $K$ be some local non-archimedean local field and $G$ a connected
reductive group over $K$. We will denote by $\mbox{Irr}(G(K))$ (resp.
$\mbox{Irr}_{l}(G(K))$) the set of isomorphism classes of irreducible
admissible representations of $G(K)$ over $\mathbb{C}$ (resp. over
$\bar{\mathbb{Q}}_{l}$). Let $C_{c}^{\infty}(G(K))$ be the space
of smooth compactly supported $\mathbb{C}$-valued functions on $G(K)$.
Let $P$ be a $K$-rational parabolic subgroup of $G$ with a Levi
subgroup $M$. For $\pi_{M}\in\mbox{Irr}(M(K))$ and $\pi\in\mbox{Irr}(G(K))$
we can define the normalized Jacquet module $J_{P}^{G}(\pi)$ and
the normalized parabolic induction $\mbox{n-Ind}_{P}^{G}\pi_{M}$.
We can define a character $\delta_{P}:M(K)\to\mathbb{R}_{>0}^{\times}$
by \[
\delta_{P}(m)=|\det(\mbox{ad}(m))|_{\mbox{Lie}(P)/\mbox{Lie}(M)}|_{K}.\]
We can view $\delta_{P}$ as a character valued in $\bar{\mathbb{Q}}_{l}^{\times}$
via $\iota_{l}^{-1}$. If \[
J^{(h)}(\mathbb{Q}_{p})\simeq D_{K,\frac{1}{n-h}}^{\times}\times GL_{h}(K),\]
where $K/\mathbb{Q}_{p}$ is finite, then we define $\bar{\delta}_{P(J^{(h)})}^{\frac{1}{2}}(g):=\delta_{P_{n-h,h}}^{\frac{1}{2}}(g^{*}),$where
$g^{*}\in GL_{n-h}(K)\times GL_{h}(K)$ is any element whose conjugacy
class matches that of $g$. If \[
J^{(h_{1},h_{2})}\simeq GL_{1}\times\prod_{i=1}^{2}(D_{F_{\mathfrak{p}_{i}},\frac{1}{n-h_{i}}}^{\times}\times GL_{h_{i}})\times\prod_{i>2}R_{F_{\mathfrak{p}_{i}}/\mathbb{Q}_{p}}GL_{n},\]
we define $\bar{\delta}_{P(J^{(h_{1},h_{2}})}^{\frac{1}{2}}:J^{(h_{1},h_{2})}(\mathbb{Q}_{p})\to\bar{\mathbb{Q}}_{l}^{\times}$
to be the product of the characters $\bar{\delta}_{P(J^{(h_{i})})}^{\frac{1}{2}}$
for $i=1,2$. 

Let $\vec{n}=(n_{i})_{i=1}^{r}$ for some $n_{i}\in\mathbb{Z}_{\geq1}$.
Let $G_{\vec{n}}$ be the $\mathbb{Q}$-group defined by \[
G_{\vec{n}}(R)=\{(\lambda,g_{i})\in GL_{1}(R)\times GL_{\vec{n}}(F\otimes_{\mathbb{Q}}R)\mid g_{i}\cdot\Phi_{\vec{n}}\cdot^{t}g_{i}^{c}=\lambda\Phi_{\vec{n}}\}\]
for any $\mathbb{Q}$-algebra $R$. The group $G_{\vec{n}}$ is quasi-split
over $\mathbb{Q}$, so the group $G$ is an inner form of $G_{\vec{n}}$.
Also define \[
\mathbb{G}_{n}=R_{E/\mathbb{Q}}(G_{\vec{n}}\times_{\mathbb{Q}}E).\]
Let $\theta$ denote the action on $\mathbb{G}_{\vec{n}}$ induced
by $(\mbox{id},c)$ on $G_{\vec{n}}\times_{\mathbb{Q}}E$. Let $\epsilon:\mathbb{Z}\to\{0,1\}$
be the unique map such that $\epsilon(n)\equiv n\pmod{2}$. 

Let $\varpi:\mathbb{A}_{E}^{\times}/E^{\times}\to\mathbb{C}^{\times}$
be any Hecke character such that $\varpi|_{\mathbb{A}^{\times}/\mathbb{Q}^{\times}}$
is the composite of $\mbox{Art}_{\mathbb{Q}}$ and the natural surjective
character $W_{\mathbb{Q}}\twoheadrightarrow\mbox{Gal}(E/\mathbb{Q})\stackrel{\sim}{\to}\{\pm1\}$.
Using the Artin map $\mbox{Art}_{E}$, we view $\varpi$ as a character
$W_{E}\to\mathbb{C}^{\times}$ as well. 

Assume that $\mbox{Ram}_{F/\mathbb{Q}}\cup\mbox{Ram}_{\mathbb{Q}}(\varpi)\subset\mbox{Spl}_{F/F_{2},\mathbb{Q}}$. 

Let $\mathcal{E}^{\mathrm{ell}}(G_{n})$ be a set of representatives
of isomorphism classes of endoscopic triples for $G_{n}$ over $\mathbb{Q}$.
Then $\mathcal{E}^{\mathrm{ell}}(G_{n})$ can be identified with the
set of triples \[
\{(G_{n},s_{n},\eta_{n}\}\cup\{G_{n_{1},n_{2}},s_{n_{1},n_{2}},\eta_{n_{1},n_{2}}\mid n_{1}+n_{2}=0,n_{1}\geq n_{2}>0\},\]
where $(n_{1},n_{2})$ may be excluded in some cases. As we are only
interested in the stable part of the cohomology of Igusa varieties,
we will not be concerned with these exclusions. Here $s_{n}=1\in\hat{G}_{n},s_{n_{1},n_{2}}=(1,(I_{n_{1}},-I_{n_{2}}))\in\hat{G}_{n_{1},n_{2}}$,
$\eta_{n}:\hat{G}_{n}\to\hat{G}_{n}$ is the identity map whereas
\[
\eta_{n_{1},n_{2}}:(\lambda,(g_{1},g_{2}))\mapsto\left(\lambda,\left(\begin{array}{cc}
g_{1} & 0\\
0 & g_{2}\end{array}\right)\right).\]
We can extend $\eta_{n_{1},n_{2}}$ to a morphism of $L$-groups,
which sends $z\in W_{E}$ to \[
\left(\varpi(z)^{-N(n_{1},n_{2})},\left(\begin{array}{cc}
\varpi(z)^{\epsilon(n-n_{1})}\cdot I_{n_{1}} & 0\\
0 & \varpi(z)^{\epsilon(n-n_{2})}\cdot I_{n_{2}}\end{array}\right)\right)\rtimes z.\]
We have the following commutative diagram of $L$-morphisms.\[
\xymatrix{^{L}G_{n_{1},n_{2}}\ar[d]_{BC_{n_{1},n_{2}}}\ar[r]\sp-{\tilde{\eta}_{n_{1},n_{2}}} & ^{L}G_{n}\ar[d]^{BC_{n}}\\
^{L}\mathbb{G}_{n_{1},n_{2}}\ar[r]\sb-{\tilde{\zeta}_{n_{1},n_{2}}} & ^{L}\mathbb{G}_{n}}
\]
We will proceed to define local transfers for each of the arrows in
the above commutative diagram so that these transfers are compatible. 

Choose the normalization of the local transfer factor $\Delta_{v}(\ ,\ )_{G_{\vec{n}}}^{G_{n}}$
defined in Section 3.4 of \cite{Shin}. It is possible to give a concrete
description of the $\Delta_{v}(\ ,\ )_{G_{\vec{n}}}^{G_{n}}$-transfer
at finite places $v$ of $\mathbb{Q}$ between functions in $C_{c}^{\infty}(G_{n}(\mathbb{Q}_{v}))$
and functions in $C_{c}^{\infty}(G_{n_{1},n_{2}}(\mathbb{Q}_{v}))$
as long as $v$ satisfies at least one of the conditions: 
\begin{itemize}
\item $v\in\mbox{Unr}_{F/\mathbb{Q}}$ and $v\not\in\mbox{Ram}_{\mathbb{Q}}(\varpi),$ 
\item $v\in\mbox{Spl}_{E/\mathbb{Q}},$
\item $v\in\mbox{Spl}_{F/F_{2},\mathbb{Q}}$ and $v\not\in\mbox{Spl}_{E/\mathbb{Q}}$. 
\end{itemize}
The transfer $\phi_{v}^{n_{1},n_{2}}$ of $\phi_{v}^{n}\in C_{c}^{\infty}(G_{n}(\mathbb{Q}_{v}))$
and $\phi_{v}^{n}$ will satisfy an identity involving orbital integrals.
Since we are assuming that $\mbox{Ram}_{F/\mathbb{Q}}\subseteq\mbox{Spl}_{F/F_{2},\mathbb{Q}}$,
we can define the transfer at all places $v$ of $\mathbb{Q}$. 

It is also possible to define a transfer of pseudo-coefficients at
infinity. Consider $(G_{\vec{n}},s_{\vec{n}},\eta_{\vec{n}})\in\mathcal{E}^{\mathrm{ell}}(G_{n})$,
which is also an endoscopic triple for $G$. Fix real elliptic maximal
tori $T\subset G$ and $T_{G_{\vec{n}}}\subset G_{\vec{n}}$ together
with an $\mathbb{R}$-isomorphism $j:T_{G_{\vec{n}}}\stackrel{\sim}{\to}T.$
Also fix a Borel subgroup $B$ of $G$ over $\mathbb{C}$ containing
$T_{\mathbb{C}}.$ Shelstad defined the transfer factor $\Delta_{j,B}$,
see \cite{She}. 

Let $\xi$ be an irreducible algebraic representation of $G_{\mathbb{C}}$.
Define $\chi_{\xi}:A_{G,\infty}\to\mathbb{C}$ to be the restriction
of $\xi$ to $A_{G,\infty}$ (the connected component of the identity
in the maximal $\mathbb{Q}$-split torus in the center of $G$). Choose
$K_{\infty}\subset G(\mathbb{R})$ to be a maximal compact subgroup
(admissible in the sense of \cite{Art}) and define \[
q(G)=\frac{1}{2}\dim(G(\mathbb{R})/K_{\infty}A_{G,\infty})=2n-2.\]

For each $\pi\in\Pi_{\mathrm{disc}}(G(\mathbb{R}),\xi^{\vee})$ there
exists $\phi_{\pi}\in C_{c}^{\infty}(G(\mathbb{R}),\chi_{\xi})$ a
pseudo-coefficient for $\pi$. Any discrete $L$-parameter $\varphi_{G_{\vec{n}}}$
such that $\tilde{\eta}_{\vec{n}}\varphi_{G_{\vec{n}}}\sim\varphi_{\xi}$
corresponds to an $L$-packet of the form $\Pi_{\mathrm{disc}}(G_{\vec{n}}(\mathbb{R}),\xi(\varphi_{G_{\vec{n}}})^{\vee}).$
Define \[
\phi_{G_{\vec{n}},\xi(\varphi_{G_{\vec{n}}})}:=\frac{1}{|\Pi(\varphi_{\vec{n}})|}\sum_{\pi_{G_{\vec{n}}}}\phi_{\pi_{G_{\vec{n}}}}\mbox{ and}\]
\[
\phi_{\pi}^{G_{\vec{n}}}:=(-1)^{q(G)}\sum_{\tilde{\eta}\varphi_{G_{\vec{n}}}\sim\varphi_{\xi}}\langle a_{\omega_{*}(\varphi_{G_{\vec{n}},\xi})\omega_{\pi}},s\rangle\det(\omega_{*}(\varphi_{G_{\vec{n}},\xi}))\cdot\phi_{G_{\vec{n}},\xi(\varphi_{G_{\vec{n}})}}.\]
Then $\phi_{\pi}^{G_{\vec{n}}}$ is a $\Delta_{j,B}$-transfer of
$\phi_{\pi}$. 

We will now review the base change for the groups $G_{\vec{n}}$ and
$\mathbb{G}_{\vec{n}}$. Define the group \[
\mathbb{G}_{\vec{n}}^{+}:=(R_{E/\mathbb{Q}}GL_{1}\times R_{F/\mathbb{Q}}GL_{\vec{n}})\rtimes\{1,\theta\},\]
where $\theta(\lambda,g)\theta^{-1}=(\lambda^{c},\lambda^{c}g^{\#})$
and $g^{\#}=\Phi_{\vec{n}}\mbox{}^{t}g^{c}\Phi_{\vec{n}}^{-1}$ .
If we denote by $\mathbb{G}_{\vec{n}}^{0}$ and $\mathbb{G}_{\vec{n}}^{0}\theta$
the cosets of $\{1\}$ and $\{\theta\}$ in $\mathbb{G}_{\vec{n}}^{+}$
then $\mathbb{G}_{\vec{n}}^{+}=\mathbb{G}_{\vec{n}}^{0}\coprod\mathbb{G}_{\vec{n}}^{0}\theta$.
There is a natural $\mathbb{Q}$-isomorphism $\mathbb{G}_{\vec{n}}\stackrel{\sim}{\to}\mathbb{G}_{\vec{n}}^{0}$
which extends to \[
\mathbb{G}_{\vec{n}}\rtimes\mbox{Gal}(E/\mathbb{Q})\stackrel{\sim}{\to}\mathbb{G}_{\vec{n}}^{+}\]
 so that $c\in\mbox{Gal}(E/\mathbb{Q})$ maps to $\theta$. 

Let $v$ be a place of $\mathbb{Q}$. A representation $\Pi_{v}\in\mbox{Irr}(\mathbb{G}_{\vec{n}}(\mathbb{Q}_{v}))$
is called $\theta$-stable if $\Pi_{v}\simeq\Pi_{v}\circ\theta$ as
representations of $\mathbb{G}_{\vec{n}}(\mathbb{Q}_{v})$. If that
is the case, then we can choose an operator $A_{\Pi_{v}}$ on the
representation space of $\Pi_{v}$ which induces $\Pi_{v}\stackrel{\sim}{\to}\Pi_{v}\circ\theta$
and which satisfies $A_{\Pi_{v}}^{2}=\mbox{id}.$ Such an operator
is called normalized and it is pinned down up to sign. We can similarly
define the notion of $\theta$-stable for $\Pi^{\mathfrak{S}}\in\mbox{Irr}(\mathbb{G}(\mathbb{A}^{\mathfrak{S}}))$
and a corresponding intertwining operator $A_{\Pi^{\mathfrak{S}}}$
for any finite set $\mathfrak{S}$ of places of $\mathbb{Q}$. There
is a correspondence between $\theta$-stable representations of $\mathbb{G}_{\vec{n}}(\mathbb{Q}_{v})$
together with a normalized intertwining operator and representations
of $\mathbb{G}_{\vec{n}}^{+}(\mathbb{Q}_{v})$. We also mention that
in order for a representation $\Pi\in\mbox{Irr}(\mathbb{G}_{\vec{n}}(\mathbb{A}))$
to be $\theta$-stable it is necessary and sufficient that $\Pi=\psi\otimes\Pi^{1}$
satisfy
\begin{itemize}
\item $(\Pi^{1})^{\vee}\simeq\Pi^{1}\circ c$, and
\item $\prod_{i=1}^{r}\psi_{i}=\psi^{c}/\psi$ where $\psi_{\Pi^{1}}=\psi_{1}\otimes\dots\otimes\psi_{r}$
is the central character of $\Pi^{1}$. 
\end{itemize}
Now we shall discuss BC-matching functions. It is possible to construct
for each finite place $v$ of $\mathbb{Q}$ and $f_{v}\in C_{c}^{\infty}(\mathbb{G}_{\vec{n}}(\mathbb{Q}_{v}))$
a function $\phi_{v}\in C_{c}^{\infty}(G_{\vec{n}}(\mathbb{Q}_{v}))$,
which is the BC-transfer of $f_{v}$. The transfer can be given a
concrete description in the cases $v\in\mbox{Unr}_{F/\mathbb{Q}}$
and $v\in\mbox{Spl}_{F/F_{2},\mathbb{Q}}$, except that in the case
$v\in\mbox{Unr}_{F/\mathbb{Q}}$ we have the condition that $f_{v}$
must be unramified. Moreover, we also have an explicit map \[
BC_{\vec{n}}:\mbox{Irr}^{(\mathrm{ur})}(G_{\vec{n}}(\mathbb{Q}_{v}))\to\mbox{Irr}^{(\mathrm{ur})\theta\mbox{-st}}(\mathbb{G}_{\vec{n}}(\mathbb{Q}_{v}))\]
where the representations must be unramified in the case $v\in\mbox{Unr}_{F/\mathbb{Q}}$
and where there is no restriction in the case $v\in\mbox{Spl}_{F/F_{2},\mathbb{Q}}$.
There are normalized operators $A_{\Pi_{v}}^{0}:\Pi_{v}\stackrel{\sim}{\to}\Pi_{v}\circ\theta$
such that if $\Pi_{v}=BC_{\vec{n}}(\pi_{v})$ and $\phi_{v}$ and
$f_{v}$ are BC-matching functions then \[
\mbox{tr}(\Pi_{v}(f_{v})A_{\Pi_{v}}^{0})=\mbox{tr}\pi_{v}(\phi_{v}).\]
Note that left side of the above equality computes the trace of $f_{v}\theta$,
the function on $\mathbb{G}_{\vec{n}}\theta$ obtained from $f_{v}$
via translation by $\theta$. 

The next step is to consider the base change at $\infty$. Let $\xi_{\vec{n}}$
be an irreducible algebraic representation of $G_{\vec{n},\mathbb{C}}$
. Consider the natural isomorphism \[
\mathbb{G}_{\vec{n}}(\mathbb{C})\simeq G_{\vec{n}}(\mathbb{C})\times G_{\vec{n}}(\mathbb{C}).\]
We can define a representation $\Xi_{\vec{n}}$ of $\mathbb{G}_{\vec{n}}$
by $\Xi_{\vec{n}}:=\xi_{\vec{n}}\otimes\xi_{\vec{n}}$. It is possible
to find an irreducible, $\theta$-stable, generic unitary representation
$\Pi_{\Xi_{\vec{n}}}\in\mbox{Irr}(\mathbb{G}_{\vec{n}}(\mathbb{R}),\chi_{\xi_{\vec{n}}}^{-1})$
together with a normalized operator $A_{\Pi_{\Xi_{\vec{n}}}}^{0}$and
a function $f_{\mathbb{G}_{\vec{n}},\Xi_{\vec{n}}}\in C_{c}^{\infty}(\mathbb{G}(\mathbb{R}),\chi_{\xi_{\vec{n}}})$
such that 
\begin{itemize}
\item $\Pi_{\Xi_{\vec{n}}}$ is the base change of the $L$-packet $\Pi_{\mbox{disc}}(G_{\vec{n}}(\mathbb{R}),\xi^{\vee}),$
\item $\mbox{tr}(\Pi_{\Xi_{\vec{n}}}(f_{\mathbb{G}_{\vec{n}},\Xi_{\vec{n}}})\circ A_{\Pi_{\Xi_{\vec{n}}}}^{0})=2$
and
\item $f_{\mathbb{G}_{\vec{n}},\Xi_{\vec{n}}}$ and $\phi_{G_{\vec{n}},\xi_{\vec{n}}}$
are BC-matching functions (where $\phi_{G_{\vec{n}},\xi_{\vec{n}}}$
is defined as a pseudocoefficient for the $L$-packet $\Pi_{\mbox{disc}}(G_{\vec{n}}(\mathbb{R}),\xi^{\vee})$. 
\end{itemize}
The transfer for $\tilde{\zeta}_{n_{1},n_{2}}$ can be defined explicitly
since the groups $\mathbb{G}_{\vec{n}}$ are essentially products
of general linear groups. It can be checked that for all finite places
$v$ of $\mathbb{Q}$ the transfers are compatible. For $v=\infty$
we have a compatibility relation on the representation-theoretic side
only. 

Now we shall describe the transfer factors $\Delta_{v}(\ ,\ )_{G_{\vec{n}}}^{G}$.
At $v\not=\infty$ we can choose \[
\Delta_{v}(\ ,\ )_{G_{\vec{n}}}^{G}=\Delta_{v}^{0}(\ ,\ )_{G_{\vec{n}}}^{G_{n}},\]
via the fixed isomorphism $G\times_{\mathbb{Q}}\mathbb{A}^{\infty}\simeq G_{n}\times_{\mathbb{Q}}\mathbb{A}^{\infty}$.
We choose the unique $\Delta_{\infty}(\ ,\ )_{G_{\vec{n}}}^{G}$ such
that the product formula \[
\prod_{v}\Delta_{v}(\gamma_{G_{\vec{n}}},\gamma)_{G_{\vec{n}}}^{G}=1\]
holds for any $\gamma\in G(\mathbb{Q})$ semisimple and $\gamma_{G_{\vec{n}}}\in G_{\vec{n}}(\mathbb{A})$
a $(G,G_{\vec{n}})$-regular semisimple element such that $\gamma$
and $\gamma_{G_{\vec{n}}}$ have matching stable conjugacy classes.
Let $e_{\vec{n}}(\Delta_{\infty})\in\mathbb{C}^{\times}$ denote the
constant for which \[
\Delta_{\infty}(\gamma_{G_{\vec{n}}},\gamma)_{G_{\vec{n}}}^{G}=e_{\vec{n}}(\Delta_{\infty})\Delta_{j,B}(\gamma_{G_{\vec{n}}},\gamma)\]
holds. Note that for $\vec{n}=(n)$, $e_{\vec{n}}(\Delta_{\infty})=1$. 

Let $\phi^{\infty,p}\cdot\phi'_{p}\in C_{c}^{\infty}(G(\mathbb{A}^{\infty,p})\times J^{(h_{1},h_{2})}(\mathbb{Q}_{p}))$
be a complex-valued acceptable function. (For a definition of the
notion of acceptable function, see Definition 6.2 of \cite{Shin-1}).
For each $G_{\vec{n}}\in\mathcal{E}^{\mbox{ell}}(G)$ we define the
function $\phi_{\mathrm{Ig}}^{\vec{n}}$ on $G_{\vec{n}}(\mathbb{A})$
(assuming that $\phi^{\infty,p}=\prod_{v\not=p,\infty}\phi_{v}).$
For $v\not=p,\infty$, we take $\phi_{\mathrm{Ig},v}^{\vec{n}}\in C_{c}^{\infty}(G_{\vec{n}}(\mathbb{Q}_{v}))$
to be the $\Delta_{v}(\ ,\ )_{G_{\vec{n}}}^{G}$-transfer of $\phi_{v}$.
We take \[
\phi_{\mathrm{Ig},\infty}^{\vec{n}}:=e_{\vec{n}}(\Delta_{\infty})\cdot(-1)^{q(G)}\langle\mu_{h},s_{\vec{n}}\rangle\sum_{\varphi_{_{\vec{n}}}}\det(\omega_{*}(\varphi_{G_{\vec{n}}}))\cdot\phi_{G_{\vec{n}},\xi(\varphi_{\vec{n}})},\]
where $\varphi_{\vec{n}}$ runs over $L$-parameters such that $\tilde{\eta}_{\vec{n}}\varphi_{\vec{n}}\sim\varphi_{\xi}$
and $\xi(\varphi_{\vec{n}})$ is the algebraic representation of $G_{\vec{n},\mathbb{C}}$
such that the $L$-packet associated to $\varphi_{\vec{n}}$ is $\Pi_{\mathrm{disc}}(G_{\vec{n}}(\mathbb{R}),\xi(\varphi_{\vec{n}})^{\vee})$. 

We also take \[
\phi_{\mathrm{Ig},p}^{\vec{n}}\in C_{c}^{\infty}(G_{\vec{n}}(\mathbb{Q}_{p}))\]
to be the function constructed from $\phi'_{p}$ in section 6.3 of
\cite{Shin-2}. We shall summarize the construction of $\phi_{\mathrm{Ig},p}^{\vec{n}}$
in the case $\vec{n}=(n)$. By definition (see the formula above Lemma
6.5 of \cite{Shin-2}) \[
\phi_{\mathrm{Ig},p}^{\vec{n}}=\sum_{(M_{G_{n},}s_{G_{n},}\eta_{G_{n})}}c_{M_{G_{n}}}\cdot\tilde{\phi}_{p}^{M_{G_{n}}},\]
where the sum is taken over $G$-endoscopic triples for $J^{(h_{1},h_{2})}$.
The set $\mathcal{I}(M_{G_{n}},G_{n})$ (which can be identified with
a set of cosets of $\mbox{Out}(M_{G_{n}},s_{G_{n}},\eta_{G_{n}})$)
consists of only one element in our case, so we suppress the index
$i\in\mathcal{I}(M_{G_{n}},G_{n})$ in $\tilde{\phi}_{p}^{M_{G_{n}},i}.$
Each $\tilde{\phi}_{p}^{M_{G_{n}}}\in C_{c}^{\infty}(G_{n}(\mathbb{Q}_{p}))$
is constructed from a function $\phi_{p}^{M_{G_{n}}}\in C_{c}^{\infty}(M_{G_{n}}(\mathbb{Q}_{p})$
which is a $\Delta_{p}(\ ,\ )_{M_{G_{n}}}^{J^{(h_{1},h_{2})}}$-transfer
of a normalized $\phi'_{p}$.

The following proposition is Theorem 7.2 of \cite{Shin-2}.
\begin{prop}
\label{stable trace formula}If $\phi^{\infty,p}\cdot\phi'_{p}\in C_{c}^{\infty}(G(\mathbb{A}^{\infty,p})\times J^{(h_{1},h_{2})}(\mathbb{Q}_{p}))$
is acceptable, then\[
\mbox{tr}(\phi^{\infty,p}\cdot\phi'_{p}|\iota_{l}H_{c}(\mbox{Ig}^{(h_{1},h_{2})},\mathcal{L}_{\xi}))=(-1)^{h_{1}+h_{2}}|\ker^{1}(\mathbb{Q},G)|\sum_{G_{\vec{n}}}\iota(G,G_{\vec{n}})ST_{e}^{G_{\vec{n}}}(\phi_{\mathrm{Ig}}^{\vec{n}})\]
 where the sum runs over the set $\mathcal{E}^{\mathrm{ell}}(G)$
of elliptic endoscopic triples $(G_{\vec{n}},s_{\vec{n}},\eta_{\vec{n}})$.\end{prop}
\begin{rem}
Theorem 7.2 of \cite{Shin-2} is proved under the {}``unramified
hypothesis'', however, the only place where this hypothesis is needed
is in the proof of Lemma 11.1 of \cite{Shin-1}. Lemma \ref{vanishing of kottwitz invariant}
provides an alternative to the proof of Lemma 11.1 of \cite{Shin-1}
in our situation, so the results of \cite{Shin-1} and \cite{Shin-2}
carry over. For details, see the discussion in the beginning of Section
5.2 of \cite{Shin}. The sign $(-1)^{h_{1}+h_{2}}$ does not show
up in the statement of the theorem in \cite{Shin-2}, but we need
to include it because our convention for the alternating sum of the
cohomology differs from the usual one by $(-1)^{h_{1}+h_{2}}$. 
\end{rem}
The constants $\iota(G,G_{\vec{n}})=\tau(G)\tau(G_{\vec{n}})^{-1}|\mbox{Out}(G_{\vec{n}},s_{\vec{n}},\eta_{\vec{n}})|^{-1}$
can be computed explicitly. We mention that \[
|\mbox{Out}(G_{\vec{n}},s_{\vec{n}},\eta_{\vec{n}})|=\begin{cases}
2 & \mbox{if }\vec{n}=(\frac{n}{2},\frac{n}{2})\\
1 & \mbox{otherwise.}\end{cases}\]
We also have by Corollary 4.7 of \cite{Shin} the relation \[
I_{\mathrm{geom}}^{\mathbb{G}_{\vec{n}}\theta}(f\theta)=\tau(G_{\vec{n}})^{-1}\cdot ST_{e}^{G_{\vec{n}}}(\phi),\]
where $\phi$ and $n$ are BC-matching functions, i.e. \[
\phi=\phi^{\mathfrak{S}}\cdot\phi_{\mathfrak{S}_{\mathrm{fin}}}\cdot\phi_{G_{\vec{n}},\xi}\mbox{ and }f=f^{\mathfrak{S}}\cdot f_{\mathfrak{S}_{\mathrm{fin}}}f_{G_{\vec{n}},\Xi}.\]
 $ $Thus, we can write \[
\mbox{tr}(\phi^{\infty,p}\cdot\phi'_{p}|\iota_{l}H_{c}(\mbox{Ig}^{(h_{1},h_{2})},\mathcal{L}_{\xi}))=|\ker^{1}(\mathbb{Q},G)|\cdot\tau(G)\sum_{\mathbb{G}_{\vec{n}}}\epsilon_{\vec{n}}I_{\mathrm{geom}}^{\mathbb{G}_{\vec{n}}\theta}(f^{\vec{n}}\theta),\]
where $f^{\vec{n}}$ is a BC-matching function for $\phi_{\mathrm{Ig}}^{\vec{n}}$
and $\epsilon_{\vec{n}}=\frac{1}{2}$ if $\vec{n}=(\frac{n}{2},\frac{n}{2})$
or $1$ otherwise. 

Furthermore, the twisted trace formula by Arthur, is an equality between
\[
I_{\mathrm{spec}}^{\mathbb{G}_{\vec{n}}\theta}(f\theta)=I_{\mathrm{geom}}^{\mathbb{G}_{\vec{n}}\theta}(f\theta).\]
 By combining Proposition 4.8 and Corollary 4.14 of \cite{Shin} we
can compute $I_{\mathrm{spec}}^{\mathbb{G}_{\vec{n}}\theta}(f\theta)$
as \[
\sum_{M}\frac{|W_{M}|}{|W_{\mathbb{G}_{\vec{n}}}|}|\det(\Phi_{\vec{n}}^{-1}\theta-1)_{\mathfrak{a}_{M}^{\mathbb{G}_{\vec{n}}\theta}}|^{-1}\sum_{\Pi_{M}}\mbox{tr}(\mbox{n-Ind}_{Q}^{\mathbb{G}_{\vec{n}}}(\Pi_{M})_{\xi}(f)\circ A'_{\mbox{n-Ind}_{Q}^{\mathbb{G}_{\vec{n}}}(\Pi_{M})_{\xi}},\]
where $M$ runs over $\mathbb{Q}$-Levi subgroups of $\mathbb{G}_{\vec{n}}$
containing a fixed minimal Levi and $Q$ is a parabolic containing
$M$ as a Levi. The rest of the notation is defined on pages 31 and
32 of \cite{Shin}. Note that $A'_{\mbox{n-Ind}_{Q}^{\mathbb{G}_{\vec{n}}}(\Pi_{M})\xi}$
is a normalized intertwining operator for $\mbox{n-Ind}_{Q}^{\mathbb{G}_{\vec{n}}}(\Pi_{M})_{\xi}.$

We will be particularly interested in making the above formula explicit
when $\vec{n}=(n)$. In that case, $I_{\mathrm{spec}}^{\mathbb{G}_{\vec{n}}\theta}(f\theta)$
is a sum of \[
\frac{1}{2}\sum_{\Pi'}\mbox{tr}(\Pi'_{\xi}(f)A'_{\Pi'_{\xi}}),\]
 where $\Pi'$ runs over $\theta$-stable subrepresentations of $R_{\mathbb{G}_{n},\mbox{disc}}$,
and \[
\sum_{M\subsetneq\mathbb{G}_{n}}\frac{|W_{M}|}{|W_{\mathbb{G}_{n}|}}|\det(\Phi_{n}^{-1}\theta-1)_{\mathfrak{a}_{M}^{\mathbb{G}_{n}\theta}}|^{-1}\sum_{\Pi'_{M}}\mbox{tr}(\mbox{n-Ind}_{Q}^{\mathbb{G}_{n}}(\Pi'_{M})_{\xi}(f)\circ A'_{\mbox{n-Ind}_{Q}^{\mathbb{G}_{n}}(\Pi'_{M})_{\xi}}),\]
where $\Pi'_{M}$ runs over $\Phi_{n}^{-1}\theta$-stable subrepresentations
of $R_{M,\mathrm{disc}}$.

Consider the finite set $\mathcal{E}_{p}^{\mathrm{eff}}(J^{(h_{1},h_{2})},G,G_{\vec{n}})$
consisting of isomorphism classes of $G$-endoscopic triples $(M_{G_{\vec{n}}},s_{\vec{n}},\eta_{\vec{n}})$
for $J^{(h_{1},h_{2})}$, defined in section 6.2 of \cite{Shin-2}.
Let $c_{M_{G_{\vec{n}}}}\in\{\pm1\}$ be the constant assigned to
each triple in \cite{Shin-2}. If $b$ is the isocrystal corresponding
to $(h_{1},h_{2})$, let $M^{(h_{1},h_{2})}(\mathbb{Q}_{p})$ be the
centralizer of $\nu_{G}(b)$. The isocrystal $b$ can be described
as $(b_{p,0},b_{\mathfrak{p}_{1}},\dots,b_{\mathfrak{p}_{r}})$ where
$b_{\mathfrak{p}_{i}}$ has slopes $0$ and $\frac{1}{n-h_{i}}$ for
$i=1,2$ and slope $0$ for $i>2$. Then $M^{(h_{1},h_{2})}$ is a
$\mathbb{Q}_{p}$ -rational Levi subgroup of $G.$ We will define
a group morphism \[
\mbox{n-Red}_{\vec{n}}^{(h_{1},h_{2})}:\mbox{Groth}(G_{n}(\mathbb{Q}_{p})\to\mbox{Groth}(J^{(h_{1},h_{2})}(\mathbb{Q}_{p}))\]
 as the composition of the following maps \[
\mbox{Groth}(G_{\vec{n}}(\mathbb{Q}_{p}))\to\bigoplus_{(M_{G_{\vec{n}}},s_{G_{\vec{n}}},\eta_{G_{\vec{n}})}}\mbox{Groth}(M_{G_{\vec{n}}}(\mathbb{Q}_{p}))\stackrel{\oplus\tilde{\eta}_{G_{\vec{n}},*}}{\longrightarrow}\mbox{Groth}(M^{(h_{1},h_{2})}(\mathbb{Q}_{p}))\]
\[
\stackrel{LJ_{J^{(h_{1},h_{2})}}^{M^{(h_{1},h_{2})}}}{\longrightarrow}\mbox{Groth}(J^{(h_{1},h_{2})}(\mathbb{Q}_{p})).\]
The sum runs over $(M_{G_{\vec{n}}},s_{\vec{n}},\eta_{\vec{n}})\in\mathcal{E}_{p}^{\mathrm{eff}}(J^{(h_{1},h_{2})},G,G_{\vec{n}})$.
The first map is the direct sum of maps $\mbox{Groth}(G_{\vec{n}}(\mathbb{Q}_{p}))\to\mbox{Groth}(M_{G_{\vec{n}}}(\mathbb{Q}_{p}))$
which are given by $\oplus_{i}c_{M_{G_{\vec{n}}}}\cdot J_{P(iM_{G_{\vec{n}}})^{\mbox{op}}}^{G_{\vec{n}}}$,
where $i\in\mathcal{I}(M_{G_{\vec{n}}},G_{\vec{n}})$ is a $\mathbb{Q}_{p}$-embedding
$M_{G_{\vec{n}}}\hookrightarrow G_{\vec{n}}$ and $P(iM_{G_{\vec{n}}})$
is a parabolic subgroup of $G_{\vec{n}}$ which contains $i(M_{G_{\vec{n}}})$
as a Levi subgroup. The map $\tilde{\eta}_{G_{\vec{n}},*}$ is functorial
transfer with respect to the $L$-morphism $\tilde{\eta}_{G_{\vec{n}}}$.
The third map, $LJ_{J^{(h_{1},h_{2})}}^{M^{(h_{1},h_{2})}}$ is the
Jacquet-Langlands map on Grothendieck groups. We also define \[
\mbox{Red}_{\vec{n}}^{(h_{1},h_{2})}(\pi_{G_{\vec{n}},p}):=\mbox{n-Red}_{\vec{n}}^{(h_{1},h_{2})}(\pi_{G_{\vec{n}},p})\otimes\bar{\delta}_{P(J^{(h_{1},h_{2})})}^{\frac{1}{2}}\]

We can describe all the groups and maps above very explicitly in the
case $\vec{n}=(n)$. Indeed, $\mathcal{E}_{p}^{\mbox{eff}}(J^{(h_{1},h_{2})},G,G_{n})$
has a unique isomorphism class represented by \[
(M_{G_{n}},s_{G_{n}},\eta_{G_{n}})=(M^{(h_{1},h_{2})},1,\mbox{id}).\]
The set $\mathcal{I}(M_{G_{n}},G_{n}$) is also a singleton in this
case, so we suppress $i$ everywhere. This means that we can also
take $\tilde{\eta}_{G_{n}}=\mbox{id}$ and $\tilde{\eta}_{G_{n},*}=\mbox{id}$
and by Remark 6.4 of \cite{Shin-2}, we may also take $c_{M_{G_{n}}}=e_{p}(J^{(h_{1},h_{2})})$,
which is the Kottwitz of the $\mathbb{Q}_{p}$-group $J^{(h_{1},h_{2})}$.
There are isomorphisms \[
G(\mathbb{Q}_{p})\simeq\mathbb{Q}_{p}^{\times}\times GL_{n}(F_{\mathfrak{p}_{1}})\times GL_{n}(F_{\mathfrak{p}_{2}})\times\prod_{i>2}GL_{n}(F_{\mathfrak{p}_{i}}),\]
\[
M^{(h_{1},h_{2})}(\mathbb{Q}_{p})\simeq\mathbb{Q}_{p}^{\times}\times(GL_{n-h_{1}}(F_{\mathfrak{p}_{1}})\times GL_{h_{1}}(F_{\mathfrak{p}_{1}}))\times(GL_{n-h_{2}}(F_{\mathfrak{p}_{2}})\times GL_{h_{2}}(F_{\mathfrak{p}_{2}}))\]
\[
\times\prod_{i>2}GL_{n}(F_{\mathfrak{p}_{i}}),\]
\[
J^{(h_{1},h_{2})}(\mathbb{Q}_{p})\simeq\mathbb{Q}_{p}^{\times}\times(D_{F_{\mathfrak{p}_{1}},\frac{1}{n-h_{1}}}^{\times}\times GL_{h_{1}}(F_{\mathfrak{p}_{1}}))\times(D_{F_{\mathfrak{p}_{2}},\frac{1}{n-h_{2}}}^{\times}\times GL_{h_{2}}(F_{\mathfrak{p}_{2}}))\times\prod_{i>2}GL_{n}(F_{\mathfrak{p}_{i}}).\]
Thus, $e_{p}(J^{(h_{1},h_{2})})=(-1)^{2n-2-h_{1}-h_{2}}$ .$ $ If
we write $\pi_{p}=\pi_{p,0}\otimes(\otimes_{i}\pi_{\mathfrak{p}_{i}})$,
then we have \[
\mbox{Red}_{n}^{(h_{1},h_{2})}(\pi_{p})=(-1)^{h_{1}+h_{2}}\pi_{p,0}\otimes\mbox{Red}^{n-h_{1},h_{1}}(\pi_{\mathfrak{p}_{1}})\otimes\mbox{Red}^{n-h_{2},h_{2}}(\pi_{\mathfrak{p}_{2}})\otimes(\otimes_{i>2}\pi_{\mathfrak{p}_{i}}).\]

\begin{lem}
\label{trace of red for n}For any $\pi_{p}\in\mbox{Groth}(G_{n}(\mathbb{Q}_{p}))$\[
\mbox{tr}\pi_{p}(\phi_{\mathrm{Ig},p}^{n})=\mbox{tr}(\mbox{Red}_{n}^{(h_{1},h_{2})}(\pi_{p}))(\phi'_{p}).\]
\end{lem}
\begin{proof}
Set $M=M_{G_{n}}$. We know that $\phi_{\mathrm{Ig},p}^{n}=e_{p}(J^{(h_{1},h_{2})})\cdot\tilde{\phi}_{p}^{M}$.
By Lemma 3.9 of \cite{Shin-2}, \[
\mbox{tr}\pi_{p}(\tilde{\phi}_{p}^{M})=\mbox{tr}(J_{P_{M}^{\mathrm{op}}}^{G_{n}}(\pi_{p}))(\phi_{p}^{M}).\]
Here $\phi_{p}^{M}$ is a $\Delta_{p}(\:,\ )_{M}^{J^{(h_{1},h_{2})}}\equiv e_{p}(J^{(h_{1},h_{2})})$-transfer
of $\phi_{p}^{0}=\phi'_{p}\cdot\bar{\delta}_{P(J^{(h_{1},h_{2})})}^{\frac{1}{2}}$
(by remark 6.4 of \cite{Shin-2}, we have an explicit description
of the transfer factor). Let $\pi_{M,p}=J_{P_{M}^{\mathrm{op}}}^{G_{n}}(\pi_{p}).$

Note that $M$ is a product of general linear groups and $J^{(h_{1},h_{2})}$
is an inner form of $M$. Lemma 2.18 and Remark 2.19 of \cite{Shin-2}
ensure that\[
\mbox{tr}\pi_{M,p}(\phi_{p}^{M})=\mbox{tr}(LJ_{M}^{J^{(h_{1},h_{2})}}(\pi_{M,p})(\phi_{p}^{0}))=\mbox{tr}(LJ_{M}^{J^{(h_{1},h_{2})}}(\pi_{M,p})\otimes\bar{\delta}_{P(J^{(h_{1},h_{2})})}^{\frac{1}{2}})(\phi'_{p}).\]
This concludes the proof. \end{proof}
\begin{lem}
\label{trace of red for n_1,n_2}Let $\vec{n}=(n_{1},n_{2})$ with
$n_{1}\geq n_{2}>0$. For any $\pi_{p}\in\mbox{Groth}(G_{n_{1},n_{2}}(\mathbb{Q}_{p})),$\[
\mbox{tr}\pi_{p}(\phi_{\mathrm{Ig},p}^{\vec{n}})=\mbox{tr}(\mbox{Red}_{\vec{n}}^{(h_{1},h_{2})}(\pi_{p}))(\phi'_{p}).\]
\end{lem}
\begin{proof}
The proof is based on making explicit the construction of $\phi_{\mathrm{Ig},p}^{\vec{n}}$
from section 6 of \cite{Shin-2} together with the definition of the
functor $\mbox{n-Red}_{\vec{n}}^{(h_{1},h_{2})}$, which is a composition
of the following maps:\[
\mbox{Groth}(G_{\vec{n}}(\mathbb{Q}_{p}))\to\bigoplus_{(M_{G_{\vec{n}}},s_{G_{\vec{n}}},\eta_{G_{\vec{n}})}}\mbox{Groth}(M_{G_{\vec{n}}}(\mathbb{Q}_{p}))\stackrel{\oplus\tilde{\eta}_{G_{\vec{n}},*}}{\longrightarrow}\mbox{Groth}(M^{(h_{1},h_{2})}(\mathbb{Q}_{p}))\]
\[
\stackrel{LJ_{J^{(h_{1},h_{2})}}^{M^{(h_{1},h_{2})}}}{\longrightarrow}\mbox{Groth}(J^{(h_{1},h_{2})}(\mathbb{Q}_{p})).\]
Recall that \[
\phi_{\mathrm{Ig},p}^{\vec{n}}=\sum_{(M_{G_{\vec{n}}},s_{G_{\vec{n}},}\eta_{G_{\vec{n}})}}\sum_{i}c_{M_{G_{\vec{n}}}}\cdot\tilde{\phi}_{p}^{M_{G_{\vec{n}}},i}\]
as functions on $G_{\vec{n}}(\mathbb{Q}_{p})$, where the first sum
is taken over $\mathcal{E}^{\mathrm{eff}}(J^{(h_{1},h_{2})},G,G_{\vec{n}})$
and the second sum is taken over $\mathcal{I}(M_{G_{\vec{n}}},G_{\vec{n}})$.
By Lemma 3.9 of \cite{Shin-2}, \begin{equation}
\mbox{tr}\pi_{p}(\tilde{\phi}_{p}^{M_{G_{\vec{n}},i}})=\mbox{tr}(J_{P(iM_{G_{\vec{n}}})^{\mathrm{op}}}^{G_{\vec{n}}}(\pi_{p})(\phi_{p}^{M_{G_{\vec{n}}}}),\label{eq:jacquet}\end{equation}
where $\phi_{p}^{M_{G_{\vec{n}}}}\in C_{c}^{\infty}(M_{G_{\vec{n}}}(\mathbb{Q}_{p}))$
is a $\Delta_{p}(\ ,\ )_{M_{G_{\vec{n}}}}^{J^{(h_{1},h_{2})}}$-transfer
of $\phi_{p}^{0}=\phi'_{p}\cdot\bar{\delta}_{P(J^{(h_{1},h_{2})})}^{\frac{1}{2}}$.
Equation \ref{eq:jacquet} tells us that \begin{equation}
\mbox{tr}\pi_{p}(\phi_{\mathrm{Ig},p}^{\vec{n}})=\sum_{(M_{G_{\vec{n}},}s_{G_{\vec{n}},}\eta_{G_{\vec{n}}})}\mbox{tr}(f_{M_{G_{\vec{n}}}}(\pi_{p}))(\phi_{p}^{M_{G_{\vec{n}}}}),\label{eq:first map}\end{equation}
where $f_{M_{G_{\vec{n}}}}(\pi_{p})=\oplus_{i}c_{M_{G_{\vec{n}}}}J_{P(iM_{G_{\vec{n}}})}^{G_{\vec{n}}}(\pi_{p})$.
The first map in the definition of $\mbox{Red}_{\vec{n}}^{(h_{1},h_{2})}$
is the direct sum of $f_{M_{G_{\vec{n}}}}$ over all $(M_{G_{\vec{n}}},s_{\vec{n}},\eta_{\vec{n}})$.

The function $\phi_{p}^{M_{G_{\vec{n}}}}$ is a $\Delta_{p}(\ ,\ )_{M_{G_{\vec{n}}}}^{M^{(h_{1},h_{2})}}$
-transfer of the function $\phi_{p}^{*}\in C_{c}^{\infty}(M^{(h_{1},h_{2})}(\mathbb{Q}_{p}))$
which is itself a transfer of $\phi_{p}^{0}$ via $\Delta_{p}(\ ,\ )_{M^{(h_{1},h_{2})}}^{J^{(h_{1},h_{2})}}\equiv e_{p}(J^{(h_{1},h_{2})})$.
(All transfer factors are normalized as in \cite{Shin-2}.) We will
focus on making the $\Delta_{p}(\ ,\ )_{M_{G_{\vec{n}}}}^{M^{(h_{1},h_{2})}}$
-transfer explicit first, for which we need to have a complete description
of all endoscopic triples $(M_{G_{\vec{n}}},s_{G_{\vec{n}}},\eta_{G_{\vec{n}}}).$ 

We have the following isomorphisms over $\mathbb{Q}_{p}$.\[
G\simeq GL_{1}\times\prod_{i\geq1}R_{F_{\mathfrak{p}_{i}}/\mathbb{Q}_{p}}GL_{n}\]
\[
G_{n_{1},n_{2}}\simeq GL_{1}\times\prod_{i\geq1}R_{F_{\mathfrak{p}_{i}}/\mathbb{Q}_{p}}GL_{n_{1},n_{2}}\]
\[
M^{(h_{1},h_{2})}\simeq GL_{1}\times\prod_{i=1}^{2}R_{F_{\mathfrak{p}_{i}}/\mathbb{Q}_{p}}GL_{n-h_{i},h_{i}}\times\prod_{i>2}R_{F_{\mathfrak{p}_{i}}/\mathbb{Q}_{p}}GL_{n}\]
\[
J^{(h_{1},h_{2})}\simeq GL_{1}\times\prod_{i=1}^{2}(D_{F_{\mathfrak{p}_{i}},\frac{1}{n-h_{i}}}^{\times}\times GL_{h_{i}})\times\prod_{i>2}R_{F_{\mathfrak{p}_{i}}/\mathbb{Q}_{p}}GL_{n}.\]
Consider also the following four groups over $\mathbb{Q}_{p}$, which
can be thought of as Levi subgroups of $G_{n_{1},n_{2}}$ via the
block diagonal embeddings. \[
M_{G_{\vec{n}},1}:=GL_{1}\times\prod_{i=1}^{2}R_{F_{\mathfrak{\mathfrak{p}_{i}}}/\mathbb{Q}_{p}}GL_{n-h_{i},h_{i}-n_{2},n_{2}}\times\prod_{i>2}R_{F_{\mathfrak{p}_{i}}/\mathbb{Q}_{p}}GL_{n_{1},n_{2}}\]
\[
M_{G_{\vec{n}},2}:=GL_{1}\times\prod_{i=1}^{2}R_{F_{\mathfrak{\mathfrak{p}_{i}}}/\mathbb{Q}_{p}}GL_{n-h_{i},h_{i}-n_{1},n_{1}}\times\prod_{i>2}R_{F_{\mathfrak{p}_{i}}/\mathbb{Q}_{p}}GL_{n_{1},n_{2}}\]
\[
M_{G_{\vec{n}},3}:=GL_{1}\times\prod_{i=1}^{2}R_{F_{\mathfrak{\mathfrak{p}_{i}}}/\mathbb{Q}_{p}}GL_{n-h_{i},h_{i}-n_{i},n_{i}}\times\prod_{i>2}R_{F_{\mathfrak{p}_{i}}/\mathbb{Q}_{p}}GL_{n_{1},n_{2}}\]
\[
M_{G_{\vec{n}},4}:=GL_{1}\times\prod_{i=1}^{2}R_{F_{\mathfrak{\mathfrak{p}_{i}}}/\mathbb{Q}_{p}}GL_{n-h_{i},h_{i}-n_{3-i},n_{3-i}}\times\prod_{i>2}R_{F_{\mathfrak{p}_{i}}/\mathbb{Q}_{p}}GL_{n_{1},n_{2}}\]
Note that we only define $M_{G_{\vec{n}},j}$ when it makes sense,
for example $M_{G_{\vec{n}},1}$ is defined only when $h_{i}\geq n_{2}$
for $i=1,2$. We define $\eta_{G_{\vec{n}},j}:\widehat{M_{G_{\vec{n}},j}}\to\widehat{M^{(h_{1},h_{2})}}$
to be the obvious block diagonal embedding. We also let \[
s_{M_{G_{\vec{n}},j}}=(1,(\pm1,\pm1,\pm1)_{i=1,2},(1,1)_{i>2}),\]
 where the signs on the $F_{\mathfrak{p}_{i}}$-component are chosen
such that $s_{M_{G_{\vec{n}},}j}$ is positive on the $GL_{n_{1}}$-block
of the $F_{\mathfrak{p}_{i}}$-component and negative on the $GL_{n_{2}}$-block
of the $F_{\mathfrak{p}_{i}}$-component. 

It is easy to check, as on page 42 of \cite{Shin}, that $\mathcal{E}^{\mathrm{eff}}(J^{(h_{1},h_{2})},G,G_{\vec{n}})$
consists of those triples $(M_{G_{\vec{n}},j},s_{G_{\vec{n}},j},\eta_{G_{\vec{n}},j})$
which make sense. For example, if $h_{i}<n_{2}$ for $i=1,2$ then
$\mathcal{E}^{\mathrm{eff}}(J^{(h_{1},h_{2})},G,G_{\vec{n}})$ is
empty, but if $h_{i}\geq n_{1}$ for $i=1,2$ then $\mathcal{E}^{\mathrm{eff}}(J^{(h_{1},h_{2})},G,G_{\vec{n}})$
consists of four elements. The key point is to notice that for a triple
$(M_{G_{\vec{n}}},s_{G_{\vec{n}}},\eta_{G_{\vec{n}}})$ to lie in
$\mathcal{E}^{\mathrm{eff}}(J^{(h_{1},h_{2})},G,G_{\vec{n}})$ it
is necessary for $s_{G_{\vec{n}}}$ to transfer to an element of the
dual group $\widehat{M^{(h_{1},h_{2)}}}=\widehat{J^{(h_{1},h_{2})}}$
which is either $+1$ or $-1$ in the $GL_{n-h_{i}}(\mathbb{C})$
block of the $F_{\mathfrak{p}_{i}}$-component. 

We can extend $\eta_{G_{\vec{n}},j}$ to an $L$-morphism $\tilde{\eta}_{G_{\vec{n}},j}:\mbox{}^{L}M_{G_{\vec{n}},j}\to\mbox{}^{L}M^{(h_{1},h_{2})}$
which is compatible with the $L$-morphism $\eta_{\vec{n}}:\mbox{}^{L}G_{\vec{n}}\to^{L}G$,
when we map $\mbox{}^{L}M_{G_{\vec{n}},j}\stackrel{\tilde{l}_{j}}{\to}\mbox{}^{L}G_{\vec{n}}$
and $\mbox{}^{L}M^{(h_{1},h_{2})}\stackrel{\tilde{l}}{\to}\mbox{}^{L}G$
$ $via (a conjugate of) the obvious block diagonal embedding (where
we always send the $GL_{n_{1}}$-block to the top left corner and
the $GL_{n_{2}}$-block to the bottom right corner). The morphism
$\tilde{\eta}_{G_{\vec{n}},j}$ is defined as on page 42 of \cite{Shin},
by sending $z\in W_{\mathbb{Q}_{p}}$ to one of the matrices\[
\left(\begin{array}{cc}
\varpi(z)^{\epsilon(n-n_{1})}I_{n_{1}} & 0\\
0 & \varpi(z)^{\epsilon(n-n_{2})}I_{n_{2}}\end{array}\right)\mbox{ or }\left(\begin{array}{cc}
\varpi(z)^{\epsilon(n-n_{2})}I_{n_{2}} & 0\\
0 & \varpi(z)^{\epsilon(n-n_{1})}I_{n_{1}}\end{array}\right),\]
on the $F_{\mathfrak{p}_{i}}$-component of $\widehat{M^{(h_{1},h_{2})}}$.
(For $i=1,2$, we send $z$ to the first matrix on the $F_{\mathfrak{p}_{i}}$-component
if the endoscopic group $M_{G_{\vec{n}},j}$ at $\mathfrak{p}_{i}$
is $GL_{n-h_{i},h_{i}-n_{2},n_{2}}$ and to the second matrix if the
component of $M_{G_{\vec{n}},j}$ at $\mathfrak{p}_{i}$ is $GL_{n-h_{i},h_{i}-n_{1},n_{1}}$.
For $i>2$, we send $z$ to the first matrix on the $F_{\mathfrak{p}_{i}}$-component.)
This map $\tilde{\eta}_{G_{\vec{n}},j}$ is the unique $L$-morphism
which makes the diagram\[
\xymatrix{\mbox{}^{L}M^{(h_{1},h_{2})}\ar[r]\sp-{\tilde{l}} & \mbox{}^{L}G\\
\mbox{}^{L}M_{G_{\vec{n}},j}\ar[r]\sp-{\tilde{l}_{j}}\ar[u]^{\tilde{\eta}_{G_{\vec{n}}}} & \mbox{}^{L}G_{\vec{n}}\ar[u]^{\tilde{\eta}_{\vec{n}}}}
\]
commutative. Thus, the function $\phi_{p}^{M_{G_{\vec{n}},j}}$ is
a transfer of $\phi_{p}^{*}$ with respect to the $L$-morphism $\tilde{\eta}_{G_{\vec{n}},j}$,
so we can define explicitly both $\phi_{p}^{M_{G_{\vec{n}},j}}$ and
the representation-theoretic map $\tilde{\eta}_{M_{G_{\vec{n}},j}*}:\mbox{Groth}(M_{G_{\vec{n}},j}(\mathbb{Q}_{p}))\to\mbox{Groth}(M^{(h_{1},h_{2})}(\mathbb{Q}_{p}))$.
There exists a unitary character $\chi_{u,j}^{+}:M_{G_{\vec{n}},j}(\mathbb{Q}_{p})\to\mathbb{C}^{\times}$
(defined similarly to the character on page 43 of \cite{Shin}) such
that the Langlands-Shelstad transfer factor with respect to $\tilde{\eta}_{G_{\vec{n}},j}$
differs from the transfer factor associated to the canonical $L$-morphism
by the cocycle associated to $\chi_{u,j}^{+}$. (See section 9 of
\cite{Bor} for an explanation of the correspondence between cocycles
in $H^{1}(W_{\mathbb{Q}_{p}},Z(\widehat{M_{G_{\vec{n}},j}}))$ and
characters $M_{G_{\vec{n}},j}(\mathbb{Q}_{p})\to\mathbb{C}^{\times}$.)

We can in fact compute $\chi_{u,j}^{+}$ on the different components
of $M_{G_{\vec{n}},j}(\mathbb{Q}_{p})$, by keeping in mind that it
is the character $M_{G_{\vec{n}},j}(\mathbb{Q}_{p})\to\mathbb{C}^{\times}$
associated to the cocycle in $H^{1}(W_{\mathbb{Q}_{p}},Z(\widehat{M_{G_{\vec{n}},j}}))$
which takes the conjugacy class of the standard Levi embedding $\widehat{M_{G_{\vec{n}},j}}\to\widehat{M^{(h_{1},h_{2})}}$
to that of $\eta_{G_{\vec{n}},j}$. Thus, we have \[
\chi_{u,j}^{+}(\lambda)=\varpi_{u}(\lambda)^{-N(n_{1},n_{2})};\]
\[
\chi_{u,j}^{+}(g_{\mathfrak{p}_{i},1},g_{\mathfrak{p}_{i},2},g_{\mathfrak{p}_{i},3})=\begin{cases}
\varpi_{u}\left(N_{F_{\mathfrak{p}_{i}}/E_{u}}\left(\det((g_{\mathfrak{p}_{i},1}g_{\mathfrak{p}_{i,2}})^{\epsilon(n-n_{1})}g_{\mathfrak{p}_{i},3}^{\epsilon(n-n_{2})})\right)\right)\\
\varpi_{u}\left(N_{F_{\mathfrak{p}_{i}}/E_{u}}\left(\det((g_{\mathfrak{p}_{i},1}g_{\mathfrak{p}_{i,2}})^{\epsilon(n-n_{2})}g_{\mathfrak{p}_{i},3}^{\epsilon(n-n_{1})})\right)\right)\end{cases}\]
when $i=1,2$ and depending on whether $M_{G_{\vec{n}},j}$ has the
group $GL_{n-h_{i},h_{i}-n_{2},n_{2}}$ or the group $GL_{n-h_{i},h_{i}-n_{1},n_{1}}$
as its $F_{\mathfrak{p}_{i}}$-component; and \[
\chi_{u,j}^{+}(g_{\mathfrak{p}_{i},1},g_{\mathfrak{p}_{i},2})=\varpi_{u}\left(N_{\mathfrak{p}_{i}/E_{u}}\left(\det(g_{\mathfrak{p}_{i},1}^{\epsilon(n-n_{1})}g_{\mathfrak{p}_{i},2}^{\epsilon(n-n_{2})})\right)\right)\mbox{ when }i>2\]
where $(\lambda,(g_{\mathfrak{p}_{i},1},g_{\mathfrak{p}_{i},2},g_{\mathfrak{p}_{i},3})_{i=1,2},(g_{\mathfrak{p}_{i},1},g_{\mathfrak{p}_{i},2})_{i>2})$
denotes an element of $M_{G_{\vec{n}},j}(\mathbb{Q}_{p})$. (The value
of $\chi_{u,j}^{+}$ is in fact the product of the three types of
factors above.) 

We let $Q_{j}$ be a parabolic subrgroup of $M^{(h_{1},h_{2})}$ containing
$M_{G_{\vec{n}},j}$ as a Levi and if we let $(\phi_{p}^{*})^{Q_{j}}$
be the constant term of $\phi_{p}^{*}$ along $Q_{j}$ then we have\[
\phi_{p}^{M_{G_{\vec{n}},}j}:=(\phi_{p}^{*})^{Q_{j}}\cdot\chi_{u,j}^{+}\mbox{ and }\]
\[
\tilde{\eta}_{G_{\vec{n}},j*}(\pi_{M_{G_{\vec{n}}},j}):=\mbox{n-Ind}_{Q_{j}}^{M^{(h_{1},h_{2})}}(\pi_{M_{G_{\vec{n}},j}}\otimes\chi_{u,j}^{+})\]
for any $\pi_{M_{G_{\vec{n}},j}}\in\mbox{Irr}_{l}(M_{G_{\vec{n}},j}(\mathbb{Q}_{p}))$.
By Lemma 3.3 of \cite{Shin} \begin{equation}
\mbox{tr}(f_{M_{G_{\vec{n}},j}}(\pi_{p}))(\phi_{p}^{M_{G_{\vec{n}},j}})=\mbox{tr}(\tilde{\eta}_{G_{\vec{n}},j*}(f_{M_{G_{\vec{n}},j}}(\pi_{p})))(\phi_{p}^{*}).\label{eq:second map}\end{equation}

The group $J^{(h_{1},h_{2})}$ is an inner form of $M^{(h_{1},h_{2})}$,
which is a product of general linear groups. By Lemma 2.18 and Remark
2.19 of \cite{Shin-2}, \[
\mbox{tr}(\tilde{\eta}_{G_{\vec{n}},*}(f_{M_{j}}(\pi_{p})))(\phi_{p}^{*})=\mbox{tr}(LJ(\tilde{\eta}_{G_{\vec{n}},*}(f_{M_{j}}(\pi_{p}))))(\phi_{p}^{0})\]
\begin{equation}
=\mbox{tr}\left(LJ(\tilde{\eta}_{G_{\vec{n}},*}(f_{M_{j}}(\pi{}_{p}))\otimes\bar{\delta}_{P(J^{(h_{1},h_{2})})}^{\frac{1}{2}}\right)(\phi_{p}^{'}),\label{eq:third map}\end{equation}
where we've abbreviated $M_{G_{\vec{n}},j}$ by $M_{j}$. Putting
together (\ref{eq:first map}),(\ref{eq:second map}) and (\ref{eq:third map}),
we get the desired result.
\end{proof}
Let $\Xi^{1}$ be the algebraic representation of $(\mathbb{G}_{n})_{\mathbb{C}}$
obtained by base change from $\iota_{l}\xi$. Let $\Pi^{1}\simeq\psi\otimes\Pi^{0}$
be an automorphic representation of $\mathbb{G}_{n}(\mathbb{A})\simeq GL_{1}(\mathbb{A}_{E})\times GL_{n}(\mathbb{A}_{F})$.
Assume that 
\begin{itemize}
\item $\Pi^{1}\simeq\Pi^{1}\circ\theta,$
\item $\Pi_{\infty}^{1}$ is generic and $\Xi^{1}$-cohomological,
\item $\mbox{Ram}_{\mathbb{Q}}(\Pi)\subset\mbox{Spl}_{F/F_{2},\mathbb{Q}}$, 
\item $\Pi^{1}$ is cuspidal.
\end{itemize}
In particular, $\Pi_{\infty}^{1}\simeq\Pi_{\Xi}$, which was defined
above. Let $\mathfrak{S}_{\mathrm{fin}}$ be a finite set of places
of $\mathbb{Q}$ such that \[
\mbox{Ram}_{F/\mathbb{Q}}\cup\mbox{Ram}_{\mathbb{Q}}(\varpi)\cup\mbox{Ram}_{\mathbb{Q}}(\Pi)\cup\{p\}\subset\mathfrak{S}_{\mathrm{fin}}\subset\mbox{Spl}_{F/F_{2},\mathbb{Q}}\]
 and let $\mathfrak{S}=\mathfrak{S}_{\mathrm{fin}}\cup\{\infty\}$.
\begin{thm}
Define $C_{G}=|\ker^{1}(\mathbb{Q},G)|\cdot\tau(G)$. For each $0\leq h_{1},h_{2}\leq n$,
the following equality holds in $\mbox{Groth}(\mathbb{G}_{n}(\mathbb{A}_{\mathfrak{S}_{\mathrm{fin}}\setminus\{p\}})\times J^{(h_{1},h_{2})}(\mathbb{Q}_{p}).$\[
BC_{\mathfrak{S}_{\mathrm{fin}}\setminus\{p\}}(H_{c}(\mbox{Ig}^{(h_{1},h_{2})},\mathcal{L}_{\xi}))\{\Pi^{1,\mathfrak{S}}\}\]
\[
=C_{G}\cdot e_{0}\cdot(-1)^{h_{1}+h_{2}}\cdot[\iota_{l}^{-1}\Pi_{\mathfrak{S}_{\mathrm{fin}}\setminus\{p\}}^{1}][\mbox{Red}_{n}^{(h_{1},h_{2})}(\pi_{p})],\]
where $e_{0}=\pm1$ is independent of $(h_{1},h_{2})$. \end{thm}
\begin{proof}
The proof goes through identically to the proof of the first part
of Theorem 6.1 of \cite{Shin}. We nevertheless give the proof in
detail.

First, we explain the choice of test functions to be used in the trace
formula. Let $(f^{n})^{\mathfrak{S}}\in\mathcal{H}^{\mathrm{ur}}(\mathbb{G}_{n}(\mathbb{A}^{\mathfrak{S}}))$
and $f_{\mathfrak{S}_{\mathrm{fin}}\setminus\{p\}}^{n}\in C_{c}^{\infty}(\mathbb{G}_{n}(\mathbb{A}_{\mathfrak{S}_{\mathrm{fin}}\setminus\{p\}}))$
be any functions. Let $\phi^{\mathfrak{S}}$ and $\phi_{\mathfrak{S}_{\mathrm{fin}}\setminus\{p\}}$
be the BC-transfers of $(f^{n})^{\mathfrak{S}}$ and $(f^{n})_{\mathfrak{S}_{\mathrm{fin}}\setminus\{p\}}$
from $\mathbb{G}_{n}$ to $G_{n}$. Let $\phi^{\infty,p}=\phi^{\mathfrak{S}}\phi_{\mathfrak{S}_{\mathrm{fin}}\setminus\{p\}}$
and choose any $\phi'_{p}\in C_{c}^{\infty}(J^{(h_{1},h_{2})}(\mathbb{Q}_{p}))$
such that $\phi^{\infty,p}\phi'_{p}$ is an acceptable function. 

For each $G_{\vec{n}}\in\mathcal{E}^{\mathrm{ell}}(G)$ we construct
the function $\phi_{\mathrm{Ig}}^{\vec{n}}\in C_{c}^{\infty}(G_{\vec{n}}(\mathbb{A}))$
associated to $\phi^{\infty,p}\phi'_{p}$ as above. Recall that $(\phi_{\mathrm{Ig}}^{\vec{n}})^{\mathfrak{S}}$
and $(\phi_{\mathrm{Ig}}^{\vec{n}})_{\mathfrak{S}_{\mathrm{fin}}\setminus\{p\}}$
are the $\Delta(\ ,\ )_{G_{\vec{n}}}^{G_{n}}$ transfers of $\phi^{\mathfrak{S}}$
and $\phi_{\mathfrak{S}_{\mathrm{fin}}\setminus\{p\}}$. Recall that
we take\[
\phi_{\mathrm{Ig},\infty}^{\vec{n}}:=e_{\vec{n}}(\Delta_{\infty})\cdot(-1)^{q(G)}\langle\mu_{h},s_{\vec{n}}\rangle\sum_{\varphi_{_{\vec{n}}}}\det(\omega_{*}(\varphi_{G_{\vec{n}}}))\cdot\phi_{G_{\vec{n}},\xi(\varphi_{\vec{n}})},\]
where $\varphi_{\vec{n}}$ runs over $L$-parameters such that $\tilde{\eta}_{\vec{n}}\varphi_{\vec{n}}\sim\varphi_{\xi}$
and $\xi(\varphi_{\vec{n}})$ is the algebraic representation of $G_{\vec{n},\mathbb{C}}$
such that the $L$-packet associated to $\varphi_{\vec{n}}$ is $\Pi_{\mathrm{disc}}(G_{\vec{n}}(\mathbb{R}),\xi(\varphi_{\vec{n}})^{\vee})$.
The construction of $\phi_{\mathrm{Ig},p}^{\vec{n}}$ can be found
in \cite{Shin-2}. 

We will need to define a function $f^{\vec{n}}$, which plays the
part of a BC-matching function for $\phi_{\mathrm{Ig}}^{\vec{n}}$
for each $\vec{n}$. We already have defined $(f^{n})^{\mathfrak{S}}$
and $f_{\mathfrak{S}_{\mathrm{fin}}\setminus\{p\}}^{n}$. We take
$(f^{n_{1},n_{2}})^{\mathfrak{S}}=\tilde{\zeta}^{*}((f^{n})^{\mathfrak{S}})$
and $f_{\mathfrak{S}_{\mathrm{fin}}\setminus\{p\}}^{n_{1},n_{2}}=\tilde{\zeta}^{*}(f_{\mathfrak{S}_{\mathrm{fin}}\setminus\{p\}}^{n})$.
We also define \[
f_{\infty}^{\vec{n}}:=e_{\vec{n}}(\Delta)\cdot(-1)^{q(G)}\langle\mu_{h},s_{\vec{n}}\rangle\sum_{\varphi_{\vec{n}}}\det(\omega_{*}(\varphi_{G_{\vec{n}}}))\cdot f_{\mathbb{G}_{\vec{n}},\Xi(\varphi_{n})}\]
where $\varphi_{\vec{n}}$ runs over $L$-parameters such that $\tilde{\eta}_{\vec{n}}\varphi_{\vec{n}}\sim\varphi_{\xi}$
and $\Xi(\varphi_{\vec{n}})$ is the algebraic representation of $\mathbb{G}_{\vec{n}}$
arising from $\xi(\varphi_{\vec{n}})$. It is straightforward to verify
from their definitions that $f_{\infty}^{\vec{n}}$ and $\phi_{\mathrm{Ig},\infty}^{\vec{n}}$
are BC-matching functions. Finally, we choose $f_{p}^{\vec{n}}$ so
that its BC-transfer is $\phi_{\mathrm{Ig},p}^{\vec{n}}$. (Since
$p$ splits in $E$ it can be checked that the base change map defined
in section 4.2 of \cite{Shin} is surjective at $p$.) We set \[
f^{\vec{n}}:=(f^{\vec{n}})^{\mathfrak{S}}\cdot f_{\mathfrak{S}_{\mathrm{fin}}\setminus\{p\}}^{\vec{n}}\cdot f_{p}^{\vec{n}}\cdot f_{\infty}^{\vec{n}}.\]
The BC-transfer of $f^{\vec{n}}$ coincides with $\phi_{\mathrm{Ig}}^{\vec{n}}$
at places outside $\mathfrak{S}$ (by compatibility of transfers),
at $p$ and at $\infty$. At places in $\mathfrak{S}_{\mathrm{fin}}\setminus\{p\}$
we know at least that the BC-transfer of $f^{\vec{n}}$ has the same
trace as $\phi_{\mathrm{Ig}}^{\vec{n}}$ against every admissible
representation of $G_{\vec{n}}(\mathbb{A}_{\mathfrak{S}_{\mathrm{fin}}\setminus\{p\}}).$ 

By the discussion following Proposition \ref{stable trace formula},
we can compute \begin{equation}
\mbox{tr}(\phi^{\infty,p}\phi'_{p}|\iota_{l}H_{c}(\mbox{Ig}^{(h_{1},h_{2})},\mathcal{L}_{\xi}))\label{eq:trace}\end{equation}
 via the spectral part of the twisted formula, to get \[
C_{G}(-1)^{h_{1}+h_{2}}\left(\frac{1}{2}\sum_{\Pi'}\mbox{tr}(\Pi'_{\xi}(f^{n})A'_{\Pi_{\xi}'})+\sum_{\mathbb{G}_{n_{1},n_{2},n_{1}\not=n_{2}}}I_{\mathrm{spec}}^{\mathbb{G}_{n_{1},n_{2}}\theta}(f^{n_{1},n_{2}})\right.\]
\[
+\frac{1}{2}I_{\mathrm{spec}}^{\mathbb{G}_{n/2,n/2}\theta}(f^{n/2,n/2})\]
\begin{equation}
\left.+\sum_{M\subsetneq\mathbb{G}_{n}}\frac{|W_{M}|}{|W_{\mathbb{G}_{n}}|}|\det(\Phi^{-1}\theta-1)_{\mathfrak{a}_{M}^{\mathbb{G}_{n}\theta}}|^{-1}\sum_{\Pi'_{M}}\mbox{tr}(\mbox{n-Ind}_{Q}^{\mathbb{G}_{n}}(\Pi'_{M})_{\xi}(f^{n})\circ A'_{\mathrm{n-Ind}_{Q}^{\mathbb{G}_{n}}(\Pi'_{M})_{\xi}})\right)\label{eq:sum}\end{equation}
where the first sum runs over $\theta$-stable subrepresentations
$\Pi'$ of $R_{\mathbb{G}_{n},\mathrm{disc}}$, the sums in the middle
run over groups $\mathbb{G}_{n_{1},n_{2}}$ coming from elliptic endoscopic
groups $G_{n_{1},n_{2}}$ for $G$ (with $n_{1}\geq n_{2}>0$ and
some $(n_{1},n_{2})$ possibly excluded). The group $M$ runs over
proper Levi subgroups of $\mathbb{G}_{n}$ containing a fixed minimal
Levi and $\Pi'_{M}$ runs over $\Phi_{n}^{-1}\theta$-stable subrepresentations
$\Pi'_{M}$ of $R_{M,\mathrm{disc}}$. 

We claim that the formula above holds for any $\phi^{\infty,p}\phi'_{p}$,
without the assumption that it is an acceptable function. To see this,
note that Lemma 6.3 of \cite{Shin-1} guarantees that there exists
some element $fr^{s}\in J^{(h_{1},h_{2})}(\mathbb{Q}_{p})$ such that
$\phi^{\infty,p}(\phi'_{p})^{(N)}(g)=\phi^{\infty,p}(g)\phi_{p}'(g(fr^{s})^{N})$
is acceptable for any sufficiently large $N$. (The paper \cite{Shin-1}
treats general Igusa varieties, and it is easy to check that our case
is covered.) So the equality of (\ref{eq:trace}) and (\ref{eq:sum})
holds when $\phi'_{p}$ is replaced by $(\phi'_{p})^{(N)}$. Both
(\ref{eq:trace}) and (\ref{eq:sum}) are finite linear combinations
of terms of the form $\mbox{tr}\rho((\phi'_{p})^{(N)})$ where $\rho\in\mbox{Irr}(J^{(h_{1},h_{2})}(\mathbb{Q}_{p}))$.
In order to see that this is true for (\ref{eq:sum}), we need to
translate it from computing the trace of $f^{\vec{n}}$ to computing
the trace of $\phi_{\mathrm{Ig}}^{\vec{n}}$ to computing the trace
of $\phi'_{p}$, using Lemmas \ref{trace of red for n} and \ref{trace of red for n_1,n_2}.
Now the same argument as that for Lemma 6.4 of \cite{Shin-1} shows
that (\ref{eq:trace}) and (\ref{eq:sum}) are equal for $\phi^{\infty,p}(\phi'_{p})^{(N)}$
for every integer $N$, in particular for $N=0$. Thus, we can work
with arbitrary $\phi^{\infty,p}\phi'_{p}$.

Choose a decomposition of the normalized intertwining operators \[
A'_{\Pi^{1}}=A'_{\Pi^{1,\mathfrak{S}}}A'_{\Pi_{\mathfrak{S}_{\mathrm{fin}}}^{1}}A'_{\Pi_{\infty}^{1}}.\]
Set \[
\frac{A'_{\Pi^{1}}}{A_{\Pi^{1}}^{0}}:=\frac{A'_{\Pi^{1,\mathfrak{S}}}}{A_{\Pi^{1,\mathfrak{S}}}^{0}}\cdot\frac{A'_{\Pi_{\mathfrak{S}_{\mathrm{fin}}}^{1}}}{A_{\Pi_{\mathfrak{S}_{\mathrm{fin}}}^{1}}^{0}}\cdot\frac{A'_{\Pi_{\infty}^{1}}}{A_{\Pi_{\infty}^{1}}^{0}}\in\{\pm1\},\]
where the denominators on the right side are the normalized interwiners
chosen above. In the sum (\ref{eq:sum}), the third term evaluates
the trace of $f^{n}$ against representations induced from proper
Levi subgroups. The second term has a similar form: outside the set
$\mathfrak{S}$ we have the identity $(f^{n_{1},n_{2}})^{\mathfrak{S}}=\tilde{\zeta}^{*}((f^{n})^{\mathfrak{S}})$
and formula 4.17 of \cite{Shin} tells us that \[
\mbox{tr}\Pi_{M}^{\mathfrak{S}}(\tilde{\zeta}_{n_{1},n_{2}}^{*}(f^{n})^{\mathfrak{S}})=\mbox{tr}(\mbox{\ensuremath{\tilde{\zeta}}}_{n_{1},n_{2}*}(\Pi_{M}^{\mathfrak{S}}))(f^{n})^{\mathfrak{S}},\]
where $\tilde{\zeta}_{n_{1},n_{2}*}$ is the transfer from $\mathbb{G}_{n_{1},n_{1}}$
to $\mathbb{G}_{n}$ on the representation-theoretic side and consists
of taking the parabolic induction of a twist of $\Pi_{M}^{\mathfrak{S}}$.
The multiplicity one result of Jacquet and Shalika (see page 200 of
\cite{AC}) implies that the string of Satake parameters outside a
finite set $\mathfrak{S}$ of a cuspidal automorphic representation
of $GL_{n}(\mathbb{A}_{F})$ unramified outside $\mathfrak{S}$ cannot
coincide with the string of Satake parameters outside $\mathfrak{S}$
of an automorphic representation of $GL_{n}(\mathbb{A}_{F})$ which
is a subquotient of a representation induced from a proper Levi subgroup.
Thus, if we are interested in the $\Pi^{1,\mathfrak{S}}$-part of
$\mbox{tr}(\phi^{\infty,p}\phi'_{p}|\iota_{l}H_{c}(\mbox{Ig}^{(h_{1},h_{2})},\mathcal{L}_{\xi}))$,
then only the first term of (\ref{eq:sum}) can contribute to it. 

Thus, we are left to consider \[
C_{G}(-1)^{h_{1}+h_{2}}\left(\frac{1}{2}\frac{A'_{\Pi^{1}}}{A_{\Pi^{1}}^{0}}\chi_{\Pi^{1,\mathfrak{S}}}((f^{n})^{\mathfrak{S}})\mbox{tr}(\Pi_{\mathfrak{S}}^{1}(f_{\mathfrak{S}}^{n})A_{\Pi_{\mathfrak{S}}^{1}}^{0})\right.\]
\[
\left.+\sum_{(\Pi')^{\mathfrak{S}}\not=\Pi^{1,\mathfrak{S}}}\chi_{(\Pi')^{\mathfrak{S}}}((f^{n})^{\mathfrak{S}})\times\left(\substack{\mathrm{expression\ in}\\
\mathrm{terms\ of\ }f_{\mathfrak{S}}^{n}}
\right)\right)\]
where $(\Pi')^{\mathfrak{S}}$ runs over a set of unramified representations
of $\mathbb{G}_{n}(\mathbb{A}^{\mathfrak{S}})$. On the other hand,
we can also decompose $\mbox{tr}(\phi^{\infty,p}\phi'_{p}|\iota_{l}H_{c}(\mbox{Ig}^{(h_{1},h_{2})},\mathcal{L}_{\xi}))$
into a $\Pi^{1,\mathfrak{S}}$-part and $(\pi')^{\mathfrak{S}}$-parts,
where $BC((\pi')^{\mathfrak{S}})\not=\Pi^{1,\mathfrak{S}}$. We conclude
as in \cite{Shin} that \begin{equation}
\mbox{tr}(\phi_{\mathfrak{S}_{\mbox{fin}}\setminus\{p\}}\phi'_{p}|\iota_{l}H_{c}(\mbox{Ig}^{(h_{1},h_{2})},\mathcal{L}_{\xi})\{\Pi^{1,\mathfrak{S}}\})=(-1)^{h_{1}+h_{2}}\frac{C_{G}}{2}\frac{A'_{\Pi^{1}}}{A_{\Pi^{1}}^{0}}\cdot\mbox{tr}(\Pi_{\mathfrak{S}}(f_{\mathfrak{S}}^{n})A_{\Pi_{\mathfrak{S}}^{0}}).\label{eq:at S}\end{equation}
Now $\Pi_{\infty}^{1}\simeq\Pi_{\Xi}$, so $\mbox{tr}(\Pi_{\infty}^{1}(f_{\infty}^{n})A_{\Pi_{\infty}}^{0})=2(-1)^{q(G)}=2$.
We also have \begin{equation}
\mbox{tr}(\Pi_{p}^{1}(f_{p}^{n})A_{\Pi_{p}}^{0})=\mbox{\mbox{tr}}\iota_{l}\pi_{p}(\phi_{\mbox{Ig},p}^{n})=\mbox{tr}\iota_{l}\mbox{Red}_{n}^{(h_{1},h_{2})}(\pi_{p})(\phi'_{p})\label{eq:at p}\end{equation}
 by Lemma \ref{trace of red for n} and \begin{equation}
\mbox{tr}(\Pi_{\mathfrak{S}_{\mathrm{fin}}\setminus\{p\}}^{1}(f_{\mathfrak{S}_{\mathrm{fin}}\setminus\{p\}}^{n})A_{\Pi_{\mathfrak{S}_{\mathrm{fin}}\setminus\{p\}}^{1}}^{0})=\mbox{tr}\iota_{l}\pi_{p}(\phi_{\mathfrak{S}_{\mathrm{fin}}\setminus\{p\}}).\label{eq:at S-p}\end{equation}
Putting together (\ref{eq:at S}), (\ref{eq:at p}) and (\ref{eq:at S-p})
and applying $BC_{\mathfrak{S}_{\mathrm{fin}}\setminus\{p\}}$ we
get the desired result with $e_{0}=A'_{\Pi^{1}}/A_{\Pi^{1}}^{0}$
which is independent of $(h_{1},h_{2})$. 
\end{proof}

\section{Proof of the main theorem}

Let $E/\mathbb{Q}$ be an imaginary quadratic field in which $p$
splits. Let $F_{1}/\mathbb{Q}$ be a totally real field and let $w$
be a prime of $F_{1}$ above $p$. Set $F'=EF_{1}$. Let $F_{2}$
be a totally real quadratic extension of $\mathbb{Q}$, in which $w=w_{1}w_{2}$
splits and set $F=EF_{2}$. Let $n\in\mathbb{Z}_{\geq2}$ . Also denote
$F_{2}$ by $F^{+}$. $ $Let $\Pi$ be a cuspidal automorphic representation
of $GL_{n}(\mathbb{A}_{F'})$. 

Consider the following assumptions on $(E,F',F,\Pi)$:
\begin{itemize}
\item $[F_{1}:\mathbb{Q}]\geq2$;
\item $\mbox{Ram}_{F/\mathbb{Q}}\cup\mbox{Ram}_{\mathbb{Q}}(\varpi)\cup\mbox{ Ram}_{\mathbb{Q}}(\Pi)\subset\mbox{Spl}_{F/F^{+},\mathbb{Q}}$;
\item $(\Pi)^{\vee}\simeq\Pi\circ c$;
\item $\Pi_{\infty}$ is cohomological for an irreducible algebraic representation
$\Xi$ of $GL_{n}(F'\otimes_{\mathbb{Q}}\mathbb{C})$. 
\end{itemize}
Set $\Pi^{0}=BC_{F/F'}(\Pi)$ and $\Xi^{0}=BC_{F/F'}(\Xi)$. The following
lemma is the same as Lemma 7.2 of \cite{Shin}.
\begin{lem}
\label{conditions on psi}Let $\Pi^{0}$ and $\Xi^{0}$ as above.
We can find a character $\psi:\mathbb{A}_{E}^{\times}/E^{\times}\to\mathbb{C}^{\times}$
and an algebraic representation $\xi_{\mathbb{C}}$ of $G$ over $\mathbb{C}$
satisfying the following conditions
\begin{itemize}
\item $\psi_{\Pi^{0}}=\psi^{c}/\psi;$
\item $\Xi^{0}$ is isomorphic to the restriction of $\Xi^{'}$ to $R_{F/\mathbb{Q}}(GL_{n})\times_{\mathbb{Q}}\mathbb{C}$,
where $\Xi'$ is obtained from $\xi_{\mathbb{C}}$ by base change
from $G$ to $\mathbb{G}_{n};$
\item $\xi_{\mathbb{C}}|_{E_{\infty}^{\times}}^{-1}=\psi_{\infty}^{x},$
and 
\item $\mbox{Ram}_{\mathbb{Q}}(\psi)\subset\mbox{Spl}_{F/F^{+},\mathbb{Q}}$. 
\end{itemize}
Moreover, if $l$ splits in $E$ then 
\begin{itemize}
\item $\psi_{\mathcal{O}_{E_{u}}^{\times}}=1$ where $u$ is the place above
$l$ induced by $\iota_{l}^{-1}\tau|_{E}$. 
\end{itemize}
\end{lem}
Set $\Pi^{1}=\psi\otimes\Pi^{0}$. Then $\Pi^{1}$ is a cuspidal automorphic
representation of $GL_{1}(\mathbb{A}_{E})\times GL_{n}(\mathbb{A}_{F})$.
Let $\xi=\iota_{l}\xi_{\mathbb{C}}$, where $\xi_{\mathbb{C}}$ is
as in Lemma \ref{conditions on psi}. 

Let $\mathcal{A}_{U}$ be the universal abelian variety over $X_{U}$.
Since $\mathcal{A}_{U}$ is smooth over $X_{U}$, $\mathcal{A}_{U}^{m_{\xi}}$
satisfies the conditions in Section 4.3. In particular, $\mathcal{A}_{U}^{m_{\xi}}$
is locally etale over a product of semistable schemes. For $S,T\subseteq\{1,\dots,n\}$,
let $\mathcal{A}_{U,S,T}^{m_{\xi}}=\mathcal{A}_{U}^{m_{\xi}}\times_{X_{U}}Y_{U,S,T}$. 

Define the following admissible $G(\mathbb{A}^{\infty,p})$-modules
with a commuting continuous action of $\mbox{Gal}(\bar{F}'/F')$:\[
H^{j}(X_{\mathrm{Iw}(m)},\mathcal{L}_{\xi})=\lim_{\stackrel{\longrightarrow}{U^{p}}}H^{j}(X_{U}\times_{F'}\bar{F}',\mathcal{L}_{\xi})=H^{j}(X,\mathcal{L}_{\xi})^{\mathrm{Iw}(m)},\]
\[
H^{j}(\mathcal{A}_{\mathrm{Iw}(m)}^{m_{\xi}},\bar{\mathbb{Q}}_{l})=\lim_{\stackrel{\longrightarrow}{U^{p}}}H^{j}(\mathcal{A}_{U}^{m_{\xi}}\times_{F'}\bar{F}',\bar{\mathbb{Q}}_{l}).\]
Also define the admissible $G(\mathbb{A}^{\infty,p})\times(\mbox{Frob}_{\mathbb{F}})^{\mathbb{Z}}$-module\[
H^{j}(\mathcal{A}_{\mathrm{Iw}(m),S,T}^{m_{\xi}},\bar{\mathbb{Q}}_{l})=\lim_{\stackrel{\longrightarrow}{U^{p}}}H^{j}(\mathcal{A}_{U,S,T}^{m_{\xi}}\times_{\mathbb{F}}\bar{\mathbb{F}},\bar{\mathbb{Q}}_{l}).\]
Note that $a_{\xi}$ is an idempotent on $H^{j}(\mathcal{A}_{\mathrm{Iw}(m),S,T}^{m_{\xi}},\bar{\mathbb{Q}}_{l}(t_{\xi}))$
and \[
a_{\xi}H^{j+m_{\xi}}(\mathcal{A}_{\mathrm{Iw}(m),S,T}^{m_{\xi}},\bar{\mathbb{Q}}_{l}(t_{\xi}))=H^{j}(Y_{\mathrm{Iw}(m),S,T},\mathcal{L}_{\xi}).\]

\begin{prop}
For each rational prime $l\not=p$ there is a $G(\mathbb{A}^{\infty,p})\times(\mbox{Frob}_{\mathbb{F}})^{\mathbb{Z}}$-equivariant
spectral sequence with a nilpotent operator $N$ \[
BC^{p}(E_{1}^{i,m+m_{\xi}-i}(\mbox{Iw}(m),\xi)[\Pi^{1,\mathfrak{S}}])\Rightarrow\]
\[
BC^{p}(WD(H^{m}(X_{\mathrm{Iw}(m)},\mathcal{L}_{\xi})|_{\mbox{Gal}(\bar{K}/K)}[\Pi^{1,\mathfrak{S}}])^{F-ss}),\]
where \[
BC^{p}(E_{1}^{i,m+m_{\xi}-i}(\mbox{Iw}(m),\xi)[\Pi^{1,\mathfrak{S}}])=\]
\[
\bigoplus_{k-l=-i}BC^{p}(a_{\xi}H^{m+m_{\xi}}(\mathcal{A}_{\mathrm{Iw}(m)}^{m_{\xi}},Gr^{l}Gr_{k}R\psi\bar{\mathbb{Q}}_{l}(t_{\xi}))[\Pi^{1,\mathfrak{S}}]).\]
The action of $N$ sends $ $$BC^{p}(a_{\xi}H^{m+m_{\xi}}(\mathcal{A}_{\mathrm{Iw}(m)}^{m_{\xi}},Gr^{l}Gr_{k}R\psi\bar{\mathbb{Q}}_{l}(t_{\xi}))[\Pi^{1,\mathfrak{S}}])$
to \textup{\[
BC^{p}(a_{\xi}H^{m+m_{\xi}}(\mathcal{A}_{\mathrm{Iw}(m)}^{m_{\xi}},Gr^{l+1}Gr_{k-1}R\psi\bar{\mathbb{Q}}_{l}(t_{\xi}))[\Pi^{1,\mathfrak{S}}]).\]
}

Furthermore, there is a second spectral sequence \[
BC^{p}(E_{1}^{j,m+m_{\xi}-j}(k,l))\Rightarrow BC^{p}(a_{\xi}H^{m+m_{\xi}}(\mathcal{A}_{\mathrm{Iw}(m)}^{m_{\xi}},Gr^{l}Gr_{k}R\psi\bar{\mathbb{Q}}_{l}(t_{\xi}))[\Pi^{1,\mathfrak{S}}]),\]
where \[
BC^{p}(E_{1}^{j,m+m_{\xi}-j}(k,l))=\bigoplus_{s=1}^{k+l}\bigoplus_{\#S=j+s,\#T=j+k+l-s+1}H_{S,T}^{j+m_{\xi},s}(k,l)\]
 and \[
H_{S,T}^{j+m_{\xi},s}(k,l)=BC^{p}(a_{\xi}H^{m+m_{\xi}-2j-k-l+1}(\mathcal{A}_{\mathrm{Iw}(m),S,T}^{m_{\xi}},\bar{\mathbb{Q}}_{l}(t_{\xi}-j-k+1))[\Pi^{1,\mathfrak{S}}])\]
\[
=BC^{p}(H^{m-2j-k-l+1}(Y_{\mathrm{Iw}(m),S,T},\bar{\mathbb{Q}}_{l}(-j-k+1))[\Pi^{1,\mathfrak{S}}]).\]
\end{prop}
\begin{proof}
Note that $\mathcal{A}_{U}^{m_{\xi}}/\mathcal{O}_{K}$ satisfies the
hypotheses of Section 4. We have a spectral sequence of $G(\mathbb{A}^{\infty,p})\times(\mbox{Frob}_{\mathbb{F}})^{\mathbb{Z}}$-modules
with a nilpotent operator $N$: \[
E_{1}^{i,m-i}(\mbox{Iw}(m),\xi)\Rightarrow H^{m}(\mathcal{A}_{U}^{m_{\xi}}\times_{F'}\bar{F}'_{\mathfrak{p}},\bar{\mathbb{Q}}_{l}(t)),\]
where \[
E_{1}^{i,m-i}(\mbox{Iw}(m),\xi)=\bigoplus_{k-l=-i}H^{m}(\mathcal{A}_{U}^{m_{\xi}}\times_{\mathbb{F}}\bar{\mathbb{F}},Gr^{l}Gr_{k}R\psi\bar{\mathbb{Q}}_{l}(t)).\]
$N$ will send $H^{m}(\mathcal{A}_{U}^{m_{\xi}},\mbox{Gr}^{l}\mbox{Gr}_{k}R\psi\bar{\mathbb{Q}}_{l}(t))$
to $H^{m}(\mathcal{A}_{U}^{m_{\xi}},\mbox{Gr}^{l+1}\mbox{Gr}_{k-1}R\psi\bar{\mathbb{Q}}_{l}(t))$.

By Corollary \ref{Spectral sequence}, we also have a $G(\mathbb{A}^{\infty,p})\times(\mbox{Frob}_{\mathbb{F}})^{\mathbb{Z}}$-equivariant
spectral sequence \[
E_{1}^{j,m-j}(k,l)\Rightarrow H^{m}(\mathcal{A}_{U}^{m_{\xi}}\times_{\mathbb{F}}\bar{\mathbb{F}},Gr^{l}Gr_{k}R\psi\bar{\mathbb{Q}}_{l}(t)),\]
where \[
E_{1}^{j,m-j}(k,l)=\bigoplus_{s=1}^{k+l}\bigoplus_{\substack{\#S=j+s\\
\#T=j+k+l-s+1}
}H_{S,T}^{j,s}(k,l)\]
 and \[
H_{S,T}^{j,s}(k,l)=H^{m-2j-k-l+1}(\mathcal{A}_{U,S,T}^{m_{\xi}}\times_{\mathbb{F}}\bar{\mathbb{F}},\bar{\mathbb{Q}}_{l}(t-j-k+1)).\]
We take $t=t_{\xi}$, apply $a_{\xi}$, replace $j$ by $ $$j+m_{\xi}$
and take the inverse limit over $U^{p}$. We get two spectral sequences
of $G(\mathbb{A}^{\infty,p})\times(\mbox{Frob}_{\mathbb{F}})^{\mathbb{Z}}$-modules,
converging to $H^{j}(X_{\mathrm{Iw}(m)},\mathcal{L}_{\xi})$. We identify
$H^{j}(X_{\mathrm{Iw}(m)},\mathcal{L}_{\xi})$ with its associated
Weil-Deligne representation and we semisimplify the action of Frobenius.
After taking $\Pi^{1,\mathfrak{S}}$-isotypical components and applying
$BC^{p}$ we get the desired spectral sequences. \end{proof}
\begin{cor}
\label{purity of cohomology}Keep the assumptions made in the beginning
of this section. The Weil-Deligne representation \[
WD(BC^{p}(H^{2n-2}(X_{\mathrm{Iw}(m)},\mathcal{L}_{\xi})|_{Gal(\bar{K}/K)}[\Pi^{1,\mathfrak{S}}]))^{F-ss}\]
 is pure of weight $m_{\xi}-2t_{\xi}+2n-2$. \end{cor}
\begin{proof}
By Proposition \ref{vanishing}, \[
BC^{p}(H^{j}(Y_{\mathrm{Iw}(m),S,T},\mathcal{L}_{\xi})[\Pi^{1,\mathfrak{S}}])=0\]
 unless $j=2n-\#S-\#T$. Thus, the terms of the spectral sequence
\[
BC^{p}(E_{1}^{j,m+m_{\xi}-j}(k,l)),\]
which are all of the form \[
BC^{p}(H^{m-2j-k-l+1}(Y_{\mathrm{Iw}(m),S,T},\bar{\mathbb{Q}}_{l}(-j-k+1))[\Pi^{1,\mathfrak{S}}])\]
with $\#S=j+s$ and $\#T=j+k+l-s+1$, vanish unless $m=2n-2$. The
spectral sequence $BC^{p}(E_{1}^{j,m+m_{\xi}-j}(k,l))$ degenerates
at $E_{1}$ and moreover, the terms of the spectral sequence $BC^{p}(E_{1}^{i,m+m_{\xi}-i}(\mbox{Iw}(m),\xi)[\Pi^{1,\mathfrak{S}}])$
vanish unless $m=2n-2$. If $m=2n-2$ then each summand of \[
BC^{p}(E_{1}^{i,2n-2+m_{\xi}-i}(\mbox{Iw}(m),\xi)[\Pi^{1,\mathfrak{S}}])\]
has a filtration with graded pieces \[
BC^{p}(H^{2n-2-2j-k-l+1}(Y_{\mathrm{Iw}(m),S,T},\mathcal{L}_{\xi}(-j-k+1))[\Pi^{1,\mathfrak{S}}]),\]
where $k-l=-i$. These graded pieces are strictly pure of weight $m_{\xi}-2t_{\xi}+2n-2+k-l-1$,
which only depends on $i$. Thus, the whole of \[
BC^{p}(E_{1}^{i,2n-2+m_{\xi}-i}(\mbox{Iw}(m),\xi)[\Pi^{1,\mathfrak{S}}])\]
 is strictly pure of weight $m_{\xi}-2t_{\xi}+2n-2-i-1$. This second
spectral sequence also degenerates at $E_{1}$ (since $E_{1}^{i,m-i}=0$
unless $m=2n-2$) and the abutment is pure of weight $m_{\xi}-2t_{\xi}+2n-2$.
Thus, \[
BC^{p}(WD(H^{m}(X_{\mathrm{Iw}(m)},\mathcal{L}_{\xi})|_{\mbox{Gal}(\bar{K}/K)}[\Pi^{1,\mathfrak{S}}])^{F-ss})\]
vanishes for $m\not=2n-2$ and is pure of weight $m_{\xi}-2t_{\xi}+2n-2$
for $m=2n-2$. \end{proof}
\begin{thm}
Let $n\in\mathbb{Z}_{\geq2}$ be an integer and $L$ be any CM field.
Let $l$ be a prime and $\iota_{l}$ be an isomorphism $\iota_{l}:\bar{\mathbb{Q}}_{l}\to\mathbb{C}$.
Let $\Pi$ be a cuspidal automorphic representation of $GL_{n}(\mathbb{A}_{L})$
satisfying
\begin{itemize}
\item $\Pi^{\vee}\simeq\Pi\circ c$
\item $\Pi$ is cohomological for some irreducible algebraic representation
$\Xi$. 
\end{itemize}
Let \[
R_{l}(\Pi):\mbox{Gal}(\bar{L}/L)\to GL_{n}(\bar{\mathbb{Q}}_{l})\]
 be the Galois representation associated to $\Pi$ by \cite{Shin,CH}.
Let $p\not=l$ and let $y$ be a place of $L$ above $p$. Then we
have the following isomorphism of Weil-Deligne respresentations \[
WD(R_{l}(\Pi)|_{Gal(\bar{L}_{y}/L_{y})})^{F-ss}\simeq\iota_{l}^{-1}\mathcal{L}_{n,L_{y}}(\Pi_{y}).\]
\end{thm}
\begin{proof}
This theorem has been proven by \cite{Shin} except in the case when
$n$ is even and $\Xi$ is not slightly regular. In that exceptional
case it is still known that we have an isomorphism of semisimplified
$W_{L_{y}}$-representations by \cite{CH}, so it remains to check
that the two monodromy operators $N$ match up. By Corollary \ref{temperedness},
$\Pi_{y}$ is tempered. This is equivalent to $\iota_{l}^{-1}\mathcal{L}_{n,L_{y}}(\Pi_{y})$
being pure of weight $2n-2$. In order to get an isomorphism of Weil-Deligne
representations, it suffices to prove that $WD(R_{l}(\Pi)|_{Gal(\bar{L}_{y}/L_{y})})$
is pure. 

We first will find a CM field $F'$ such that
\begin{itemize}
\item $F'=EF_{1},$ where $E$ is an imaginary quadratic field in which
$p$ splits and $F_{1}=(F')^{c=1}$ has $[F_{1}:\mathbb{Q}]\geq2$, 
\item $F'$ is soluble and Galois over $L$,
\item $\Pi_{F'}^{0}=BC_{F'/L}(\Pi)$ is a cuspidal automorphic representation
of $GL_{n}(\mathbb{A}_{F'})$, and 
\item there is a place $\mathfrak{p}$ of $F$ above $y$ such that $\Pi_{F',\mathfrak{p}}^{0}$
has an Iwahori fixed vector, 
\end{itemize}
and a CM field $F$ which is a quadratic extension of $F'$ such that 
\begin{itemize}
\item $\mathfrak{p}=\mathfrak{p}_{1}\mathfrak{p}_{2}$ splits in $F$,
\item $\mbox{Ram}_{F/\mathbb{Q}}\cup\mbox{Ram}_{\mathbb{Q}}(\varpi)\cup\mbox{ Ram}_{\mathbb{Q}}(\Pi)\subset\mbox{Spl}_{F/F',\mathbb{Q}},$
and
\item $\Pi_{F}^{0}=BC_{F/F'}(\Pi_{F'}^{0})$ is a cuspidal automorphic representation
of $GL_{n}(\mathbb{A}_{F}).$
\end{itemize}
To find $F$ and $F'$ we proceed as in the proof of Corollary \ref{temperedness}.
Set $\Pi_{F}^{1}=\Pi_{F}^{0}\otimes\psi$, where $\psi$ is chosen
as in Lemma \ref{conditions on psi}. 

We claim that we have isomorphisms\[
C_{G}\cdot(R_{l}(\Pi)|_{Gal(\bar{F'}/F')})^{\otimes2}\simeq C_{G}\cdot R_{l}(\Pi_{F'}^{0})^{\otimes2}\simeq\tilde{R}_{l}^{2n-2}(\Pi_{F}^{1})\otimes R_{l}(\psi)^{-1},\]
where $\tilde{R}_{l}^{k}(\Pi_{F}^{1})$ was defined in Section 4.
The first isomorphism is clear. The second isomorphism can be checked
by Chebotarev locally at unramified places, using the local global
compatibility for $R_{l}(\Pi_{F'}^{0})$ and the formula\[
\tilde{R}_{l}(\Pi_{F}^{1})=e_{0}C_{G}\cdot[(\pi_{p,0}\circ\mbox{Art}_{\mathbb{Q}_{p}}^{-1})|_{W_{F'_{\mathfrak{p}}}}\otimes\iota_{l}^{-1}\mathcal{L}_{F'_{\mathfrak{p}},n}(\Pi_{F',\mathfrak{p}}^{0})^{\otimes2}].\]
(It can be checked easily, either by computing the weight or by using
the spectral sequences above that $\tilde{R}_{l}^{k}(\Pi_{F}^{1})\not=0$
if and only if $k=2n-2$ and thus that $e_{0}=(-1)^{2n-2}=1.)$ 

We also have \[
BC^{p}(H^{2n-2}(X_{\mbox{Iw}(m)},\mathcal{L}_{\xi})[\Pi_{F}^{1,\mathfrak{S}}])\simeq(\dim\pi_{p}^{\mbox{Iw}(m)})\cdot\iota_{l}^{-1}\Pi^{\infty,p}\otimes\tilde{R}^{2n-2}(\Pi_{F}^{1})\]
as admissible representations of $G(\mathbb{A}^{\infty,p})\times Gal(\bar{F}'/F')$.
By Corollary \ref{purity of cohomology}, $WD(\tilde{R}_{l}^{2n-2}(\Pi_{F}^{1})|_{Gal(\bar{F}_{\mathfrak{p}}'/F'_{\mathfrak{p}})})$
is pure of weight $2n-2$. By Lemma 1.7 of \cite{T-Y}, \[
WD(R_{l}(\Pi_{F'}^{0})^{\otimes2}|_{Gal(\bar{F}'_{\mathfrak{p}}/F'_{\mathfrak{p}})})\]
is also pure. The monodromy operator acts on the $R_{l}(\Pi_{F'}^{0})^{\otimes2}|_{W_{F'_{\mathfrak{p}}}}$
as $1\otimes N+N\otimes1$, where $N$ is the monodromy operator on
$R_{l}(\Pi_{F'}^{0})|_{W_{F'_{\mathfrak{p}}}}$. If the latter was
pure as a Weil-Deligne representation, the former would have to be
pure as well. Moreover, for each $W_{F'_{\mathfrak{p}}}$-representation
there is at most one choice for the monodromy operator which makes
it a pure Weil-Deligne representation. Thus, $WD(R_{l}(\Pi_{F'}^{0})|_{Gal(\bar{F}'_{\mathfrak{p}}/F'_{\mathfrak{p}})})$
has to be pure. By Lemma 1.4 of \cite{T-Y}, purity is preserved under
finite extensions, so $WD(R_{l}(\Pi)|_{Gal(\bar{L}_{y}/L_{y})})$
is also pure. \end{proof}

\end{document}